\documentclass[a4paper]{article}

\usepackage[french,english]{babel}
\usepackage[utf8]{inputenc}
\usepackage[T1]{fontenc}
\usepackage{lmodern}
\usepackage{amsmath}
\usepackage{amssymb}
\usepackage{mathrsfs}
\usepackage{amsthm}
\usepackage[shortlabels]{enumitem}
\usepackage{dsfont} 
\usepackage{eurosym} 
\usepackage{stmaryrd} 
\usepackage{eurosym}

\usepackage{titlesec}
\titleclass{\part}{top}
\titleformat{\part}[display]
{\huge\bfseries\centering}{\partname~\thepart}{0pt}{}
\titlespacing*{\part}{0pt}{160pt}{40pt}

\usepackage{tikz}
\usetikzlibrary{arrows, decorations.pathreplacing, decorations.pathmorphing}

\usepackage{graphicx}
\usepackage{float}

\usepackage{hyperref}  
\hypersetup{                    
	colorlinks=true,                
	breaklinks=true,                
	urlcolor= black,                 
	linkcolor= black,                
	citecolor= magenta,                
	pdfstartview = FitH,
	bookmarksopen = true               
}

\usepackage{listings}
\usepackage{xcolor}
\colorlet{red}{black}
\colorlet{blue}{black}
\colorlet{purple}{black}

\usepackage{framed}

\usepackage{fullpage} 

\newcommand{\R}{\mathbb{R}}
\newcommand{\N}{\mathbb{N}}
\newcommand{\E}{\mathbb{E}}

\renewcommand{\epsilon}{\varepsilon}

\newcommand{\1}{\mathds{1}}

\newcommand{\mathbbm}[1]{\1}

\title{Approximation of martingale couplings on the line in the adapted weak topology}
\makeatletter
\newcommand{\specificthanks}[1]{\@fnsymbol{#1}}
\makeatother
\author{
	M. Beiglböck\footnote{University of Vienna, Austria. Email: \href{mailto:mathias.beiglboeck@univie.ac.at}{\texttt{mathias.beiglboeck@univie.ac.at}}}\thanks{acknowledges support from the Austrian Science Fund (FWF) through grant number Y00782.} 
	\and 
	B. Jourdain\footnote{CERMICS, Ecole des Ponts, INRIA, Marne-la-Vallée, France. 
	E-mails: \href{mailto:benjamin.jourdain@enpc.fr}{\texttt{benjamin.jourdain@enpc.fr}}, \href{mailto:william.margheriti@enpc.fr}{\texttt{william.margheriti@enpc.fr}}} 
	\and
	W. Margheriti\textsuperscript{\specificthanks{2},~}\thanks{acknowledges support from the \textquotedblleft Chaire Risques Financiers\textquotedblright , Fondation du Risque.}
	\and
	G. Pammer\footnote{ETH Zürich, Switzerland. Email: \href{mailto:gudmund.pammer@math.ethz.ch}{\texttt{gudmund.pammer@math.ethz.ch}}}\thanks{acknowledges support from the Austrian Science Fund (FWF) through grant number W1245.}
}
\date{\today}

\numberwithin{equation}{section}
\theoremstyle{plain}
\newtheorem{prooff}{Proof}[section]

\addto\captionsfrench{}

\newtheorem{lemma}[prooff]{Lemma}

\newtheorem{theorem}[prooff]{Theorem}
\newtheorem{proposition}[prooff]{Proposition}

\theoremstyle{definition}
\newtheorem{example}[prooff]{Example}

\newtheorem{remark}[prooff]{Remark}

\theoremstyle{definition}

\theoremstyle{plain}
\newtheorem{demo}[prooff]{Démonstration}

\newtheorem{fact}[prooff]{Fact}

\allowdisplaybreaks

\newcommand\x{0.5}

\date{}

\begin{document}

\maketitle

\begin{abstract} 

 Our main result is to establish stability of martingale couplings: suppose that $\pi$ is a martingale coupling with marginals $\mu, \nu$.
 Then, given approximating marginal measures $\tilde \mu \approx \mu, \tilde \nu\approx \nu$  in convex order,  we show that there exists an approximating martingale coupling $\tilde\pi \approx \pi$ with   marginals $\tilde \mu, \tilde \nu$.
 
In mathematical finance, prices of European call / put option yield information on the marginal measures of the arbitrage free pricing measures. The above result asserts that small variations of call / put prices lead only to  small variations on the level of arbitrage free pricing measures.

While these facts have been anticipated for some time, the actual proof requires somewhat intricate stability results for the adapted Wasserstein distance.  
Notably the result has consequences for several related problems. Specifically, it is relevant for numerical approximations, it leads to a new proof of the monotonicity principle of martingale optimal transport and it implies stability of weak martingale optimal transport as well as optimal Skorokhod embedding. On the mathematical finance side this yields continuity of the robust pricing problem for exotic options and VIX options with respect to market data. These applications will be detailed in two companion papers. 
\end{abstract}


\section{Introduction}


Before carefully explaining all required notation and describing relevant literature, let us give a first description of our main result and  its relevance for the martingale transport theory. 

While classical transport theory is concerned with the set $\Pi(\mu, \nu)$ of \emph{couplings} or \emph{transport plans}  of probability measures $\mu,\nu$, the martingale variant restricts the problem to the set  $\Pi_M(\mu, \nu)$ of \emph{martingale couplings}, that is, transport plans which preserve the barycenter of each particle. Even though the main interest lies in the case where $\mu, \nu$ are probabilities on the real line, many of the basic arguments and results appear significantly more involved in the martingale context. A basic explanation lies in the rigidity of the martingale condition that makes classically  simple  approximation results quite intricate. Specifically, the martingale theory has been missing a counterpart to the following straightforward fact of the classical transport theory:
\begin{fact}[Stability of couplings]\label{RoleModel}
Let $\pi\in \Pi(\mu, \nu)$ and assume that $\mu^k,\nu^k$, $k\in\N$, are probabilities that converge weakly to $\mu$ and $\nu$. Then there exist couplings $\pi^k\in\Pi(\mu^k,\nu^k), k\in \N$ converging weakly to $\pi$.
\end{fact}
This result is so  basic and  straightforward that its implicit use is easily overlooked. Note however that it plays a crucial role in a number of occasions, e.g.\ for stability of optimal transport,  providing numerical approximations, or in the characterisation of optimality through cyclical monotonicity. 

The main result of this article is to establish Fact \ref{RoleModel} for martingale transports on the real line, see Theorem \ref{thm:1} below. This closes a gap in the theory of martingale transport and yields basic fundamental results in a unified fashion that is much closer to the classical theory.  It allows to address questions in martingale optimal transport, optimal Skorokhod embedding and robust finance that have previously remained open. These applications are considered systematically in two accompanying articles, \textcolor{blue}{see \cite{BeJoMaPa21b} for the first of the two.}
\textcolor{purple}{Among other results, we establish therein the stability of the superreplication bound for VIX futures as well as the stability of the stretched Brownian motion.
Moreover, we derive sufficiency of a monotonicity principle, in the spirit of cyclical monotonicity of classical optimal transport, for the weak martingale optimal transport problem and are able to generalize the results concerning the corresponding notion of monotonicity in martingale optimal transport. }

We note that 
while virtually all (to the best of our knowledge) applications of martingale optimal transport are concerned with the case where $\mu, \nu$ are supported on $\R$, it is a highly intriguing challenge to extend the martingale transport theory to the case where $\mu, \nu$ are supported on $\R^d, d> 1$. 
In a remarkable contrast to our main result, stability of martingale optimal transport breaks down in higher dimensions as has been recently established by Br\"uckerhoff and Juillet \cite{BrJu21}.

\subsection{The Martingale Optimal Transport problem} Let $(X,d_X)$, $(Y,d_Y)$ be Polish spaces and $C:X\times Y\to\R_+$ be a nonnegative measurable function. Denote by $\mathcal P(X)$ the set of probability measures on $X$. For $\mu\in\mathcal P(X)$ and $\nu\in\mathcal P(Y)$, the classical Optimal Transport problem consists in minimising
\begin{equation}\label{OT2}
\tag{OT}
\inf_{\pi\in\Pi(\mu,\nu)}\int_{X\times Y}C(x,y)\,\pi(dx,dy),
\end{equation}
where $\Pi(\mu,\nu)$ denotes the set of probability measures in $\mathcal P(X\times Y)$ with the first marginal $\mu$ and the second marginal $\nu$. When $X=Y$ and $C=d_X^r$ for some $r\ge1$, \eqref{OT2} corresponds to the well-known Wasserstein distance with index $r$ to the power $r$, denoted $\mathcal W_r^r(\mu,\nu)$, see \cite{AmGi13,Sa15,Vi1,Vi09} for a study in depth.

The theory of OT goes back \textcolor{red}{to Monge \cite{Mo81} in its original formulation and Kantorovich \cite{Ka42} in its modern formulation}. It was rediscovered many times under various forms and has an impressive scope of applications. A variant of OT that is motivated by applications in mathematical finance, in particular in model-independent pricing, was introduced in \cite{BeHePe12} in a discrete time setting and in \cite{GaHeTo13} in a continuous time setting. Compared to the usual OT,  the difference is that one requires an additional martingale constraint to \eqref{OT2} which reflects the condition for a financial market to be free of arbitrage.

In detail, the Martingale Optimal Transport (MOT) problem is formulated as follows: given $\pi\in\mathcal P(\R\times \R)$, we denote by $(\pi_x)_{x\in X}$ \textcolor{blue}{a regular conditional disintegration} with respect to its first marginal $\mu$. We then write $\pi(dx,dy)=\mu(dx)\,\pi_x(dy)$, or with a slight abuse of notation, $\pi=\mu\times\pi_x$ if the context is not ambiguous. Let $C:\R\times\R\to\R_+$ be a nonnegative measurable function and $\mu$, $\nu$ be two probability distributions on the real line with finite first moment. Then the MOT problem consists in minimising
\begin{equation}\label{MOT2}
\tag{MOT}
\inf_{\pi\in\Pi_M(\mu,\nu)}\int_{\R\times\R}C(x,y)\,\pi(dx,dy),
\end{equation}
where $\Pi_M(\mu,\nu)$ denotes the set of martingale couplings between $\mu$ and $\nu$, that is
\[
\Pi_M(\mu,\nu)=\left\{\pi=\mu\times\pi_x\in\Pi(\mu,\nu)\mid\mu(dx)\text{-almost everywhere},\ \int_\R y\,\pi_x(dy)=x\right\}.
\]
According to Strassen's theorem \cite{St65}, the existence of a martingale coupling between two probability measures $\mu,\nu\in\mathcal P(\R)$ with finite first moment is equivalent to $\mu\le_c\nu$, where $\le_c$ denotes the convex order. We recall that two finite positive measures $\mu,\nu$ on $\R$ with finite first moment and are said to be in the convex order \textcolor{red}{if and only if}
 we have
\[
\int_{\R}f(x)\,\mu(dx)\le\int_{\R}f(y)\,\nu(dy),
\]
for every convex function \textcolor{red}{$f:\R\to(-\infty, \infty]$.}
Note that there holds equality for all affine functions, from which we deduce that $\mu$ and $\nu$ have equal masses and satisfy $\int_{\R} x\,\mu(dx)=\int_{\R} y\,\nu(dy)$.

For adaptations of classical
optimal transport theory to the MOT problem, we refer to  \cite{HoNe12, HeTato, HeTo13}. Concerning duality results, we
refer to \cite{DoSo12, BeNuTo16, De18, ChKiPrSo20}. We also refer to \cite{ObSi17, De18b, DeTo17, GhKiLi19} for the multi-dimensional case and to \cite{BeCoHu, BeNuSt19} for connections to Skorokhod embedding problem.

Concerning the numerical resolution of the MOT problem, we refer to the articles \cite{AlCoJo17a, AlCoJo17b, De18c, GuOb, He19}. When $\mu$ and $\nu$ are finitely supported, then the MOT problem amounts to linear programming. In the general case, once the MOT problem is discretised by approximating $\mu$ and $\nu$ by probability measures with finite support and
in the convex order, Alfonsi, Corbetta and Jourdain  \cite{AlCoJo17b} raised the question of the convergence of optimal costs of the discretised problem towards the
costs of the original problem. 
A first partial result was obtained by Juillet \cite{Ju14b} who established stability of left-curtain coupling.
Guo and Obłój \cite{GuOb} establish the result under moment conditions.  More recently,  \cite{BaPa19, Wi20} independently gave a definite positive answer. 

\subsection{The Adapted Wasserstein distance} The stability result shown in \cite{BaPa19} involves Wasserstein convergence. More precisely, let $\mu^k,\nu^k\in\mathcal P(\R)$, $k\in\N$ be in the convex order and respectively converge to $\mu$ and $\nu$ in $\mathcal W_r$. Under mild assumption, for all $k\in\N$ there exists $\pi^k\in\Pi_M(\mu^k,\nu^k)$, optimal for \eqref{MOT2}, and any accumulation point of $(\pi^k)_{k\in\N}$ with respect to the $\mathcal W_r$-convergence is a martingale coupling between $\mu$ and $\nu$ optimal for \eqref{MOT2}.

However, it turns out that the usual weak topology / Wasserstein distance is not well suited  in setups where accumulation of information plays a distinct role, e.g.\ in mathematical finance. Indeed, the symmetry of this distance does not take into account the temporal structure of stochastic processes. It is easy to convince oneself that two stochastic processes very close in Wasserstein distance can yield radically unalike information, as illustrated in  \cite[Figure 1]{BaBaBeEd19a}. Therefore, one needs to strengthen, the usual topology of weak convergence accordingly. Over time numerous researchers have independently introduced refinements of the weak topology, we mention Hellwig's information topology \cite{He96}, Aldous's extended weak topology \cite{Al81}, the nested distance / adapted Wasserstein distance of Plug-Pichler \cite{PfPi12} and the optimal stopping topology \cite{BaBaBeEd19b}. Strikingly, all those seemingly different definitions lead to same topology in the present discrete time \cite[Theorem 1.1]{BaBaBeEd19b} framework. We refer to this topology as the \emph{adapted weak topology}. 
A natural compatible metric is given by the \emph{adapted Wasserstein distance}, see \cite{PfPi12,PfPi14,PfPi15,PfPi16,Las18,BiTa18} among others. 

Fix $x_0\in X$ and $r\ge1$. We denote the set of all probability measures on $X$ with finite $r$-th moment by $\mathcal P_r(X)$, i.e.
\[
\mathcal P_r(X)=\left\{p\in\mathcal P(X)\mid\int_Xd_X^r(x,x_0)\,p(dx)<+\infty\right\}.
\]

Let $\mathcal M(X)$ (resp.\ $\mathcal M_r(X)$) denote the set of all finite positive measures (resp.\ with finite $r$-th moment). The sets $\mathcal M(X)$ and $\mathcal M_r(X)$, resp.\ are equipped with the weak topology induced by the set $C_b(X)$ of all real-valued absolutely bounded continuous functions on $X$ and, resp.,\ the set $\Phi_r(X)$ of all real-valued continuous functions on $X$, $C(X)$, which satisfy the growth constraint
\[
\Phi_r(X)=\left\{f\in C(X)\mid\exists\alpha>0,\ \forall x\in X,\ \vert f(x)\vert\le\alpha\left(1+d_X^r(x,x_0)\right)\right\}.
\]
 A sequence $(\mu^k)_{k\in\N}$ converges in $\mathcal M_r(X)$ to $\mu$ \textcolor{red}{if and only if}

\begin{equation}\label{eq:convergencePr}
\forall f\in\Phi_r(X),\quad \mu^k(f) \underset{k\to+\infty}{\longrightarrow} \mu(f).
\end{equation}
If moreover $\mu$ and $\mu^k$, $k\in\N$, have equal masses, then the convergence \eqref{eq:convergencePr} can be equivalently formulated \textcolor{blue}{(see for instance \cite[Theorem 6.9]{Vi09})} in terms of the Wasserstein distance with index $r$:
\[
\mathcal W_r(\mu^k,\mu) :=\inf_{\pi \in \Pi(\mu^k,\mu)} \left(\int_{X\times X} d_X^r(x,y)\,\pi(dx,dy)\right)^{\frac1r}\underset{k\to+\infty}{\longrightarrow}0.
\]
Given $m_0>0$, we can then equip the set of finite positive measures in $\mathcal M_r(X\times Y)$ with mass $m_0$ with the Wasserstein topology. However, we can also equip it with a stronger topology, namely the adapted Wasserstein topology. It is induced by the metric $\mathcal{AW}_r$ defined for all $\pi,\pi'\in\mathcal M_r(X\times Y)$ such that $\pi(X\times Y)=\pi'(X\times Y)=m_0$ by
\begin{equation}\label{eq:defAW2}
\mathcal{AW}_r(\pi,\pi') =\inf_{\chi \in \Pi(\mu,\mu')} \left(\int_{X\times X} \left(d_X^r(x,x') + \mathcal W_r^r(\pi_{x},\pi'_{x'})\right)\, \chi(dx,dx')\right)^{\frac1r},
\end{equation}
where $\mu$, resp.\ $\mu'$, is the first marginal of $\pi$, resp.\ $\pi'$. It is easy to check that $\mathcal W_r \leq \mathcal{AW}_r$, and therefore $\mathcal{AW}_r$ indeed induces a stronger topology than $\mathcal W_r$. Another useful point of view is the following: let $J:\mathcal M(X\times Y)\to\mathcal M(X\times\mathcal P(Y))$ be the inclusion map defined for all $\pi=\mu\times\pi_x\in\mathcal M(X\times Y)$ by
\[
J(\pi)(dx,dp)=\mu(dx)\,\delta_{\pi_x}(dp).
\]
For all $\pi,\pi'\in\mathcal M_r(X\times Y)$ with equal masses, their adapted Wasserstein distance coincides with
\begin{align}\label{eq:AW=W}
\mathcal{AW}_r(\pi,\pi') = \mathcal W_r(J(\pi),J(\pi')).
\end{align}
It follows that the topology induced by $\mathcal{AW}_r$ \textcolor{blue}{coincides with the weak topology induced by $J$}.

Finally, let us mention the interpretation of the adapted Wasserstein distance in terms of bicausal couplings (cf.\ \cite{BaBaBeWi}). Let $\pi,\pi'\in\mathcal P_r(X\times Y)$.  Let $Z_1,Z_2,Z'_1,Z'_2$ be random variables such that the distribution of $(Z_1,Z_2,Z'_1,Z'_2)$ is a $\mathcal W_r$-optimal coupling between $\pi$ and $\pi'$. In many cases, there exists a Monge transport map $T:X\times Y\to X\times Y$ such that $(Z'_1,Z'_2)=T(Z_1,Z_2)$. As mentioned in \cite{BaBaBeEd19a}, the temporal structure of stochastic processes is then not taken into account since the present value $Z'_1$ is determined from the future value $Z_2$. Therefore, it is more suitable to restrict to couplings $(Z_1,Z_2,Z'_1,Z'_2)$ between $\pi$ and $\pi'$ such that the conditional distribution of $Z'_1$ (resp.\ $Z_1$) given $(Z_1,Z_2)$ (resp.\ $(Z'_1,Z'_2)$) is equal to the conditional distribution of $Z'_1$ (resp.\ $Z_1$) given $Z_1$ (resp.\ $Z'_1$).

Let $\mu$ and $\mu'$ denote the respective first marginal distributions of $\pi$ and $\pi'$ and let $\eta\in\Pi(\pi,\pi')$ be a coupling between $\pi$ and $\pi'$. Let $\chi(dx,dx')=\int_{(y,y')\in Y\times Y}\eta(dx,dy,dx',dy')\in\Pi(\mu,\mu')$. W
e write $\chi(dx,dx')=\mu(dx)\,\chi_x(dx')=\mu'(dx')\,\overleftarrow{\chi}_{x'}(dx)$. Then $\eta$ is called bicausal \textcolor{red}{if and only if}

\[
\int_{y'\in Y}\eta(dx,dy,dx',dy')=\pi(dx,dy)\,\chi_x(dx')\quad\textrm{and}\quad\int_{y\in Y}\eta(dx,dy,dx',dy')=\pi'(dx',dy')\,\overleftarrow{\chi}_{x'}(dx).
\]
We denote by $\Pi_{bc}(\pi,\pi')$ the set of bicausal couplings between $\pi$ and $\pi'$. Let $(\gamma_{(x,x')}(dy,dy'))_{(x,x')\in X\times X}$ be a probability kernel such that $\eta(dx,dy,dx',dy')=\chi(dx,dx')\,\gamma_{(x,x')}(dy,dy')$. Another useful characterisation is that $\eta$ is bicausal \textcolor{red}{if and only if}
 $\chi(dx,dx')\text{-almost everywhere}$, $\gamma_{(x,x')}(dy,dy')\in\Pi(\pi_x,\pi_{x'})$. Then the adapted Wasserstein distance coincides with
\[
\mathcal{AW}_r(\pi,\pi')=\inf_{\eta\in\Pi_{bc}(\pi,\pi')}\left(\int_{X\times Y}\left(d_X^r(x,x')+d_Y^r(y,y')\right)\,\eta(dx,dy,dx',dy')\right)^{\frac1r}.
\]
One of the objectives of the present paper is to prove that well-known stability results for the $\mathcal W_r$-convergence also hold for the $\mathcal{AW}_r$-convergence. More details are given in Section \ref{sec:main results}.


\subsection{Outline} Section \ref{sec:main results} presents the main result of this article, Theorem \ref{thm:1}. We also provide a discussion of the result and give a sketch of its proof in order to help seeing through the technical details provided later on.

In
Section \ref{sec:adapted weak topology} we provide certain technical lemmas which allow us to deal with difficulties specific to the adapted Wasserstein distance with more ease. They mainly explore properties of approximations and when addition (in a sense explained below) is continuous.

Section \ref{sec: convex order} focuses on the convex order. It deals with potential functions which are a convenient tool to address the convex order in dimension one. 

Section \ref{sec:proof main thm} is  devoted to the proof of the main theorem. Before entering into actual argument, we establish that is is enough to prove $\mathcal{AW}_1$-convergence for irreducible pairs of marginals.

 \section{Main result}
    \label{sec:main results}

    Our main result is Theorem \ref{thm:1} below. Before stating it, we give a proposition which enlightens us why the conclusion of the theorem should  be the least hoped for. We also state a generalisation of this proposition to Polish spaces. Then, we state a proposition which is a key result to argue that the theorem needs only to be proved when the limit pair is irreducible. Next, we state the theorem together with a sketch of its proof. It is understood that $(X,d_X)$ and $(Y,d_Y)$ denote arbitrary Polish spaces  and that $(x_0,y_0)$ is a fixed element of $X\times Y$.

As already mentioned above,    it is well-known (and easy to show) that when one considers convergent sequences of marginals $(\mu^k)_{k\in\N}$, $(\nu^k)_{k\in\N}$  (with equal masses)  to $\mu,\nu \in \mathcal M_r(X)$, then, informally speaking, we have\footnote{\textcolor{red}{Note that this can  be made precise in terms of hemicontinuity. We also refer to \cite{neufeld2021stability}.}}
    \begin{align}\label{eq:informal convergence}
    \Pi(\mu^k,\nu^k) \underset{k\to+\infty}{\longrightarrow} \Pi(\mu,\nu)\quad \text{in }\mathcal W_r,
    \end{align}
	i.e., any sequence with convergent marginals has accumulation points in $\Pi(\mu,\nu)$, and for any $\pi \in \Pi(\mu,\nu)$ it  holds
    \begin{align}\label{eq:rate}
    \inf_{\pi^k\in\Pi(\mu^k,\nu^k)}\mathcal W_r^r(\pi,\pi^k)\leq \mathcal W_r^r(\mu,\mu^k) + \mathcal W_r^r(\nu,\nu^k)\underset{k\to+\infty}\longrightarrow0.
    \end{align}

    \textcolor{blue}{Indeed, if $\eta^k\in\Pi(\mu^k,\mu)$, resp.\ $\tau^k\in\Pi(\nu,\nu^k)$ is optimal for $\mathcal W_r(\mu^k,\mu)$, resp.\ $\mathcal W_r(\nu,\nu^k)$, then the measure $\eta^k(dx^k,dx)\,\pi_x(dy)\,\tau^k_y(dy^k)$ is a coupling between $\pi(dx,dy)$ and $\int_{(x,y)\in X\times Y}\eta^k(dx^k,dx)\,\pi_x(dy)\,\tau^k_y(dy^k)$
    which belongs to $\Pi(\mu^k,\nu^k)$ and
  \begin{align*}
   \inf_{\pi^k\in\Pi(\mu^k,\nu^k)}\mathcal W_r^r(\pi,\pi^k)&\leq \int_{X\times X\times Y\times Y}\left(d_X^r(x^k,x)+d_Y^r(y,y^k)\right)\eta^k(dx^k,dx)\,\pi_x(dy)\,\tau^k_y(dy^k)\\&=\mathcal W^r_r(\mu,\mu^k)+\mathcal W^r_r(\nu,\nu^k).
  \end{align*}}
    The next two propositions establish \eqref{eq:informal convergence} with respect to $\mathcal{AW}_r$ for finite positive measures with common mass. The first one is formulated for $X=Y=\R$ and provides under mild assumptions an estimate of $\inf_{\pi^k\in\Pi(\mu^k,\nu^k)}\mathcal{AW}_r^r(\pi,\pi^k)$ with respect to the marginals as in \eqref{eq:rate}. Its proof relies on unidimensional tools, which we recall here. For $\eta$ a probability distribution on $\R$, we denote by $F_\eta:x\mapsto\eta((-\infty,x])$ its cumulative distribution function, and by $F_\eta^{-1}:(0,1)\to\R$ its quantile function defined for all $u\in(0,1)$ by
    \[
    F_\eta^{-1}(u)=\inf\{x\in\R\mid F_\eta(x)\ge u\}.
    \]
    The following properties are standard results (see for instance \cite[Section 6]{JoMa18} for proofs):
\begin{enumerate}[label=(\alph*)]
	\item\label{it:Fcadlag3} $F_\eta$ is c\`adl\`ag \textcolor{red}{i.e.\ right-continuous with left-hand limits}, $F_\eta^{-1}$ is c\`agl\`ad \textcolor{red}{i.e.\ left-continuous with right-hand limits};
	\item\label{it:inequality equivalence quantile function3} For all $(x,u)\in\R\times(0,1)$,
	\begin{equation}\label{eq:equivalence quantile cdf3}
	F_\eta^{-1}(u)\le x\iff u\le F_\eta(x),
	\end{equation}
	which implies, \textcolor{red}{using the notation $F_\eta(y-)$ for the left-hand limit of $F_\eta$ at $y\in\R$},
	\begin{align}\label{eq:jumps F3}
	&F_\eta(x-)<u\le F_\eta(x)\implies x=F_\eta^{-1}(u),\\
	\label{eq:jumps F23}
	\textrm{and}\quad&F_\eta(F_\eta^{-1}(u)-)\le u\le F_\eta(F_\eta^{-1}(u));
	\end{align}
	\item\label{it:Fminus of F3} For $\eta(dx)$-almost every $x\in\R$,
	\begin{equation} \label{eq:Fminus of F3}
	0<F_\eta(x),\quad F_\eta(x-)<1\quad\text{and}\quad F_\eta^{-1}(F_\eta(x))=x;
	\end{equation}
	\item\label{it:inverse transform sampling3} The image of the Lebesgue measure on $(0,1)$ by $F_\eta^{-1}$ is $\eta$.
\end{enumerate}
The property \ref{it:inverse transform sampling3} is referred to as the inverse transform sampling.

    \begin{proposition}\label{prop:adapted approximation dim1}
    	Let $\mu,\mu^k,\nu,\nu^k\in\mathcal M_r(\R)$, $k\in\N$, be measures of equal masses such that $\mu^k$ (resp.\ $\nu^k)$ converges to $\mu$ (resp.\ $\nu$) in $\mathcal W_r$.
    	Let $\pi\in\Pi(\mu,\nu)$.
    	Then:
    	\begin{enumerate}[label = (\alph*)]
    		\item \label{it:adapted approximation dim11} There exists a sequence $\pi^k\in \Pi(\mu^k,\nu^k)$, $k\in\N,$ converging to $\pi$ in $\mathcal{AW}_r$;
    		\item \label{it:adapted approximation dim12} If for all $x \in \R$ and $k\in\N$ with $\mu^k (\{ x \}) > 0$, there exists $x' \in \R$ such that
    		\[
    		\mu((-\infty,x'))\le\mu^k((-\infty,x))<\mu^k((-\infty,x])\le\mu((-\infty,x']) 
    		\]
    		(which is for instance always satisfied for $\mu^k$  non-atomic) then
    		\begin{equation}\label{eq:estimate Awr}
    		\mathcal{AW}_r^r (\pi,\pi^k) \le \mathcal W_r^r (\mu,\mu^k) + \mathcal W_r^r (\nu,\nu^k).
    		\end{equation}
    	\end{enumerate}
    \end{proposition}

\begin{remark}\label{rk: almost martingale} If $\pi$ is a martingale coupling, i.e.\ $\int_\R y'\,\pi_{x'}(dy')=x'$, $\mu(dx')$-almost everywhere, then for $\chi^k\in\Pi(\mu^k,\mu)$ an optimal coupling for $\mathcal{AW}_r(\pi^k,\pi)$, we have
	\begin{align*}
	\int_\R\left\vert x-\int_\R y\,\pi^k_x(dy)\right\vert^r\,\mu^k(dx)&=\int_{\R\times\R}\left\vert x-\int_\R y\,\pi^k_x(dy)\right\vert^r\,\chi^k(dx,dx')\\
	&\le2^{r-1}\int_{\R\times\R}\left(\vert x-x'\vert^r+\left\vert x'-\int_\R y\,\pi^k_x(dy)\right\vert^r\right)\,\chi^k(dx,dx')\\
	&=2^{r-1}\int_{\R\times\R}\left(\vert x-x'\vert^r+\left\vert \int_\R y'\,\pi_{x'}(dy')-\int_\R y\,\pi^k_x(dy)\right\vert^r\right)\,\chi^k(dx,dx')\\
	&\le2^{r-1}\int_{\R\times\R}\left(\vert x-x'\vert^r+\mathcal W_1^r(\pi^k_x,\pi_{x'})\right)\,\chi^k(dx,dx')\\
	&\le2^{r-1}\mathcal{AW}_r^r(\pi,\pi^k)\underset{k\to+\infty}{\longrightarrow}0.
	\end{align*}
	In that sense, $\pi^k,k\in\N$ is almost a sequence of martingale couplings.
\end{remark}

In the setting of Proposition \ref{prop:adapted approximation dim1} \textcolor{red}{and Remark \ref{rk: almost martingale}}, if $\mu^k$ and $\nu^k$ are also in the convex order and $\pi$ is a martingale coupling, then in view of Remark \ref{rk: almost martingale} one would naturally expect that $\pi^k$ can be slightly modified into a martingale coupling and still converge to $\pi$ in $\mathcal{AW}_r$. This actually requires a considerable amount of work and is the main message of Theorem \ref{thm:1} below. We mention that the previous proposition generalises to arbitrary Polish spaces $X$ and $Y$, as the next proposition states, but unfortunately without providing an estimate.

	 \begin{proposition}\label{prop:adapted approximation}
	Let $\mu,\mu^k\in\mathcal M_r(X),\nu,\nu^k\in\mathcal M_r(Y)$, $k\in\N$, all with equal masses and such that $\mu^k$ (resp.\ $\nu^k)$ converges to $\mu$ (resp.\ $\nu$) in $\mathcal W_r$.
	Let $\pi\in\Pi(\mu,\nu)$.
	Then there exists a sequence $\pi^k\in \Pi(\mu^k,\nu^k)$, $k\in\N,$ converging to $\pi$ in $\mathcal{AW}_r$.
\end{proposition}

The next proposition is a key ingredient which allows us to reduce the proof of Theorem \ref{thm:1} below to the case of irreducible pairs of marginals. For $\mu\in\mathcal M_1(\R)$, we denote by $u_\mu$ its potential function, that is the map defined for all $y\in\R$ by $u_\mu(y) = \int_\R |y-x|\,\mu(dx)$ (see Section \ref{sec: convex order} for more details). We recall that a pair $(\mu,\nu)$ of finite positive measures in convex order is called irreducible if $I=\{u_\mu<u_\nu\}$ is an interval and \textcolor{blue}{then, $\mu(I)=\mu(\R)$ and $\nu(\overline I)=\nu(\R)$.
\begin{remark}\label{remassbord}
   If $(\mu,\nu)$ is an irreducible pair of non-zero measures in the convex order and $a\in\R$ is such that $\nu([a,+\infty))=0$, then the convex order implies $\mu([a,+\infty))=0$, hence
\[
u_\mu(a)=a-\int_\R x\,\mu(dx)=a-\int_\R y\,\nu(dy)=u_\nu(a),
\]
so $a\notin I$. Similarly, $\nu((-\infty,a])=0\implies a\notin I$. We deduce that $\nu$ must assign positive mass to any neighbourhood of each of the boundaries of $I$.
\end{remark}}

According to \cite[Theorem A.4]{BeJu16}, for any pair $(\mu,\nu)$ of probability measures in convex order, there exist $N\subset\N$ and a sequence $(\mu_n,\nu_n)_{n\in N}$ of irreducible pairs of sub-probability measures in convex order such that
\begin{align*}
\mu=\eta+\sum_{n\in N}\mu_n,\quad\nu=\eta+\sum_{n\in\N}\nu_n\quad\text{and}\quad\left\{u_\mu<u_\nu\right\}=\bigcup_{n\in N}\left\{u_{\mu_n}<u_{\nu_n}\right\},
\end{align*}
where the union is disjoint and $\eta=\mu\vert_{\{u_\mu=u_\nu\}}$. The sequence $(\mu_n,\nu_n)_{n\in N}$ is unique up to rearrangement of the pairs and is called the decomposition of $(\mu,\nu)$ into irreducible components. Moreover, for any martingale coupling $\pi\in\Pi_M(\mu,\nu)$, there exists a unique sequence of martingale couplings $\pi_n\in\Pi_M(\mu_n,\nu_n)$, $n\in N$ such that
\[
\pi=\chi+\sum_{n\in N}\pi_n,
\]
where $\chi=(\operatorname{id},\operatorname{id})_\ast\eta$ and $\ast$ denotes the pushforward operation. This sequence satisfies
\begin{equation}\label{eq:decomposition martingale couplings}
\forall n\in N,\quad\pi_n(dx,dy)=\mu_n(dx)\,\pi_x(dy).
\end{equation}

\begin{proposition}\label{prop:approximationofdecomp}
	Let $(\mu^k,\nu^k)_{k\in\N}$ be a sequence of pairs of probability measures on the real line in convex order which converge to $(\mu,\nu)$ in $\mathcal W_1$.
	Let $(\mu_n,\nu_n)_{n\in N}$ be the decomposition of $(\mu,\nu)$ into irreducible components and $\eta=\mu\vert_{\{u_\mu=u_\nu\}}$. Then there exists for any $k\in\N$ a decomposition of $(\mu^k,\nu^k)$ into pairs of sub-probability measures $(\mu^k_n,\nu^k_n)_{n\in N}$, $(\eta^k,\upsilon^k)$ which are in convex order such that
	\begin{gather} \label{eq:approx decomp}
	\eta^k + \sum_{n\in N} \mu^k_n = \mu^k,\quad \upsilon^k + \sum_{n\in N} \nu^k_n = \nu^k,\quad k\in\N, \\ \label{eq:conv of decomp}
	\lim_{k\to+\infty} \eta^k = \eta,\quad \lim_{k\to+\infty} \mu^k_n = \mu_n,\quad \lim_{k\to+\infty} \nu^k_n = \nu_n,\quad \lim_{k\to+\infty} \upsilon^k =\eta\quad \text{in }\mathcal W_1.
	\end{gather}
\end{proposition}

We can now state our main result, namely Theorem \ref{thm:1} below. Any martingale coupling whose marginals are approximated by probability measures in convex order can be approximated by martingale couplings with respect to the adapted Wasserstein distance.

\begin{theorem}\label{thm:1} Let $\mu^k,\nu^k\in\mathcal P_r(\R)$, $k\in\N$, be in convex order and respectively converge to $\mu$ and $\nu$ in $\mathcal W_r$. Let $\pi\in\Pi_M(\mu,\nu)$. Then there exists a sequence of martingale couplings $\pi^k\in\Pi_M(\mu^k,\nu^k)$, $k\in\N$ converging to $\pi$ in $\mathcal{AW}_r$.
\end{theorem}
\begin{proof}[Sketch of the proof]

	We will first argue that it is enough to consider the case $r=1$. Thanks to Proposition \ref{prop:approximationofdecomp}, we can also reduce the proof to the case of irreducible pairs of marginals $(\mu,\nu)$, whose single irreducible component is denoted $(\ell,\rho)=I$.

	\emph{Step 1.} Fix a martingale coupling $\pi\in\Pi_M(\mu,\nu)$. When directly approximating $\pi$ we would face technical obstacles. First, for $K$ a compact subset of $I$, $\mu\vert_K\times\pi_x$ is not necessarily compactly supported. Moreover, $\nu$ may put mass on the boundary of $I$. To overcome \textcolor{red}{successively these two difficulties, the kernel $\pi_x$ is first compactified to a compact set $[-R,R]$, where $R>0$ (when $|\ell|\vee |\rho|<\infty$, one may choose $R$ equal to this maximum)}, and then pushed forward by the map $y\mapsto\alpha(y-x)+x$, where $\alpha\in(0,1)$. This yields a martingale coupling $\pi^{R,\alpha}$ close to $\pi$ and easier to approximate, between $\mu$ and a probability measure $\nu^{R,\alpha}$ dominated by $\nu$ in the convex order. We find compact sets $K,L\subset I$ such that the restriction $\pi^{R,\alpha}\vert_{K\times\R}$ is compactly supported on $K\times L$ and concentrated on $K\times\mathring L$, where $\mathring L$ denotes the interior of $L$. Since, by irreducibility, $\nu$ puts mass onto any neighbourhood of the boundary of $I$, $\nu^{R,\alpha}$ assigns positive mass to two open sets $L_-$, $L_+$ on both sides of $K$ with positive distance to $K$. This is summarised in Figure \ref{fig:bild1}, where $J$ denotes a compact subset of $I$ that is  large enough.

\begin{figure}[h]
	\centering
	\begin{tikzpicture}[decoration = {random steps,segment length=2pt,amplitude=0.1pt}]
	\draw[decorate, line width=\x mm, color = {rgb,255:red,92; green,74; blue,114}, (-)] (0,0) -- (10,0) ;
	\draw[decorate, line width=\x mm, color = {rgb,255:red,163; green,88; blue,109}, |-| ] (0.5,0.2) -- (9.5,0.2) ;

	\draw[decorate, line width=\x mm, color = {rgb,255:red,244; green,135; blue,75}, |-|] (1,-0.2) -- (2,-0.2) ;
	\draw[decorate, line width=\x mm, color = {rgb,255:red,244; green,135; blue,75}, |-| ] (8,-0.2) -- (9,-0.2) ;
	\draw[decorate, line width=\x mm, color = {rgb,255:red,194; green,85; blue,25}, |-| ] (4,-0.2) -- (7,-0.2) ;

	\draw (0,0.1) node[below left, scale = 1.2]{$\ell$} ;
	\draw (10,0) node[below right, scale = 1.2]{$\rho$} ;
	\draw (1.5,-0.2) node[below, scale = 1.2]{$L_-$} ;
	\draw (8.5,-0.2) node[below, scale = 1.2]{$L_+$} ;
	\draw (5.5,-0.2) node[below, scale = 1.2]{$K$} ;
	\draw (5,0.2) node[above, scale = 1.2]{$J$} ;
	\end{tikzpicture}
	\caption{Intervals involved in the proof. The boundaries of the closed intervals are vertical bars and those of the open intervals are parentheses.\label{fig:bild1}}
\end{figure}
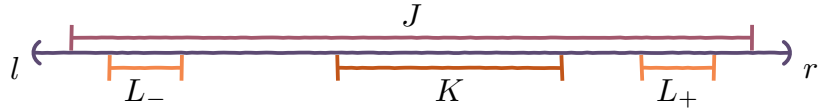


	\emph{Step 2.} It is possible to find an approximating sequence $(\hat\pi^k=\hat\mu^k\times\hat\pi^k_x)_{k\in\N}$ for the sub-probability martingale coupling $\pi^{R,\alpha}\vert_{K\times\R}$ from step 1. Unfortunately $\hat\pi^k$ is not necessarily a martingale coupling. Therefore, we free up some mass, and use the one available on the left and right of $K$ in $L_-$ and $L_+$ to adjust the barycenters of the kernels $\pi^k_x$. Hence we find a sequence $(\tilde\pi^k=\hat\mu^k\times\tilde\pi^k_x)_{k\in\N}$ of sub-probability martingale couplings approximating $\pi^{R,\alpha}\vert_{K\times\R}$.

	\emph{Step 3.} By construction, up to multiplication by a factor smaller than and close to $1$, the first marginal of $\tilde\pi^k$ satisfies $\hat\mu^k\le\mu^k$. Moreover, its second marginal denoted $\tilde\nu^k$ is such that there exists a probability measure $\nu^{R,\alpha,k}$ which satisfies $\tilde\nu^k\le\nu^{R,\alpha,k}\le_c\nu^k$. Then by using the uniform convergence of potential functions, we show that for $k$ sufficiently large there exist sub-probability martingale couplings $\eta^k\in\Pi_M(\mu^k-\hat\mu^k,\nu^{R,\alpha,k}-\tilde\nu^k)$ so that the sum $\eta^k+\tilde\pi^k$ is a martingale coupling in $\Pi_M(\mu^k,\nu^{R,\alpha,k})$, where the second marginal is dominated by $\nu^k$ in the convex order.

	\emph{Step 4.} In the last step, we use the inverse-transform martingale coupling between $\nu^{R,\alpha,k}$ and $\nu^k$, see \cite{JoMa18}, to change $\eta^k+\tilde\pi^k$ to a martingale coupling $\pi^k \in \Pi(\mu^k,\nu^k)$.
	Finally, we estimate the $\mathcal{AW}_1$-distance of $\pi$ to $\pi^k$.
\end{proof}

	\section{On the adapted weak topology}
	\label{sec:adapted weak topology}

We begin this section with a lemma on uniform integrability which will prove very handy throughout the paper. We formulate it for finite positive measures on $X$, but it is understood that $(X,x_0)$ is replaced with $(Y,y_0)$ for measures on $Y$.

	\begin{lemma}\label{lem:uniform integrability}
    Let $r\ge 1$ and $\mu\in\mathcal M_r(X)$.
    For $\epsilon>0$, let
		\begin{equation}\label{eq:def Iepsilonr}
		I_\epsilon^r(\mu):=\sup_{\substack{\tau \in \mathcal M(X) \\ \tau \le \mu,~\tau(X)\leq \epsilon}}\int_{X}d^r_X(x,x_0)\,\tau(dx).
		\end{equation}
    \begin{enumerate}[label = (\alph*)]
      \item \label{it:croisieps} $I_\epsilon^r$ is monotone in $\mu$, i.e., $\mu \leq \mu' \in \mathcal M_r(X)$ implies that $I_\epsilon^r(\mu) \leq I_\epsilon^r(\mu')$.
      \item \label{it:IepsilonrVanishes} The value of $I_\epsilon^r(\mu)$ vanishes as $\epsilon \to 0$.
      \item \label{it:Iespilonr} For any $\mu' \in \mathcal M_r(X)$ such that $\mu(X) = \mu'(X)$ we have
      \begin{equation}\label{ineq: Iepsilonr}
  		    I_\epsilon^r(\mu)\le2^{r-1}\left(I_\epsilon^r(\mu')+\mathcal W_r^r(\mu,\mu')\right).
      \end{equation}
      \item \label{it:uniform integrability} Let $\mu,\mu^k\in\mathcal M_r(X)$, $k\in\N$ be with equal masses such that $\mu^k$ converges weakly to $\mu$. Then
      \[
      \mathcal W_r(\mu^k,\mu)\underset{k\to+\infty}{\longrightarrow}0\iff\sup_{k\in\N}I^r_\varepsilon(\mu^k)\underset{\varepsilon\to0}{\longrightarrow}0\quad\text{and}\quad\sup_{k\in\N}\int_{X}d_X^r(x,x_0)\,\mu^k(dx)<+\infty.
      \]
      \item \label{it:IepsilonrCvx}
      Finally, if $X=\R^d$ and $\mu\le_c\nu$ with $\nu\in\mathcal M_1(\R^d)$, then $I^1_\epsilon(\mu)\le I^1_\epsilon(\nu)$.
    \end{enumerate}
	\end{lemma}
	\begin{remark} If $\mu(X)\le\epsilon$, then $I_\epsilon^r(\mu)$ is simply the $r$-th moment of $\mu$.
	\end{remark}
	\begin{proof}
    The first point \ref{it:croisieps} is an easy consequence of the definition of $I^r_\epsilon$.

    Next we check \ref{it:IepsilonrVanishes}. Let $\mu\in\mathcal M_r(X)$ be such that $\mu(X)>0$. Since
    \begin{equation}\label{scali}
    I^r_\varepsilon(\mu)=\mu(X)I_{\frac{\varepsilon}{\mu(X)}}^r\left(\frac{\mu}{\mu(X)}\right),
    \end{equation}
    to check convergence of $I^r_\varepsilon(\mu)$ to $0$ as $\varepsilon\to 0$, we may suppose that $\mu\in\mathcal P_r(X)$. Let $\varepsilon\in(0,1)$. For $\eta\in\mathcal M_r(X)$, we denote by $\overline\eta$ the image of $\eta$ under the map $x\mapsto d_X^r(x,x_0)$. \textcolor{red}{Let $\tau\in\mathcal M(X)$ be such that $\tau\le\mu$ and $0<\tau(X)\le\varepsilon$. Since $\tau\le\mu$, we have $\overline\tau\le\overline\mu$. Using \eqref{eq:jumps F23} for the last inequality, we get for all $u\in(0,1)$
    \[
    1-F_{\overline\tau/\tau(X)}\left(F_{\overline\mu}^{-1}(1-\tau(X)u)\right)=\frac{\overline\tau((F_{\overline\mu}^{-1}(1-\tau(X)u),+\infty))}{\tau(X)}\le\frac{\overline\mu((F_{\overline\mu}^{-1}(1-\tau(X)u),+\infty))}{\tau(X)}\le u,
    \]
    hence $F_{\overline\tau/\tau(X)}(F_{\overline\mu}^{-1}(1-\tau(X)u))\ge1-u$ and by \eqref{eq:equivalence quantile cdf3}, $F_{\overline\tau/\tau(X)}^{-1}(1-u)\le F_{\overline\mu}^{-1}(1-\tau(X)u)$.
    Using the inverse transform sampling, we deduce
    \begin{align}
    \int_Xd_X^r(x,x_0)\,\tau(dx)&=\tau(X)\int_0^1F_{\overline\tau/\tau(X)}^{-1}(1-u)\,du\notag\\
    &\le\tau(X)\int_0^1F_{\overline\mu}^{-1}(1-\tau(X)u)\,du=\int_{1-\tau(X)}^1F_{\overline\mu}^{-1}(u)\,du\le\int_{1-\varepsilon}^1F_{\overline\mu}^{-1}(u)\,du.\label{majointdrtau}
    \end{align}
Hence $I^r_\varepsilon(\mu)\le \int_{1-\varepsilon}^{1}F_{\overline\mu}^{-1}(u)\,du$ where the right-hand side vanishes as $\varepsilon \to 0$ since, as $\mu\in\mathcal P_r(X)$, $\int_0^{1}F_{\overline\mu}^{-1}(u)\,du=\int_X d_X(x,x_0)^r\,\mu(dx)<+\infty$. Let us  check the equality
    \begin{equation}\label{eq:expression Irepsilon}
    I^r_\varepsilon(\mu)=\int_{1-\varepsilon}^{1}F_{\overline\mu}^{-1}(u)\,du,
    \end{equation}
    that will come in handy for the proof of claim \ref{it:IepsilonrCvx}} \textcolor{purple}{by setting 
  \begin{align}\label{eq:righthand}
   \tau^*(dx)=\left(\mathbbm1_{A_\varepsilon}(x)+\frac{F_{\overline\mu}(y_\varepsilon)-(1-\varepsilon)}{\mu(B_\varepsilon)}\mathbbm1_{B_\varepsilon}(x)\right)\,\mu(dx),
    \end{align}
    where
    \[
    y_\varepsilon=F_{\overline\mu}^{-1}(1-\varepsilon),\quad A_\varepsilon=\{x\in\R\mid d_X^r(x,x_0)>y_\varepsilon\}\quad\text{and}\quad B_\varepsilon=\{x\in\R\mid d_X^r(x,x_0)=y_\varepsilon\},
    \]
    and the second summand of the right-hand side in \eqref{eq:righthand} is taken to be zero if $\mu(B_\varepsilon)=0$. Since $A_\varepsilon\cap B_\varepsilon=\emptyset$ and, by \eqref{eq:jumps F23}, $$\mu(B_\varepsilon)={\overline\mu}(\{y_\varepsilon\})=F_{\overline\mu}(y_\varepsilon)-F_{\overline\mu}(F_{\overline\mu}^{-1}(1-\varepsilon)-)\ge F_{\overline\mu}(y_\varepsilon)-(1-\varepsilon),$$ hence $\tau^*\le\mu$.  Moreover, $\overline{\tau^*}$ is the measure dominated by $\overline{\mu}$ with mass equal to $\varepsilon$ which is the largest in  stochastic order. Indeed, one easily checks that
    \begin{align*}
      &\overline{\tau^*}(dy)=\mathbbm1_{y>y_\varepsilon}\overline{\mu}(dy)+\left(F_{\overline\mu}(y_\varepsilon)-(1-\varepsilon)\right)\delta_{y_\varepsilon}(dy)\mbox{ so that }\overline{\tau^*}(\R)=\varepsilon,\\
      &\forall y\in\R,\,F_{\overline{\tau^*}/\varepsilon}(y)=\mathbbm1_{y\ge y_\varepsilon}\frac{F_{\overline{\mu}}(y)-(1-\varepsilon)}{\varepsilon}\mbox{ and }\forall u\in (0,1),\,F^{-1}_{\overline{\tau^*}/\varepsilon}(1-u)=F_{\overline{\mu}}^{-1}(1-\varepsilon u).
    \end{align*}
With the inverse transform sampling, the latter equality implies that \begin{align*}
   \int_Xd_X^r(x,x_0)\,\tau^*(dx)=\varepsilon\int_0^1F^{-1}_{\overline{\tau^*}/\varepsilon}(u)du=\varepsilon\int_0^1F_{\overline{\mu}}^{-1}(1-\varepsilon u)du
=\int_{1-\varepsilon}^1F_{\overline\mu}^{-1}(u)\,du\end{align*} so that \eqref{eq:expression Irepsilon} holds.}

    To see \ref{it:Iespilonr}, fix $\mu' \in \mathcal M(X)$ with $\mu(X) = \mu'(X)$.
    We denote by $\pi(dx,dx')=\mu(dx)\,\pi_{x}(dx') \in \Pi(\mu,\mu')$ a $\mathcal W_r$-optimal coupling.
    Let $\tau\in\mathcal M(X)$ be such that $\tau\leq\mu$ and $\tau(X)\le\epsilon$. Let $\tau'\in\mathcal M(X)$ be defined by
    \[
    \tau'(dx')=\int_{x\in X}\pi_{x}(dx')\,\tau(dx).
    \]
    Since $\pi$ is element of $\Pi(\mu,\mu')$, we find $\tau'\leq\mu'$ and $\tau(X)=\tau'(X)$. Then
    \begin{align*}
      \int_X d_X^r(x,x_0) \, \tau(dx) &\leq 2^{r-1} \int_{X \times X} \left(d_X^r(x',x_0) + d_X^r(x,x')\right) \, \pi_{x}(dx')\, \tau(dx) \\
      &\leq 2^{r-1} \left( I_\epsilon^r(\mu') + \int_{X\times X }d_X^r(x,x') \, \pi(dx,dx')\right),
    \end{align*}
    which shows by optimality of $\pi$ the assertion.

    We now show \ref{it:uniform integrability}. Let $\mu,\mu^k\in\mathcal M_r(X)$ be with equal masses such that $\mu^k$ converges weakly to $\mu$. According to \eqref{scali}, we may suppose that $\mu,\mu^k\in\mathcal P_r(X)$.

    Suppose that $\mathcal W_r(\mu^k,\mu)$ vanishes as $k$ goes to $+\infty$. Then the sequence of the $r$-th moments of $\mu^k$, $k\in\N$ is bounded since it converges to the $r$-th moment of $\mu$. Let $\eta>0$. Let $k_0\in\N$ be such that for all $k>k_0$, $\mathcal W_r^r(\mu^k,\mu)<\eta$. Then \ref{it:Iespilonr} yields for $\varepsilon>0$
    \[
    \sup_{k\in\N}I^r_\varepsilon(\mu^k)\le\sum_{k\le k_0}I^r_\varepsilon(\mu^k)+\sup_{k>k_0}I^r_\varepsilon(\mu^k)\le\sum_{k\le k_0}I^r_\varepsilon(\mu^k)+2^{r-1}(I^r_\varepsilon(\mu)+\eta).
    \]
    According to \ref{it:IepsilonrVanishes} we then get
    \[
    \limsup_{\varepsilon\to0}\sup_{k\in\N}I^r_\varepsilon(\mu^k)\le2^{r-1}\eta.
    \]
    Since $\eta>0$ is arbitrary, we deduce that $\sup_{k\in\N}I^r_\varepsilon(\mu^k)$ vanishes with $\varepsilon$.

    Conversely, suppose that $\sup_{k\in\N}I^r_\varepsilon(\mu^k)$ vanishes with $\varepsilon$ and the sequence of the $r$-th moments of $\mu^k$, $k\in\N$ is bounded. By Skorokhod's representation theorem, there exist random variables $X$ and $X^k$, $k\in\N$, defined on a common probability space such that $X$, resp.\ $X^k$ is distributed according to $\mu$, resp.\ $\mu^k$ and $X^k$ converges almost surely to $X$. Then for all $M>0$,
    \begin{equation*}
    \mathcal W_r^r(\mu^k,\mu)\le\E[d_X^r(X^k,X)]=\E[d_X^r(X^k,X)\mathbbm1_{\{d_X^r(X^k,X)<M\}}]+\E[d_X^r(X^k,X)\mathbbm1_{\{d_X^r(X^k,X)\ge M\}}].
    \end{equation*}
    By the dominated convergence theorem, we deduce
    \[
    \limsup_{k\to+\infty}\mathcal W_r^r(\mu^k,\mu)\le\limsup_{k\to+\infty}\E[d_X^r(X^k,X)\mathbbm1_{\{d_X^r(X^k,X)\ge M\}}].
    \]
    Let us then prove that the right-hand side vanishes as $M$ goes to $+\infty$. Let $\eta>0$. Let $\varepsilon>0$ be such that $I_\varepsilon^r(\mu)+\sup_{k\in\N}I_\varepsilon^r(\mu^k)<\eta$. By Markov's inequality, we have
    \begin{equation*}
    \sup_{k\in\N}\E[\mathbbm1_{\{d_X^r(X^k,X)\ge M\}}]\le\sup_{k\in\N}\frac{\E[d_X^r(X^k,X)]}{M}\le\frac{2^{r-1}}{M}\sup_{k\in\N}\int_X d_X^r(x,x_0)\,(\mu^k+\mu)(dx),
    \end{equation*}
    where the right-hand side vanishes as $M$ goes to $+\infty$. Therefore, there exists $M_0>0$ such that for all $k\in\N$ and $M>M_0$,
    \begin{align*}
    \E[d_X^r(X^k,X)\mathbbm1_{\{d_X^r(X^k,X)\ge M\}}]&\le2^{r-1}\left(\E[d^r_X(X^k,x_0)\,\mathbbm1_{\{d_X^r(X^k,X)\ge M\}}]+\E[d^r_X(x_0,X)\mathbbm1_{\{d_X^r(X^k,X)\ge M\}}]\right)\\
    &\le 2^{r-1}\left(I^r_\varepsilon(\mu^k)+I^r_\varepsilon(\mu)\right)<2^{r-1}\eta.
    \end{align*}
   Therefore, for all $M>M_0$,
    \[
    \limsup_{k\to+\infty}\E[d_X^r(X^k,X)\mathbbm1_{\{d_X^r(X^k,X)\ge M\}}]\le2^{r-1}\eta.
    \]
    Since $\eta$ is arbitrary, this proves the assertion.

    Finally, we want to show \ref{it:IepsilonrCvx}. Let $X = \R^d$ and $\mu \leq_c \nu$ with $\nu\in\mathcal M_1(\R^d)$. According to \eqref{scali}, we may suppose that $\mu,\nu\in\mathcal P_1(\R^d)$. Again, we write $\overline{\mu}$ and $\overline{\nu}$ for the pushforward measures of $\mu$ and $\nu$ under the map $(x \mapsto |x-x_0|^r)$.
    First, we note that $\overline{\mu}$ is dominated by $\overline{\nu}$ in the increasing convex order.
    Indeed, let $f \in C(X)$ be convex and nondecreasing, then $x \mapsto f(|x-x_0|^r)$ constitutes a convex, continuous function.
    Thus,
    \[
      \int_\R f(y) \, \overline{\mu}(dy) = \int_{\R^d} f(|x-x_0|^r) \, \mu(dx) \leq \int_{\R^d} f(|x-x_0|^r) \, \nu(dx) = \int_\R f(y) \, \overline{\nu}(dy).
    \]
    The convex increasing order is characterised by the following family of inequalities (see for instance \cite[Theorem 2.4]{AlCoJo17a}): for all $0 \leq \varepsilon \leq 1$,
    \[
      \int_{1-\varepsilon}^{1} F_{\overline{\mu}}^{-1}(y) \, dy \leq \int_{1-\varepsilon}^{1} F_{\overline{\nu}}^{-1}(y) \, dy.
    \]
    The identity \eqref{eq:expression Irepsilon} concludes the proof.
  \end{proof}

We now prove Proposition \ref{prop:adapted approximation dim1}. A handy tool in the construction of the approximative sequence $(\pi^k)_{k\in \N}$ are copulas.
Recall that a two-dimensional copula is an element $C$ of $\Pi(\lambda,\lambda)$ where $\lambda$ is the uniform distribution on $(0,1)$.
A coupling $\pi$ is an element of $\Pi(\mu,\nu)$ if and only if it can be written as the push-forward of a copula $C$ under the quantile map $(F_\mu^{-1}, F_\nu^{-1}) \colon  (0,1) \times (0,1) \to \R \times \R$.
Clearly, if $C$ is a copula then $\pi = (F_\mu^{-1}, F_\nu^{-1})_\ast C$ is contained in $\Pi(\mu,\nu)$.
On the other hand, if $\pi \in \Pi(\mu,\nu)$ is given, we can construct a copula $C$ by
\[
C(du,dv) = \mathbbm{1}_{(0,1)}(u) \, du \, C_u(dv),
\]
where $C_u$ is given by
\begin{equation}\label{eq:def Copula}
C_u = ((y,w) \mapsto F_\nu(y-) + w \nu(\{y\}))_\ast (\pi_{F_\mu^{-1}(u)} \times \lambda).
\end{equation}
In particular, we have that $u\mapsto C_u$ is constant on the jumps on $F_\mu$.
\textcolor{blue}{The fact that the second marginal distribution of $C$ is indeed uniformly distributed on $(0,1)$ is a direct consequence of the inverse transform sampling and the well-known result (see for instance \cite[Lemma 6.6]{JoMa18} for a proof) that for any $\eta\in\mathcal P(\R)$,
\begin{equation}\label{eq:pushforward uniform}
((z,w) \mapsto F_\eta(z-) + w \eta(\{z\}))_\ast (\eta \times \lambda) = \lambda.
\end{equation}}
\textcolor{purple}{Finally, we check the identity $\pi = (F_\mu^{-1}, F_\nu^{-1})_\ast C$.
Let $w \in (0,1]$ and continue by distinguishing two cases:
On the one hand, if $\nu(\{y\}) > 0$ then we have by \eqref{eq:jumps F3}
\begin{equation} \label{eq:ForMargsOfCopula}
	F_\nu^{-1}(F_\nu(y-) + w \nu(\{y\}))=y.
\end{equation}
On the other hand, we derive from \eqref{eq:Fminus of F3} that \eqref{eq:ForMargsOfCopula} holds for $\nu$-almost every $y \in \{ z \in \R \colon \nu(\{ z \}) = 0 \}$. 
Hence, we obtain for $\lambda$-almost every $u \in (0,1)$
\begin{equation}\label{eq:piQuantile=Quantile C}
\pi_{F_\mu^{-1}(u)}=(F_\nu^{-1})_\ast C_u
\end{equation}
and conclude with $\pi= (F_\mu^{-1}, F_\nu^{-1})_\ast C$.}
\begin{proof}[Proof of Proposition \ref{prop:adapted approximation dim1}]
	Because of homogeneity of the $\mathcal{AW}_r$- and $\mathcal W_r$-distances, we can suppose w.l.o.g. that $\mu,\mu^k,\nu,\nu^k$ and $\pi$ are probability measures. Let $C$ be the copula defined by $C(du,dv)=\mathbbm1_{(0,1)}(u)\,du\,C_u(dv)$, where $C_u$ is given by \eqref{eq:def Copula}.

	In order to define $\pi^k$, we construct associated copulas $C^k$ where $u\mapsto C^k_u$ is constant on the jumps of $F_{\mu^k}$.
	Let
	\begin{gather*}
	\theta_k \colon \R \times (0,1) \to (0,1),\ (x,w) \mapsto F_{\mu^k}(x-) + w \mu^k( \{ x \} ), \\
	C_u^k(dv) = \int_{w = 0}^1 C_{\theta_k\left(F_{\mu^k}^{-1}(u),w\right)}(dv) \, dw, \\
	\pi^k = (F_{\mu^k}^{-1}, F_{\nu^k}^{-1})_\ast C^k = (F_{\mu^k}^{-1}, F_{\nu^k}^{-1})_\ast (\mathbbm{1}_{(0,1)}(u) \, du \, C_u^k(dv)).
	\end{gather*}
	The fact that $C^k$ is a copula, and therefore $\pi^k\in\Pi(\mu^k,\nu^k)$, is a direct consequence of \eqref{eq:pushforward uniform} and the inverse transform sampling. Since $u \mapsto C_u$ and $u\mapsto C^k_u$ are constant on the jumps of $F_\mu$ and $F_{\mu^k}$ respectively, reasoning like in the derivation of \eqref{eq:piQuantile=Quantile C}, we have for $du$-almost every $u$ in $(0,1)$
	\[
	\pi_{F_{\mu}^{-1}(u)} = (F_{\nu}^{-1})_\ast C_u, \quad \pi^k_{F_{\mu^k}^{-1}(u)} = (F_{\nu^k}^{-1})_\ast C^k_u.
	\]
	Moreover, since $(u \mapsto (F_\mu^{-1}(u),F_{\mu^k}^{-1}(u)))_\ast \lambda$ is a coupling between $\mu$ and $\mu^k$, namely the comonotonous coupling, we have \textcolor{red}{using the definition of $\mathcal{AW}_r(\pi,\pi^k)$ as an infimum over $\Pi(\mu,\mu^k)$, \textcolor{purple}{cf.}\ \eqref{eq:defAW2},}
	\begin{align}\label{eq:pseudo estimate Awr}
	\begin{split}
	\mathcal{AW}_r^r(\pi,\pi^k) & \le \int_0^1\left(\vert F_\mu^{-1}(u) - F_{\mu^k}^{-1}(u) \vert^r + \mathcal W_r^r(\pi_{F_\mu^{-1}(u)}, \pi^k_{F_{\mu^k}^{-1}(u)})\right)\, du \\
	& = \mathcal W_r^r(\mu,\mu^k) + \int_0^1 \mathcal W_r^r \left( (F_\nu^{-1})_\ast C_u, (F_{\nu^k}^{-1})_\ast C^k_u \right) \, du.
	\end{split}
	\end{align}
	By Minkowski's inequality we have
	\begin{align}\label{eq:triangular inequality Wr}
	\begin{split}
	\left( \int_0^1 \mathcal W_r^r \left( (F_\nu^{-1})_\ast C_u, (F_{\nu^k}^{-1})_\ast  C^k_u \right) \, du \right)^{\frac{1}{r}} & \le \left( \int_0^1 \mathcal W_r^r \left( (F_\nu^{-1})_\ast C_u, (F_\nu^{-1})_\ast C^k_u \right) \, du \right)^{\frac{1}{r}} \\
	& \phantom{ \le } + \left( \int_0^1 \mathcal W_r^r \left( (F_\nu^{-1})_\ast C^k_u, (F_{\nu^k}^{-1})_\ast C^k_u \right) \, du \right)^{\frac{1}{r}}.
	\end{split}
	\end{align}
	Since for any $\eta \in \mathcal P(\R)$ the map $F_\eta^{-1} \circ F_{C^k_u}^{-1}$ is non-decreasing, we have (see for instance \cite[Lemma A.3]{AlCoJo17b}) that for $dw$-almost every $w\in(0,1)$,
	\[ F_{\eta}^{-1}(F_{C^k_u}^{-1}(w)) = F_{(F_{\eta}^{-1})_\ast C^k_u}^{-1}(w). \]
	Hence, we deduce
	\begin{align}\label{eq:rewrite Wr}
	\begin{split}
	\int_{(0,1)} \mathcal W_r^r \left( (F_\nu^{-1})_\ast C^k_u, (F_{\nu^k}^{-1})_\ast C^k_u \right) \, du & = \int_{(0,1)}\int_{(0,1)}\vert F_\nu^{-1} (F_{C^k_u}^{-1} (w) ) - F_{\nu^k}^{-1} (F_{C^k_u}^{-1} (w) ) \vert^r \, dw \, du \\
	& = \int_{(0,1)}\int_{(0,1)} \vert F_\nu^{-1} (v) - F_{\nu^k}^{-1} (v) \vert^r \, C^k_u(dv) \, du \\
	& = \int_{(0,1)}\vert F_\nu^{-1} (v) - F_{\nu^k}^{-1} (v) \vert^r \, dv = \mathcal W_r^r(\nu,\nu^k) \to 0,
	\end{split}
	\end{align}
	where we used inverse transform sampling in the second equality.
	At this stage, we can already show \ref{it:adapted approximation dim12} \textcolor{red}{of Proposition \ref{prop:adapted approximation dim1}}.
	Indeed, the assumption made in \ref{it:adapted approximation dim12} ensures that any jump of $F_{\mu_k}$ is included in a jump of $F_\mu$.
	We already noted that $u\mapsto C_u$ is constant on the jumps of $F_\mu$ and therefore also constant on the jumps of $F_{\mu^k}$.
	This yields for all $u,w\in(0,1)$ that $C_{\theta_k(F_{\mu^k}^{-1}(u),w)}=C_u$ and particularly $C^k_u=C_u$, which  causes the first term on the right-hand side of \eqref{eq:triangular inequality Wr} to vanish.
	Then the estimate \eqref{eq:estimate Awr} follows immediately from \eqref{eq:pseudo estimate Awr}, \eqref{eq:triangular inequality Wr} and \eqref{eq:rewrite Wr}.

	To  obtain \ref{it:adapted approximation dim11} and in view of \eqref{eq:pseudo estimate Awr}, \eqref{eq:triangular inequality Wr} and \eqref{eq:rewrite Wr}, it is sufficient to show
	\[ \int_0^1 \mathcal W_r^r \left( (F_\nu^{-1})_\ast C_u, (F_\nu^{-1})_\ast C^k_u \right) \, du \to 0. \]
	This is achieved in two steps:
	First, we show for $du$-almost every $u \in(0,1)$ that
	\begin{equation}\label{eq:pointwise convergence}
	\mathcal W_r((F_\nu^{-1})_\ast C_u, (F_\nu^{-1})_\ast C_u^k) \to 0.
	\end{equation}
	Second, we prove that
	\begin{equation}\label{eq:uniformly integrable}
	u \mapsto \mathcal W_r^r((F_\nu^{-1})_\ast C_u, (F_\nu^{-1})_\ast C_u^k) \quad k \in \N,
	\end{equation}
	is uniformly integrable on $(0,1)$ with respect to $\lambda$.

	To show \eqref{eq:pointwise convergence}, note that $\mathcal W_r$-convergence is already determined by a countable family $\mathcal C \subset \Phi_r(\R)$ (see \cite[Theorem 4.5.(b)]{EtKu}).
	For this reason, it is sufficient to show that for all $f \in \mathcal C$, for $du$-almost every $u\in(0,1)$,
	\begin{equation}\label{eq:Wr Convergence copula}
	\int_{(0,1)} f(F_\nu^{-1}(v)) \, C_u^k(dv) \to g(u):=\int_{(0,1)} f(F_\nu^{-1}(v)) \, C_u(dv),\quad k\to+\infty,
	\end{equation}
	where the integrals are $du$-almost everywhere well defined because of the inverse transform sampling, the fact that $f\in\Phi_r(\R)$ and $\nu\in\mathcal P_r(\R)$. For $u\in(0,1)$, let $x_u = F_\mu^{-1}(u)$ and $x^k_u = F_{\mu^k}^{-1}(u)$. Let $\mathcal U\subset(0,1)$ be the set of continuity points of $F_\mu^{-1}$ and define
	\[
	\mathcal U_c=\{u\in\mathcal U\mid F_\mu\text{ is continuous at }x_u\}\quad\text{and}\quad\mathcal U_d=\{u\in\mathcal U\backslash\mathcal U_c\mid u\in(F_\mu(x_u-),F_\mu(x_u))\}.
	\]
	By monotonicity of $F_\mu^{-1}$, the complement of $\mathcal U$ in $(0,1)$ is at most countable, and since $\mu$ has countably many atoms, the complement of $\mathcal U_d$ in $\mathcal U\backslash\mathcal U_c$ is also at most countable. We deduce that it is sufficient to show \eqref{eq:Wr Convergence copula} for $du$-almost all $u\in\mathcal U_c\cup\mathcal U_d$.
	Let then $u\in\mathcal U$. If $\mu^k(\{x^k_u\})=0$, then $C_u^k = C_u$ and
	\[
	\int_{(0,1)} f(F_\nu^{-1}(v)) \, C_u^k(dv)=g(u).
	\]
	From now on and until \eqref{eq:Wr Convergence copula} is proved, we suppose w.l.o.g.\ that $\mu^k(\{x^k_u\})>0$ for all $k\in\N$. Then
	\begin{equation}\label{eq:Wr Copula}
	\int_{(0,1)} f(F_\nu^{-1}(v)) \, C_u^k(dv) =\frac{1}{\mu^k(\{ x_u^k \})}\int_{F_{\mu^k}(x_u^k - )}^{F_{\mu^k}(x_u^k)} g(w) \, dw .
	\end{equation}
	Define $l_k = \inf_{n \geq k} x_u^n$ and $r_k = \sup_{n \geq k} x_u^n$.
	Since $u \in \mathcal U$ we find $l_k \nearrow x_u$ and $r_k \searrow x_u$ when $k$ goes to $+\infty$.
	Due to right continuity of $F_\mu$ and left continuity of $x \mapsto F_\mu(x-)$ we have
	\[
	F_\mu(x_u-) = \lim_p F_\mu(l_p -)\quad\text{and}\quad\lim_p F_\mu(r_p) = F_\mu(x_u).
	\]
	By Portmanteau's theorem and monotonicity of cumulative distribution functions we have
	\[ F_\mu(l_p-) \leq \liminf_k F_{\mu^k}(l_p -)\le\liminf_k F_{\mu^k}(x^k_u-)\le\limsup_k F_{\mu^k}(x^k_u) \leq \limsup_k F_{\mu^k}(r_p) \leq F_\mu(r_p). \]
	By taking the limit $p\to+\infty$, we find
	\begin{equation}\label{eq:upper and lower bound Muk}
	F_\mu(x_u-)\leq \liminf_k F_{\mu^k}(x^k_u-) \leq \limsup_k F_{\mu^k}(x_u^k) \leq F_\mu(x_u).
	\end{equation}
	By \eqref{eq:jumps F23}, the interval $[F_{\mu^k}(x^k_u-),F_{\mu^k}(x^k_u)]$ contains $u$, and if $u\in\mathcal U_c$, then \eqref{eq:upper and lower bound Muk} implies that its length $\mu^k(\{x^k_u\})$ vanishes when $k$ goes to $+\infty$. Consequently, \eqref{eq:Wr Copula} and the Lebesgue differentiation theorem yield that for $du$-almost every $u\in\mathcal U_c$,
	\[ \int_{(0,1)} f(F_\nu^{-1}(v)) \, C_u^k(dv) \to g(u). \]
	Suppose now $u\in\mathcal U_d$ and define
	\[
	a_k = F_{\mu^k}(x_u^k-) \vee F_\mu(x_u-),\quad b_k = F_{\mu^k}(x_u^k) \wedge F_\mu(x_u).
	\]
	Note that on the interval $(a_k,b_k)$ the function $g$ is constant equal to $g(u)$, so \eqref{eq:Wr Copula} amounts to
	\begin{align*}
	\int_{(0,1)} f(F_\nu^{-1}(v)) \, C_u^k(dv) = \frac{1}{\mu^k(\{x_u^k\})} \left( \int_{b_k}^{F_{\mu^k}(x_u^k)} g(w) \, dw + \int_{a_k}^{b_k} g(u) \, dw + \int_{F_{\mu^k}(x_u^k-)}^{a_k} g(w) \, dw \right).
	\end{align*}
	According to \eqref{eq:upper and lower bound Muk},
	\begin{equation}\label{eq:convergence ak bk}
	a_k - F_{\mu^k}(x_u^k-)\to0\quad\text{and}\quad F_{\mu^k}(x_u^k) - b_k\to0,\quad k\to+\infty.
	\end{equation}
	Moreover, having \eqref{eq:jumps F23} in mind it is clear that
	\begin{gather}
		\label{eq:cases in prop 2.1}
		\begin{split}
	F_{\mu^k}(x^k_u-)<a_k \implies\mu^k(\{x^k_u\})\ge u-F_\mu(x_u-),\\
	\textrm{and}\quad b_k<F_{\mu^k}(x^k_u) \implies\mu^k(\{x^k_u\})\ge F_\mu(x_u)-u.
		\end{split}
	\end{gather}
Using the latter fact and the equality
\[
b_k-a_k=\mu_k(\{x^k_u\})-(F_{\mu_k}(x^k_u)-b_k)-(a_k-F_{\mu_k}(x^k_u-)),
\]
we get
\[
1-\frac{F_{\mu_k}(x^k_u)-b_k}{F_\mu(x_u)-u}-\frac{a_k-F_{\mu_k}(x^k_u-)}{u-F_\mu(x_u-)}\le\frac{b_k-a_k}{\mu_k(\{x^k_u\})}\le1.
\]
Hence by \eqref{eq:convergence ak bk} we have $\frac{b_k-a_k}{\mu_k(\{x^k_u\})}\to1$ as $k$ goes to $+\infty$, which implies that $\frac{1}{\mu^k(\{x^k_u\})}\int_{a_k}^{b_k}g(u)\,dw\to g(u)$ as $k\to+\infty$. Therefore, we just have to show that
\begin{equation}\label{eq:copula caseUd}
\frac{1}{\mu^k(\{x_u^k\})} \left( \int_{b_k}^{F_{\mu^k}(x_u^k)} g(w) \, dw + \int_{F_{\mu^k}(x_u^k-)}^{a_k} g(w) \, dw \right)\to0,\quad k\to+\infty.
\end{equation}
	Note that we can assume w.l.o.g.\ that for all $k\in\N$ either $F_{\mu^k}(x^k_u-) < a_k$ or $b_k < F_{\mu^k}(x^k_u)$.
	Let $d=(u-F_\mu(x_u-))\wedge(F_\mu(x_u)-u)$, which is positive since $u\in\mathcal U_d$. Then we have by \eqref{eq:cases in prop 2.1}
	\begin{align}\label{eq:copula caseUd 2}
    \begin{split}
      \frac{1}{\mu^k(\{x_u^k\})} \left| \int_{b_k}^{F_{\mu^k}(x_u^k)} g(w) \, dw + \int_{F_{\mu^k}(x_u^k-)}^{a_k} \right. & g(w) \, dw \Bigg| \\
      & \le \frac{1}{d} \left|\int_{b_k}^{F_{\mu^k}(x_u^k)} g(w) \, dw + \int_{F_{\mu^k}(x_u^k-)}^{a_k} g(w) \, dw \right|.
    \end{split}
	\end{align}
	By the inverse transform sampling and the facts that $f\in\Phi_r(\R)$ and $\nu\in\mathcal P_r(\R)$, we have $\int_0^1\vert g(w)\vert\,dw=\int_\R\vert f(y)\vert\,\nu(dy)<+\infty$. Then \eqref{eq:copula caseUd} is a direct consequence of \eqref{eq:copula caseUd 2}, \eqref{eq:convergence ak bk} and the dominated convergence theorem. Hence \eqref{eq:pointwise convergence} is proved for $du$-almost every $u\in(0,1)$.

	Next, we show uniform integrability of \eqref{eq:uniformly integrable}.
	We can estimate
	\begin{align*}
	\mathcal W_r^r((F_\nu^{-1})_\ast C_u, (F_\nu^{-1})_\ast C_u^k) \leq 2^{r-1} \left( \int_{(0,1)} |F_\nu^{-1}(v)|^r \, C_u(dv) + \int_{(0,1)} |F_\nu^{-1}(v)|^r \, C_u^k(dv) \right).
	\end{align*}
	Since by the inverse transform sampling we have
	\[ \int_{(0,1)}\int_{(0,1)} |F_\nu^{-1}(v)|^r \,C_u(dv) \, du = \int_\R \vert y \vert^r \, \nu(dy) < \infty, \]
	it is enough to show uniform integrability of $u \mapsto \int_{(0,1)} |F_\nu^{-1}(v)|^r \, C_u^k(dv)$, $k\in \N$.

	On the one hand, using the inverse transform sampling and $\nu\in\mathcal P_r(\R)$, we have
	\[
	\textcolor{blue}{\forall k\in\N,}\;\int_{(0,1)}\int_{(0,1)}\vert F_\nu^{-1}(v)\vert^r C^k_u(dv)\,du=\int_\R\vert y\vert^r\,\nu(dy)<+\infty.
	\]
	On the other hand, let $\epsilon > 0$ and $A$ be a measurable subset of $(0,1)$ such that $\lambda(A) < \epsilon$. We have
	\[
	\int_A \int_{(0,1)} |F_\nu^{-1}(v)|^r \, C_u^k(dv) \, du = \int_\R |y|^r \, (F_\nu^{-1})_\ast \tau^k(dy),
	\]
	where $\tau^k(dv) = \int_{u = 0}^1 \mathbbm{1}_A(du) \, C_u^k(dv) \, du$.
	Note that $\tau^k \leq \lambda$, $(F_\nu^{-1})_\ast \tau^k \leq \nu$ and  $(F_\nu^{-1})_\ast \tau^k(\R) =  \tau^k((0,1)) =  \lambda(A)$. Therefore,
	\begin{align*}
	\sup_{\substack{A \in \mathcal B((0,1)), \\ \lambda(A) \leq \epsilon}} \sup_k \int_A \int_{(0,1)} |F_\nu^{-1}(v)|^r \, C_u^k(dv) \, du \leq I_\epsilon^r(\nu),
	\end{align*}
	where $I^r_\varepsilon(\nu)$ is defined by \eqref{eq:def Iepsilonr}. By Lemma \ref{lem:uniform integrability}, the right-hand side converges to $0$ with $\epsilon \to 0$. This yields uniform integrability of \eqref{eq:uniformly integrable}, which completes the proof.
\end{proof}

As mentioned in Section \ref{sec:main results}, Proposition \ref{prop:adapted approximation dim1} generalises to Polish spaces. Unsurprisingly, the proof of Proposition \ref{prop:adapted approximation} requires radically different tools from its unidimensional equivalent. In particular, we need to 
recall the so-called Weak Optimal Transport (WOT) problem introduced by Gozlan, Roberto, Samson and Tetali \cite{GoRoSaTe17} and studied in \cite{GoRoSaSh18}. Let $C:X\times\mathcal P_r(Y)\to\R_+$ be nonnegative, continuous, strictly convex in the second argument and such that there exists a constant $K>0$ which satisfies
\begin{equation}\label{eq:condition cost function}
\forall(x,p)\in X\times\mathcal P_r(Y),\quad C(x,p)\le K\left(1+d_X^r(x,x_0)+\int_Yd_Y^r(y,y_0)\,p(dy)\right).
\end{equation}
Then the WOT problem consists in minimising
\begin{equation}\label{WOT}
\tag{WOT}
V_C(\mu,\nu):=\inf_{\pi\in\Pi(\mu,\nu)}\int_X C(x,\pi_x)\,\mu(dx).
\end{equation}
In view of the definition \eqref{eq:AW=W} of the adapted Wasserstein distance which involves measures on the extended space $X\times\mathcal P(Y)$, it is natural to consider an extension of \eqref{WOT} which also involves this space. Hence we also consider the extended problem
\begin{equation}\label{WOT'}
\tag{WOT'}
V'_C(\mu,\nu):=\inf_{P\in\Lambda(\mu,\nu)}\int_{X\times\mathcal P(Y)}C(x,p)\,P(dx,dp),
\end{equation}
where $\Lambda(\mu,\nu)$ is the set of couplings between $\mu$ and an arbitrary measure on $\mathcal P(Y)$ with mean $\nu$, that is
\begin{equation}\label{def:big lambda}
\Lambda(\mu,\nu)=\left\{P\in\mathcal P(X\times\mathcal P(Y))\mid\int_{(x',p)\in X\times\mathcal P(Y)}\delta_{x'}(dx)\,p(dy)\,P(dx',dp) \in \Pi(\mu,\nu)\right\}.
\end{equation}

\begin{remark} \label{rem:WOT results}
	We gather here useful results on weak transport problems which hold under the standing  assumptions on $C$:
	\begin{enumerate}[label = (\alph*)]
		\item \label{it:owt1} \textcolor{purple}{according to \cite[Theorem 1.2]{BaBePa18} and the paragraph following this theorem,} \eqref{WOT} admits a unique minimiser $\pi^*$; 
		\item \label{it:owt3}As a consequence of the necessary optimality condition  \cite[Theorem 2.2]{BaPa19}, 
		$J(\pi^*)$ is the only minimiser of \eqref{WOT'}. \textcolor{blue}{Indeed, if we assume the opposite then there is a minimizer $P^\ast \in \Lambda(\mu,\nu)$ of \eqref{WOT'} which does not lie in the image of $\Pi(\mu,\nu)$ under $J$. Hence, any measurable set $\mathcal A \subset X \times \mathcal P_r(Y)$ with $P^\ast(\mathcal A) = 1$ contains $(x,p), (x,q) \in \mathcal A$ with $p \neq q$. Due to strict convexity of $C$ in its second argument, we find
		\[	
			C\left( x, \frac{p + q}{2} \right) < \frac{1}{2} \left( C(x,p) + C(x,q) \right).
		\]
		Since $\mathcal A$ was an arbitrary set supporting $P^\ast$, the strict inequality above contradicts the necessary optimality condition in \cite[Theorem 2.2]{BaPa19}};
		\item \label{it:owt4} $V(\mu,\nu)=V'(\mu,\nu)$ \cite[Lemma 2.1]{BaBePa18};
		\item \label{it:owt5} Stability of \eqref{WOT} and \eqref{WOT'}: Let $\mu^k\in\mathcal P_r(X),\nu^k\in\mathcal P_r(Y)$, $k\in\N$ converge respectively to $\mu\in\mathcal P_r(X)$ and $\nu\in\mathcal P_r(Y)$ in $\mathcal W_r$. For $k\in\N$, let $\pi^k\in\Pi(\mu^k,\nu^k)$ be optimal for $V(\mu^k,\nu^k)$. Then $\pi^k$, resp.\ $J(\pi^k)$, converges to the unique minimiser $\pi^*$, resp.\ $J(\pi^*)$, in $\mathcal W_r$ \cite[Theorem 1.3 and Corollary 2.8]{BaPa19}. In particular, this shows that $\pi^k$ converges to $\pi^*$ even in $\mathcal{AW}_r$.
	\end{enumerate}
\end{remark}

\begin{proof}[Proof of Proposition \ref{prop:adapted approximation}]
\textcolor{blue}{Since $\nu\in\mathcal P_r(Y)$, we have that
	\begin{equation}\label{pixinPrae}
	 \int_X\int_Yd_Y^r(y,y_0)\,\pi_x(dy)\,\mu(dx)=\int_Yd_Y^r(y,y_0)\,\nu(dy)<+\infty,
	 \end{equation}
	 hence up to a modification on a $\mu$-null set, we can suppose w.l.o.g. that for all $x\in X$, $\pi_x\in\mathcal P_r(Y)$. Let $\varepsilon>0$ and $y_0\in Y$. Define for $R > 0$ the $\mathcal W_r$-open ball $B_R$ of radius $R^{1/r}$ and centre $\delta_{y_0}$ and the set
	 \[
	 A_R=\{x\in X\mid\pi_x\in B_R\}=\left\{x\in X \;\middle|\;\int_Yd_Y^r(y,y_0)\,\pi_x(dy)<R\right\}.
	 \]
	}
\textcolor{blue}{
	 By \eqref{pixinPrae} again, $\mu$ is concentrated on $\bigcup_{R > 0}A_R$ and we can choose $R$ large enough such that
	 \[
	 \mu(X\backslash A_R)<\epsilon.
	 \]
}
\textcolor{red}{Since $\mu$ is a probability measure on the Polish space $X$, it is a Radon measure. Moreover, $\mathcal P_r(Y)$ endowed with $\mathcal W_r$ is a separable metric space, hence it is second-countable. Therefore we can apply Lusin's theorem to the map $X\ni x\mapsto\pi_x\in\mathcal P_r(Y)$ in order to deduce the existence of 	
%
%
	a closed set $F\subset A_R$ such that
	\[
	\mu(X\backslash F)< \epsilon  \quad\text{and}\quad x\mapsto\pi_x\text{ restricted to $F$ is continuous}.
	\]
}
	Let $\widetilde{\mathcal M}_r(Y)$ be the linear space of all finite signed measures on $Y$, the positive and negative parts of which are contained in $\mathcal M_r(Y)$, equipped with the weak topology induced by $\Phi_r(Y)$. Since weak topologies are locally convex, an extension of Tietze's theorem \cite[Theorem 4.1]{Du51} yields the existence of a continuous map $x\mapsto\overline\pi_x$ defined on $X$ with values in $\widetilde{\mathcal M}_r(Y)$ such that $\overline\pi_x=\pi_x$ for all $x\in F$ and
	\[
	\{\overline\pi_x\mid x\in X\}\subset\operatorname{co}\{\pi_x\mid x\in F\}\subset B_R,
	\]
	where $\operatorname{co}$ denotes the convex hull.
	
	Next, we define a nonnegative, continuous, strictly convex in the second argument function which satisfies a condition of the form \eqref{eq:condition cost function} in order to use the results on weak transport problems detailed in Remark \ref{rem:WOT results}. Let $\{ g_k\mid k \in \N\}\subset \Phi_1(Y)$ be a family of $1$-Lipschitz continuous functions and absolutely bounded by $1$, which separates $\mathcal P(Y)$ (see \cite[Theorem 4.5.(a)]{EtKu}).
	We have for any pair $p,p' \in \mathcal P(Y)$, $p \neq p'$ that there is $l \in \N$ such that
	\begin{align}\label{eq:p neq q}
	\int_Y g_l(y) \, p(dy) \neq \int_Y g_l(y) \, p'(dy).
	\end{align}
	Define $C:X\times\mathcal P_r(Y)\to\R_+$ for all $(x,p)\in X\times\mathcal P_r(Y)$ by
	\[
	C(x,p):=\rho(\overline\pi_x,p)+\sum_{k \in \N} \frac{1}{2^k} \left\vert\int_Y g_k(y)\,\overline\pi_x(dy)-\int_Yg_k(y)\,p(dy)\right\vert^2,
	\]
	where $\rho:\mathcal P(Y)\times\mathcal P(Y)\to[0,1]$ is defined for all $p,p'\in\mathcal P(Y)$ by
	\[
	\rho(p,p')=\inf_{\chi\in\Pi(p,p')}\int_{Y\times Y}(d_Y(y,y')\wedge1)\,\chi(dy,dy').
	\]
	Since $\rho$ can be interpreted as a Wasserstein distance with respect to a bounded distance, it is immediate that it is a metric on $\mathcal P(Y)$ which induces the weak convergence topology.
	On the one hand, the map $(x,p)\mapsto\rho(\overline\pi_x,p)$ is continuous by continuity of $x\mapsto\overline\pi_x$.
	On the other hand, by Kantorovich and Rubinstein's duality theorem and Jensen's inequality, we have for all $(x,p),(x',p')\in X\times\mathcal P_r(Y)$
	\begin{align*}
	&\sum_{k \in \N} \frac{1}{2^k} \left| \left \vert \int_Y g_k(y) \, \overline \pi_x(dy) - \int_Y g_k(y) \, p(dy)\right \vert^2 -\left\vert\int_Y g_k(y) \, \overline\pi_{x'}(dy)-\int_Y g_k(y) \, p'(dy)\right\vert^2 \right| \\
	=&\sum_{k \in \N} \frac{1}{2^k} \left|\int_Y g_k(y) \, \overline \pi_x(dy) - \int_Y g_k(y) \, p(dy)+\int_Y g_k(y) \, \overline\pi_{x'}(dy)-\int_Y g_k(y) \, p'(dy) \right|\\
	&\phantom{\sum_{k\in\N}\ }\times\left|\int_Y g_k(y) \, \overline \pi_x(dy) - \int_Y g_k(y) \, p(dy)-\int_Y g_k(y) \, \overline\pi_{x'}(dy)+\int_Y g_k(y) \, p'(dy) \right|\\
	\le& \sum_{k \in \N} \frac{4}{2^k} \left( \left \vert \int_Y g_k(y) \, \overline\pi_x(dy)-\int_Yg_k(y)\,\overline\pi_{x'}(dy)\right\vert+\left\vert\int_Yg_k(y)\,p(dy)-\int_Yg_k(y)\,p'(dy)\right\vert\right) \\
	\le& 8\left(\mathcal W_1(\overline\pi_x,\overline\pi_{x'})+\mathcal W_1(p,p')\right)  \le 8\left(\mathcal W_r(\overline\pi_x,\overline\pi_{x'})+\mathcal W_r(p,p')\right),
	\end{align*}
	where the right-hand side vanishes when $(x',p')$ converges to $(x,p)$ by continuity of $x\mapsto\overline\pi_x$. We deduce that $C$ is continuous.
	
	Note that $\rho$ is convex in the second argument.
	Therefore, to obtain strict convexity of $C(x,\cdot)$ in the second argument, it is sufficient to verify that
	\[ F(p) = \sum_{k \in \N} \frac{1}{2^k}  \left \vert \int_Y g_k(y) \, p(dy)\right|^2 \]
	is strictly convex.
	Let $p,p' \in \mathcal P(Y)$, $p \neq p'$ and $l \in \N$ such that \eqref{eq:p neq q} holds.
	Hence, strict convexity of the square proves
	\[ \left| \alpha \int_Y g_l(y)\, p(dy) + (1 - \alpha) \int_Y g_l(y) \, p'(dy) \right|^2 < \alpha \left|  \int_Y g_l(y)\, p(dy) \right|^2 + (1 - \alpha) \left|  \int_Y g_l(y)\, p'(dy) \right|^2, \]
	which yields strict convexity of $F$ on $\mathcal P(Y)$.
	
	Moreover, we have for all $(x,p)\in X\times\mathcal P_r(Y)$, $C(x,p)\le1+8=9$, hence $C$ satisfies \eqref{eq:condition cost function}.
	Remember the definitions of $V_C$ and $V'_C$ given in \eqref{WOT} and \eqref{WOT'}.
	Since for all $x\in F$, $C(x,\pi_x)=C(x,\overline\pi_x)=0$, we have
	\[
	V_C(\mu,\nu) \leq \int_{X \backslash F} C(x,\pi_x)\, \mu(dx)<9\varepsilon.
	\]
	Let $\pi^{*,\varepsilon}\in\Pi(\mu,\nu)$ be optimal for $V_C(\mu,\nu)$. For $P,P'\in\mathcal P(X\times\mathcal P(Y))$, let
	\[
	\tilde\rho(P,P')=\inf_{\chi\in\Pi(P,P')}\int_{X\times\mathcal P(Y)\times X\times\mathcal P(Y)}\left(\left(d_X(x,x')+\rho(p,p')\right)\wedge1\right)\,\chi(dx,dp,dx',dp').
	\]
	Since $\mu(dx)\,\delta_{\pi_x}(dp)\,\delta_{x}(dx')\,\delta_{\pi^{*,\varepsilon}_{x'}}(dp')$ is a coupling between $J(\pi)$ and $J(\pi^{*,\varepsilon})$, we can estimate
	\begin{align*}
	\tilde\rho(J(\pi),J(\pi^{*,\varepsilon}))&\le\int_X\rho(\pi_x,\pi^{*,\varepsilon}_x)\,\mu(dx)\\
	&\le\int_F\rho(\pi_x,\pi^{*,\varepsilon}_x)\,\mu(dx)+\int_{X\setminus F}\int_Y(d_Y(y,y_0)\wedge1)\,(\pi_x+\pi^{*,\varepsilon}_x)(dy)\,\mu(dx)\\
	&\le V_C(\mu,\nu)+2\varepsilon<11\varepsilon.
	\end{align*}
	For $k\in\N$, let $\pi^{k,\varepsilon}\in\Pi(\mu^k,\nu^k)$ be optimal for $V_C(\mu^k,\nu^k)$.
	Then $J(\pi^{k,\varepsilon})$ is optimal for $V'_C(\mu^k,\nu^k)$ by Remark \ref{rem:WOT results} \ref{it:owt3}, and converges to $J(\pi^{*,\varepsilon})$ in $\mathcal W_r$ and therefore weakly by Remark \ref{rem:WOT results} \ref{it:owt5}. Then we get
	\begin{equation}\label{eq:limsup Awr pikepsilon}
	\limsup_{k\to+\infty}\tilde\rho(J(\pi^{k,\varepsilon}),J(\pi)) \\ \le \limsup_{k \to + \infty}\left(\tilde\rho(J(\pi^{k,\varepsilon}),J(\pi^{*,\varepsilon}))+\tilde\rho(J(\pi^{*,\varepsilon}),J(\pi))\right)\le11\varepsilon. 
	\end{equation}	
	So far $\epsilon > 0$ was arbitrary.
	Therefore, there exists a strictly increasing sequence $(k_N)_{N\in\N^*}$ of positive integers such that
	\[ \forall N \in\N^*,\quad\forall k \ge k_N, \quad \tilde\rho(J(\pi^{k,\textcolor{red}{1/N}}),J(\pi)) \le \textcolor{red}{\frac{12}{N}}.
	\]	
	For $k\in\N$, let $N_k = \max \{ N \in \N^* \mid k \geq k_N \}$, where the maximum of the empty set is defined as $1$.
	Since $(k_N)_{N\in\N^*}$ is strictly increasing, we find that $N_k \to +\infty$ as $k \to +\infty$.
	Then the sequence of couplings
	\[ \pi^k = \pi^{k,1/N_k}\in\Pi(\mu^k,\nu^k),~k\in\N \]
	is such that $\tilde\rho(J(\pi^k),J(\pi))$ vanishes as $k$ goes to $+\infty$, and therefore $J(\pi^k)$ converges weakly to $J(\pi)$. Moreover, since $\mathcal W_r$-convergence is equivalent to weak convergence coupled with convergence of the $r$-moments, we have that the $r$-moments of $\mu^k$ and $\nu^k$ respectively converge to the $r$-moments of $\mu$ and $\nu$, which implies
	\begin{align*}
	&\int_{X\times\mathcal P(Y)}\left(\textcolor{red}{d_X^r(x,x_0)}+\mathcal W_r^r(p,\delta_{y_0})\right)\,J(\pi^k)(dx,dp)=\int_X\left(\textcolor{red}{d_X^r(x,x_0)}+\mathcal W_r^r(\pi^k_x,\delta_{y_0})\right)\,\mu^k(dx)\\&=\int_X\textcolor{red}{d_X^r(x,x_0)}\,\mu^k(dx)+\int_Yd_Y^r(y,y_0)\,\nu^k(dy)\underset{k\to+\infty}{\longrightarrow}\int_X\textcolor{red}{d_X^r(x,x_0)}\,\mu(dx)+\int_Yd_Y^r(y,y_0)\,\nu(dy)\\
	&\phantom{=\int_Xd_X^r(x,x_0)\,\mu^k(dx)+\int_Yd_Y^r(y,y_0)\,\nu^k(dy)}=\int_{X\times\mathcal P(Y)}\left(\textcolor{red}{d_X^r(x,x_0)}+\mathcal W_r^r(p,\delta_{y_0})\right)\,J(\pi)(dx,dp).
	\end{align*}	
	We deduce that $J(\pi^k)$ converges to $J(\pi)$ in $\mathcal{W}_r$ as $k\to+\infty$. According to \eqref{eq:AW=W}, $\pi^{k,\varepsilon}$ converges to $\pi^{*,\varepsilon}$ in $\mathcal{AW}_r$, which concludes the proof.
\end{proof}
	 In the proof of Theorem~\ref{thm:1} we need to be able to confine approximative sequences of couplings to certain sets. The next result provides all necessary tools for this.

	 \begin{lemma}\label{lem:subprob convergence}
     Let $\mu,\mu^k\in\mathcal M_r(X)$, $\nu,\nu^k\in\mathcal M_r(Y)$, $k\in\N$ all with equal masses and $\pi^k \in \Pi(\mu^k,\nu^k)$, $k\in\N$, converge to $\pi \in \Pi(\mu,\nu)$ in $\mathcal{AW}_r$. \textcolor{red}{Let also $A\subset X$ be measurable and $B\supset A$ be open.}
     \begin{enumerate}[label = (\roman*)]
       \item \label{it:subprob convergence1}
        
        There are $\tilde \mu^k \leq \mu^k|_B$ and $\epsilon_k\ge0$, $k\in\N$ such that \textcolor{red}{$\tilde \mu^k(B)=(1-\epsilon_k)\mu(A)$} and $\tilde \pi^k:=\tilde \mu^k\times  \pi^k_x$ satisfies
       \[
        \mathcal{AW}_r(\tilde\pi^k,(1-\epsilon_k)\pi|_{A\times Y})+\epsilon_k\underset{k\to+\infty}{\longrightarrow}0.
       \]

       \item \label{it:subprog convergence2}
        Let $C\subset \textcolor{red}{Y}$ be an open set on which $\nu$ is concentrated.
        There are $\hat \mu^k \leq \tilde \mu^k$, $\hat\nu^k\le\nu^k$, $\hat\pi^k=\hat\mu^k\times\hat\pi^k_x\in\Pi(\hat\mu^k,\hat\nu^k)$ concentrated on $B\times C$ and $\epsilon'_k\ge 0$, $k\in\N$ such that
        \[
          \mathcal{AW}_r^r(\hat\pi^k,(1-\epsilon'_k)\pi\vert_{A\times Y})+\int_X\mathcal W_r^r(\hat\pi^k_x,\pi^k_x)\,\hat\mu^k(dx)+\varepsilon'_k\underset{k\to+\infty}{\longrightarrow}0.
        \]
     \end{enumerate}
	\end{lemma}

	\begin{proof}
		\textcolor{red}{To give the reader some guidance we first give an informal description of the strategy of the proof: In order to find $(\tilde \pi^k)_{k \in \N}$ and $(\hat \pi^k)_{k \in \N}$, we first pick, for $k\in\N$, optimizers $\chi^k \in \Pi(\mu^k,\mu)$ for $\mathcal{AW}_r(\pi^k,\pi)$.
	Denote by  $\tilde \pi^k$ the composition of the first marginal of $\chi^k|_{B \times A}$ with the kernel $(\pi^k_x)_{x \in X}$.
		By approximation arguments we will then  deduce that $\tilde \pi^k$ has the desired properties.
		In the last step, we adequately modify $\tilde \pi^k$ to a coupling $\hat \pi^k$ with second marginal concentrated on $C$.}

    Both assertions are trivial if $\mu(A) = 0$ \textcolor{red}{(and also when $A=X$)}.
    So assume that $\mu(A) > 0$.
    \begin{enumerate}[label = (\roman*), wide]
    \item Let $\chi^k \in \Pi(\mu^k,\mu)$ be optimal for $\mathcal{AW}_r(\pi^k,\pi)$ and \textcolor{red}{$\tilde \mu^k$ be the first marginal of $\chi^k|_{B\times A}$, $k\in\N$. We set $\tilde \pi^k=\tilde \mu^k\times\pi^k_x$ and \begin{equation}\label{eq:def epsilonk}
      \varepsilon_k=1-\frac{\chi^k(B\times A)}{\chi^k(X\times A)}=1-\frac{\tilde\mu^k(X)}{\mu(A)}\cdot
      \end{equation}
Let us prove that $\varepsilon_k$ goes to $0$ as $k\rightarrow\infty$ before checking that the same holds for $\mathcal{AW}_r(\tilde\pi^k,(1-\epsilon_k)\pi|_{A\times Y})$. \\ Let $\chi = (\operatorname{id},\operatorname{id})_\ast \mu$.} Since $\chi^k(dx_1,dx_2)\, \delta_{(x_2,x_2)}(dx_3,dx_4)$ defines a coupling in $\Pi(\chi^k,\chi)$, we find
      \begin{align*}
        \mathcal{W}_r^r (\chi^k,\chi) & \le \int_{X^4} (d_X(x_1,x_3)^r + d_X(x_2,x_4)^r) \, \chi^k(dx_1,dx_2) \, \delta_{(x_2,x_2)}(dx_3,dx_4) \\
        & = \int_{X\times X} d_X(x_1,x_2)^r \, \chi^k(dx_1,dx_2) \le \mathcal{AW}_r^r(\pi^k,\pi) \to 0, \quad k\to+\infty.
      \end{align*}
      Further, let $P:\mathcal P_r(X\times X)\to\mathcal P(X\times X)$ be the homeomorphism given by
  		\begin{align*}
  		P(\eta)(dx_1,dx_2)=\frac{(1+d_X(x_1,x_0)^r+d_X(x_2,x_0)^r)\,\eta(dx_1,dx_2)}{\int_{X\times X}(1+d_X(x'_1,x_0)^r+d_X(x'_2,x_0)^r)\,\eta(dx'_1,dx'_2)},
  		\end{align*}
      for $\eta\in\mathcal P_r(X\times X)$.
      Recall \eqref{eq:convergencePr}, then it is easy to deduce that $P(\eta') \to P(\eta)$ weakly \textcolor{red}{if and only if} $\eta' \to \eta$ in $\mathcal W_r$.
      In particular, we find that $P(\chi^k) \to P(\chi)$ weakly as $k$ goes to $+\infty$.
      Let $f\in\Phi_r(X\times X)$ and
      \[
        \varphi \colon X\times X \colon(x_1,x_2) \mapsto \frac{\mathbbm 1_{X\times A}(x_1,x_2) f(x_1,x_2)}{1+d_X(x_1,x_0)^r+d_X(x_2,x_0)^r}.
      \]
      Then $\varphi$ is a bounded measurable map which is continuous w.r.t.\ the first coordinate.
      As a consequence of \cite[Lemma 2.1]{La18}, we find
      \[
        \int_{X\times X} \varphi(x_1,x_2) \, P(\chi^k)(dx_1,dx_2) \to \int_{X\times X} \varphi(x_1,x_2) \, P(\chi)(dx_1,dx_2),\quad k\to+\infty,
      \]
      which amounts to
      \[
        \int_{X\times X} f(x_1,x_2) \, \chi^k \vert_{X\times A} (dx_1,dx_2) \to \int_{X\times X} f(x_1,x_2) \, \chi \vert_{X\times A} (dx_1,dx_2),\quad k\to+\infty.
      \]
      Therefore \eqref{eq:convergencePr} yields $\mathcal W_r$-convergence of $\chi^k|_{X\times A}$ to $\chi|_{X\times A}$.
      By Portmanteau's theorem, we have
      \[
      \chi^k(B\times A)\le\textcolor{red}{\chi^k(X\times A)=}\mu(A)=\chi\vert_{X\times A}(B\times B)\le\liminf_{k\to+\infty}\chi^k\vert_{X\times A}(B\times B)=\liminf_{k\to+\infty}\chi^k(B\times A),
      \]
\textcolor{red}{By the first equality in \eqref{eq:def epsilonk}, we deduce that $\epsilon_k$, $k\in\N$ is a null sequence of nonnegative real numbers.} We \textcolor{red}{now} want to show that
      \begin{align}\label{eq:subprob convergence3}
          \mathcal{AW}_r(\tilde \mu^k \times \pi^k_x, (1- \epsilon_k) \mu\vert_A \times \pi_x) \to 0.
      \end{align}
      On the one hand, \textcolor{red}{denoting by $\bar\mu^k$ the second marginal of $\chi^k|_{B\times A}$, we have that}
      \begin{align} \label{eq:subprob convergence1}
        \begin{split}
        \mathcal{AW}_r^r(\tilde\mu^k \times \pi^k_x, \bar \mu^k \times \pi_x) & \le \int_{X\times X}\left(d_X^r(x,x') + \mathcal W_r^r (\pi^k_x,\pi_{x'})\right) \chi^k \vert_{B\times A} (dx,dx')\\
        & \le \int_{X\times X} \left(d_X^r(x,x') + \mathcal W_r^r (\pi^k_x,\pi_{x'})\right) \chi^k (dx,dx')\\
        & = \mathcal{AW}_r^r (\pi^k, \pi) \to 0,\quad k\to+\infty.
      \end{split}
  		\end{align}
      On the other hand, let
      \[\check \mu^k = (1 - \epsilon_k) \mu\vert_A,\quad\zeta^k = \check \mu^k \wedge \bar \mu^k\quad\text{and}\quad \alpha_k = \bar \mu^k(X) - \zeta^k(X) = \check \mu^k(X) - \zeta^k(X).
      \]
      Let $\overline\chi^k\in\Pi(\bar\mu^k-\zeta^k,\check\mu^k-\zeta^k)$ be optimal for $\mathcal{AW}_r^r((\bar\mu^k-\zeta^k)\times\pi_x,(\check\mu^k-\zeta^k)\times\pi_x)$. Since $((\operatorname{id},\operatorname{id})_\ast\zeta^k+\overline\chi^k)$ is a coupling between $\bar\mu^k$ and $\check\mu^k$, we find
      \begin{align*}
          \mathcal{AW}_r(\bar\mu^k \times \pi_x, \check \mu^k \times \pi_x)&\le\int_X\left(d_X^r(x,x')+\mathcal W_r^r(\pi_x,\pi_{x'})\right)\,\overline\chi^k(dx,dx') \\
          &=\mathcal{AW}_r^r((\bar\mu^k-\zeta^k)\times\pi_x,(\check\mu^k-\zeta^k)\times\pi_x)\\
          &\leq \mathcal{AW}_r((\bar \mu^k - \zeta^k) \times \pi_x, \alpha_k \delta_{(x_0,y_0)}) + \mathcal{AW}_r((\check \mu^k - \zeta^k) \times \pi_x, \alpha_k \delta_{(x_0,y_0)}).
      \end{align*}
      In the next estimates we use \eqref{eq:def Iepsilonr}.
      Note that the first marginal of $(\bar \mu^k - \zeta^k) \times \pi_x$ is dominated by $\mu$ whereas its second marginal is dominated by $\nu$.
      Thus, denoting $\tau^k(dy)=\int_X\pi_x(dy)\,(\bar \mu^k - \zeta^k)(dx)$, we find
      \begin{align*}
        \mathcal{AW}_r^r((\bar \mu^k - \zeta^k) \times \pi_x, \alpha_k \delta_{(x_0,y_0)}) &= \int_X\left(d_X^r(x,x_0) +\mathcal W_r^r(\pi_x,\delta_{y_0})\right)\,(\bar \mu^k - \zeta^k)(dx) \\
        &=\int_X d_X^r(x,x_0)\,(\bar \mu^k - \zeta^k)(dx)+\int_Yd_Y^r(y,y_0)\,\tau^k(dy)\\
        &\leq I_{\alpha_k}^r(\mu) + I_{\alpha_k}^r(\nu).
      \end{align*}
     Similarly, we find
      \[
        \mathcal{AW}_r^r((\check \mu^k - \zeta^k) \times \pi_x, \alpha_k \delta_{(x_0,y_0)}) \leq I_{\alpha_k}^r(\mu) + I_{\alpha_k}^r(\nu).
      \]
      If we can show that $\alpha_k$ vanishes for $k\to+\infty$, then we find by Lemma \ref{lem:uniform integrability} \ref{it:IepsilonrVanishes} that
      \begin{align}\label{eq:subprob convergence2}
        \mathcal{AW}_r(\bar\mu^k \times \pi_x, \check \mu^k \times \pi_x)\underset{k\to+\infty}{\longrightarrow}0,
      \end{align}
      and the triangle inequality together with \eqref{eq:subprob convergence1} and \eqref{eq:subprob convergence2} yield the assertion, \eqref{eq:subprob convergence3}.

      Since $\check\mu^k,\bar\mu^k\le\mu\vert_A$, the  densities of $\check\mu^k$ and $\bar\mu^k$ with respect to $\mu\vert_A$ satisfy $\frac{d\check\mu^k}{d\mu\vert_A},\frac{d\bar\mu^k}{d\mu\vert_A} \leq 1$. Then we conclude by
      \begin{align*}
       \alpha_k = \bar\mu^k(X) - \zeta^k(X) &= \int_A\left(\frac{d\bar\mu^k}{d\mu\vert_A}(x)-\frac{d\check\mu^k}{d\mu\vert_A}(x)\right)^+\,\mu(dx)\le\int_A\left(1-\frac{d\check\mu^k}{d\mu\vert_A}(x)\right)\,\mu(dx)\\
       &=\mu(A)-\check\mu^k(A)=\varepsilon_k\mu(A)\underset{k\to+\infty}{\longrightarrow}0.
      \end{align*}

    \item 
      \textcolor{red}{Let $\tilde \nu^k$ and $\tilde \nu$ denote the second marginals of $\tilde \mu^k \times \pi^k_x$ and $\mu\vert_A \times \pi_x$ respectively. Since $\mu\vert_A\times\pi_x\le\mu\times\pi_x$ with the second marginal $\nu$ of the right-hand side concentrated on $C$, $\tilde\nu$ is concentrated on $C$ and $\tilde\nu(C)=\mu(A)$. In a similar way, since $\tilde\mu^k\times\pi^k_x\le\mu^k\times\pi^k_x$, we have $\tilde\nu^k\le\nu^k$. In order to modify $\tilde \mu^k \times \pi^k_x$ into a coupling with second marginal concentrated on $C$, we consider $\tilde \mu^k \times \mathring\pi^{k}_x$ with $\mathring\pi^{k}_x(dz) = \int_{Y} \mathring\chi^{k}_y(dz)\,\pi^k_x(dy)$ where the coupling $\mathring\chi^k\in\Pi(\tilde \nu^k,(1-\epsilon_k)\tilde \nu)$ is $\mathcal W_r$-optimal. To enable comparison of the second marginal with $\nu^k$ as in the statement, we take advantage of the inequality $\tilde\nu^k\le\nu^k$ and introduce $\tilde \mu^k \times \hat \pi^{k}_x$ with $\hat \pi^{k}_x(dt) = \int_Y \hat\chi^{k}_z(dt)\,\mathring\pi^{k}_x(dz)$ where the coupling $\hat\chi^k\in\Pi((1-\epsilon_k)\tilde \nu,(1-\epsilon_k)\frac{\tilde \nu(C)}{\tilde \nu^k(C)}\tilde \nu^k|_{C})$ is $\mathcal W_r$-optimal. The second marginal of $\textcolor{purple}{\hat \pi^k} = \tilde \mu^k \times \hat \pi^{k}_x$ is $(1-\epsilon_k)\frac{\tilde \nu(C)}{\tilde \nu^k(C)}\tilde \nu^k|_{C}$. By the equality $\tilde\nu(C)=\mu(A)$ and \eqref{eq:def epsilonk} for the equality then the definition of $\tilde\nu^k$ for the inequality, one has $$(1-\epsilon_k)\frac{\tilde \nu(C)}{\tilde \nu^k(C)}=\frac{\tilde\mu^k(X)}{\tilde \nu^k(C)}\ge 1.$$
Setting $\hat \mu^k =  \frac{\tilde \nu^k(C)}{\tilde \mu^k(X)}\tilde \mu^k \leq \tilde \mu^k$ then ensures that the second marginal $\hat\nu^k=
\tilde \nu^k|_{C}$ of $\hat \mu^k\times\hat\pi^k_x$ is both concentrated on $C$ and not greater than $\nu^k$. Moreover, $\hat\nu^k(C)=\hat\nu^k(Y)=\hat \mu^k(X)\le \tilde \mu^k(X)$ with the right-hand side not greater than $\mu(A)$ by \eqref{eq:def epsilonk}. Hence 
   \begin{equation}\label{eq:def epsilonkprime}
      \epsilon'_k := 1-\frac{\tilde \nu^k(C)}{\mu(A)}\in[0,1]\cdot
      \end{equation}  }
 Then it remains to show that
      \begin{equation}\label{eq:subprob convergence4}
      \mathcal{AW}_r(\hat\pi^k,(1-\epsilon'_k)\pi\vert_{A\times Y})+\int_X\mathcal W_r^r(\hat\pi^k_x,\pi^k_x)\,\hat\mu^k(dx)+\epsilon'_k\underset{k\to+\infty}{\longrightarrow}0.
      \end{equation}
  \end{enumerate}
Since we have
\begin{gather*}
\pi^k_x(dy)\,\mathring\chi^k_y(dz)\in\Pi(\pi^k_x,\mathring\pi^k_x),\quad\mathring\pi^k_x(dz)\,\hat\chi^k_z(dt)\in\Pi(\mathring\pi^k_x,\hat\pi^k_x),\\
\int_{x\in X}\hat\mu^k(dx)\,\pi^k_x(dy)\,\mathring\chi^k_y(dz)=\frac{\tilde\nu^k(C)}{\tilde\mu^k(X)}\mathring\chi^k(dy,dz),\\
\int_{x\in X}\hat\mu^k(dx)\,\mathring\pi^k_x(dz)\,\hat\chi^k_z(dt)=\textcolor{red}{\frac{\tilde\nu^k(C)}{\tilde\mu^k(X)}}\hat \chi^k(dz,dt),
\end{gather*}
we find plugging the expressions \eqref{eq:def epsilonk} and \eqref{eq:def epsilonkprime} that
\begin{align*}
&\mathcal{AW}_r^r(\hat\mu^k\times\pi^k_x,\hat\mu^k\times\hat\pi^k_x)\\
&\le\int_X\mathcal W_r^r(\pi^k_x,\hat\pi^k_x)\,\hat\mu^k(dx)&\\
&\le2^{r-1}\int_X\left(\mathcal W_r^r(\pi^k_x,\mathring\pi^k_x)+\mathcal W_r^r(\mathring\pi^k_x,\hat\pi^k_x)\right)\,\hat\mu^k(dx)\\
&\le2^{r-1}\int_X\left(\int_{Y\times Y}d_Y^r(y,z)\,\pi^k_x(dy)\,\mathring\chi^k_y(dz)+\int_{Y\times Y}d_Y^r(z,t)\,\mathring\pi^k_x(dz)\,\hat\chi^k_z(dt)\right)\,\hat\mu^k(dx)\\
&=2^{r-1}\left(\frac{\tilde\nu^k(C)}{\tilde\mu^k(X)}\int_{Y\times Y}d_Y^r(y,z)\,\mathring\chi^k(dy,dz)+\textcolor{red}{\frac{\tilde\nu^k(C)}{\tilde\mu^k(X)}}\int_{Y\times Y}d_Y^r(z,t)\,\hat \chi^k(dz,dt)\right)\\
&=2^{r-1}\left(\frac{1-\varepsilon'_k}{1-\varepsilon_k}\mathcal W_r^r(\tilde\nu^k,(1-\varepsilon_k)\tilde\nu)+\frac{1}{\mu(A)}\mathcal W_r^r(\tilde\nu^k(C)\tilde\nu,\tilde\nu(C)\tilde\nu^k\vert_C)\right).
\end{align*}
      To see convergence to 0, note that since $\mathcal{AW}_r$ dominates $\mathcal W_r$, we find by continuity of the projection on the second marginal that \eqref{eq:subprob convergence3} implies
      \[ \mathcal W_r(\tilde \nu^k, (1 - \epsilon_k) \tilde \nu) \to 0,\quad k\to+\infty. \]

      Using Portmanteau's theorem and the fact that $(1-\varepsilon_k)\to1$ as $k$ goes to $+\infty$, we have for all nonnegative function $f\in\Phi_r(Y)$
      \[
      \limsup_{k\to+\infty}\tilde\nu^k(\mathbbm1_Cf)\le\limsup_{k\to+\infty}\tilde\nu^k(f)=\tilde\nu(f)=\tilde\nu(\mathbbm1_Cf)\le\liminf_{k\to+\infty}\tilde\nu^k(\mathbbm1_Cf),
      \]
      hence
      \begin{equation}\label{eq:convergence nutildek}
      \tilde\nu^k\vert_C(f)\to\tilde\nu(f),\quad k\to+\infty.
      \end{equation}
      Moreover, \eqref{eq:convergence nutildek} applied with $f=1$ yields $\tilde\nu^k(C)\to\tilde\nu(C)=\mu(A)$ as $k$ goes to $+\infty$, hence $\varepsilon'_k$ vanishes as $k$ goes to $+\infty$ and
      \[\mathcal W_r(\tilde \nu^k(C) \tilde\nu, \tilde \nu(C) \tilde \nu^k\vert_C) \to 0,\quad k\to+\infty. \]
      We deduce that
      \[
      \mathcal{AW}_r^r(\hat\mu^k\times\pi^k_x,\hat\mu^k\times\hat\pi^k_x)\le\int_X\mathcal W_r^r(\pi^k_x,\hat\pi^k_x)\,\hat\mu^k(dx)\to 0,\quad k\to+\infty.
      \]
      On the other hand, \textcolor{red}{by the definition of $\hat\mu^k$ as $\frac{\tilde\nu^k(C)}{\tilde\mu^k(X)}\tilde\mu^k$, \eqref{eq:def epsilonk} and \eqref{eq:def epsilonkprime} we have $\hat\mu^k=\frac{1-\varepsilon_k'}{1-\varepsilon_k}\tilde\mu^k$}, hence
      \[
      \mathcal{AW}_r(\hat\mu^k\times\pi^k_x,(1-\varepsilon'_k)\mu\vert_A\times\pi_x)=\frac{1-\varepsilon'_k}{1-\varepsilon_k}\mathcal{AW}_r(\tilde\mu^k\times\pi^k_x,(1-\varepsilon_k)\mu\vert_A\times\pi_x),
      \]
      where the right-hand side vanishes as $k$ goes to $+\infty$ by the first part. Then \eqref{eq:subprob convergence4} follows by triangle inequality and the latter convergences, which completes the proof.
	\end{proof}
	 The addition of measures is continuous with respect to the weak and Wasserstein topology. More precisely, we have the estimate
	 \[
	 \mathcal W_r^r(\mu+\mu',\nu+\nu')\le\mathcal W_r^r(\mu,\nu)+\mathcal W_r^r(\mu',\nu')
	 \]
	 for all measures $\mu,\mu',\nu,\nu'\in\mathcal P_r(X)$ such that $\mu$ and $\nu$, resp.\ $\mu'$ and $\nu'$ have equal masses.

   When considering the adapted weak topology, the next example disproves a comparable statement.
	 \begin{example}
	 	Let $X = Y = \R$, and $\pi^k = \delta_{\left(\frac{1}{k},1\right)}$, $\chi^k = \delta_{\left(-\frac{1}{k},-1\right)}$, $k\in\N$. Then both sequences are convergent in $\mathcal{AW}_1$, but
	 	$$\mathcal{AW}_1(\pi^k + \chi^k, \delta_{(0,1)} + \delta_{(0,-1)}) = \frac{2}{k} + 2$$
	 	does not vanish.
	 \end{example}
 However, we show in the next \textcolor{red}{lemma} that the addition of measures with respect to the adapted weak topology can still be  considered to be continuous in a certain sense if one of the limits has mass significantly smaller than the other.

\begin{lemma}\label{lem:inequality addition AWr}
  Let $\hat\mu,\hat\mu^k,\hat\nu,\hat\nu^k\in\mathcal M_r(Y)$, $k\in\N$ be with equal masses and $\tilde\mu,\tilde\mu^k,\tilde\nu,\tilde\nu^k\in\mathcal M_r(Y)$, $k\in\N$ be with equal masses smaller than $\varepsilon$.
  Let $\hat\pi^k\in\Pi(\hat\mu^k,\hat\nu^k),\tilde\pi^k\in\Pi(\tilde\mu^k,\tilde\nu^k)$, $k\in\N$, $\hat\pi\in\Pi(\hat\mu,\hat\nu)$ and $\tilde\pi\in\Pi(\tilde\mu,\tilde\nu)$.
  Let  $\mu=\hat\mu+\tilde\mu$ and $\nu=\hat\nu+\tilde\nu$. Then
	\begin{enumerate}[label = (\alph*)]
		\item\label{it:addition Awr} We have for all $k\in\N$
		\begin{align}\label{ineq:addition Awr}
		\begin{split}
		&\mathcal{AW}_r^r(\hat\pi^k+\tilde\pi^k,\hat\pi+\tilde\pi)\\
		&\phantom{\mathcal{AW}}\le\mathcal{AW}_r^r(\hat\pi^k,\hat\pi)+2^{r-1}\left(I_\varepsilon^r(\tilde\mu)+I_\varepsilon^r(\tilde\mu^k)+I_\varepsilon^r(\tilde\nu)+I_\varepsilon^r(\tilde\nu^k) + 2I_\varepsilon^r(\hat\nu) + 2I_\varepsilon^r(\hat\nu^k)\right)\\
		&\phantom{\mathcal{AW}}\le\mathcal{AW}_r^r(\hat\pi^k,\hat\pi)+(2^{r-1})^2\left(\mathcal W_r^r(\tilde\mu^k,\tilde\mu)+\mathcal W_r^r(\tilde\nu^k,\tilde\nu)+2\mathcal W_r^r(\hat\nu^k,\hat\nu)\right)\\
		&\phantom{\mathcal{AW}\le\ }+2^{r-1}(1+2^{r-1})I_\varepsilon^r(\mu)+3\cdot 2^{r-1}(1+2^{r-1})I_\varepsilon^r(\nu),
		\end{split}
		\end{align}
		where $I_\varepsilon^r(\cdot)$ is defined by \eqref{eq:def Iepsilonr}.

		\item\label{it:addition AWr limsup} If $(\hat\pi^k)_{k\in\N}$ converges to $\hat\pi$ in $\mathcal{AW}_r$ and $(\mu^k = \hat\mu^k+\tilde\mu^k)_{k\in\N}$, resp.\ $(\nu^k = \hat\nu^k+\tilde\nu^k)_{k\in\N}$, converges to $\mu$, resp.\ $\nu$, in $\mathcal W_r$, then
		\begin{equation}\label{eq:addition AWr limsup}
		\limsup_{k\to+\infty}\mathcal{AW}_r^r(\hat\pi^k+\tilde\pi^k,\hat\pi+\tilde\pi)\le C(I^r_\varepsilon(\mu)+I^r_\varepsilon(\nu)),
		\end{equation}
		where $C>0$ depends only on $r$.
	\end{enumerate}
\end{lemma}
\begin{proof}
  The second inequality of \eqref{ineq:addition Awr} is easily deduced from the first one, \eqref{ineq: Iepsilonr} and the fact that $I_\varepsilon^r(\tilde\mu)\le I_\varepsilon^r(\mu)$, $I_\varepsilon^r(\tilde\nu)\le I_\varepsilon^r(\nu)$ and $I_\varepsilon^r(\hat\nu)\le I_\varepsilon^r(\nu)$.

  To see \ref{it:addition AWr limsup}, assume for a moment that the first inequality of \eqref{ineq:addition Awr} holds true and suppose
  \[ \hat\pi^k \to \hat\pi\; \text{in }\mathcal{AW}_r,\quad \mu^k = \hat\mu^k+\tilde\mu^k \to \mu \quad \text{and}\quad \nu^k = \hat\nu^k+\tilde\nu^k\to \nu \; \text{in }\mathcal W_r\]
  as $k\to+\infty$.  Using Lemma \ref{lem:uniform integrability} \ref{it:croisieps} and then \ref{it:Iespilonr}, we obtain
	\begin{align*}
	\limsup_{k\to+\infty}\mathcal{AW}_r^r(\hat\pi^k+\tilde\pi^k,\hat\pi+\tilde\pi)&\le C'\limsup_{k\to+\infty}\left(I^r_\varepsilon(\mu^k)+I_\varepsilon^r(\nu^k)+I^r_\varepsilon(\mu)+I^r_\varepsilon(\nu)\right)\\
	&\le C\limsup_{k\to+\infty}\left(\mathcal W_r^r(\mu^k,\mu)+\mathcal W_r^r(\nu^k,\nu)+I_\varepsilon^r(\mu)+I_\varepsilon^r(\nu)\right)\\
	&=C(I_\varepsilon^r(\mu)+I_\varepsilon^r(\nu)),
	\end{align*}
	where $C,C'>0$ depend only on $r$. Hence \ref{it:addition AWr limsup} is proved.

	To conclude the proof, it remains to show the first inequality in \eqref{ineq:addition Awr}.
  Let $\hat\rho^k\in\Pi(\hat\mu^k,\hat\mu)$ be optimal for $\mathcal{AW}_r(\hat\pi^k,\hat\pi)$ and $\tilde\rho^k\in\Pi(\tilde\mu^k,\tilde\mu)$ be arbitrary.
  We write $\rho^k = \hat \rho^k + \tilde \rho^k$.
  Then
	\begin{equation}\label{eq:lemsum 1}
	\mathcal{AW}_r^r(\hat\pi^k+\tilde\pi^k,\hat\pi+\tilde\pi) \le\int_{X\times X}\left(d_X^r(x,x')+\mathcal W_r^r((\hat\pi^k+\tilde\pi^k)_x,(\hat\pi+\tilde\pi)_{x'})\right)\,\rho^k(dx,dx').
\end{equation}
	Let $\hat p=\frac{d\hat\mu}{d\mu}$ and $\hat p^k=\frac{d\hat\mu^k}{d\mu^k}$.
  Notice that $\hat p$ and $\hat p^k$ take values in $[0,1]$.
  The identities
	\begin{gather*}
    (\hat\pi+\tilde\pi)(dx,dx')=\mu(dx)\,\Big(\hat p(x)\,\hat\pi_x(dx')+(1-\hat p(x))\,\tilde\pi_x(dx')\Big),\\
    (\hat\pi^k+\tilde\pi^k)(dx,dx')=\mu^k(dx)\,\Big(\hat p^k(x)\,\hat\pi^k_x(dx')+(1-\hat p^k(x))\,\tilde\pi^k_x(dx')\Big),
  \end{gather*}
	provide representations for the disintegrations of $(\hat \pi + \tilde \pi)$ and $(\hat \pi^k + \tilde \pi^k)$ respectively for $\mu(dx)$- and $\mu^k(dx)$-almost every $x$:
  \[
    (\hat\pi+\tilde\pi)_x=\hat p(x)\,\hat\pi_x+(1-\hat p(x))\tilde\pi_x,\quad (\hat\pi^k+\tilde\pi^k)_x=\hat p^k(x)\,\hat\pi^k_x+(1-\hat p^k(x))\,\tilde\pi^k_x.
  \]
  Thus, we have when letting $\alpha^k_+(x,x') = (\hat p^k(x) - \hat p(x'))^+$, $\alpha^k_-(x,x') = (\hat p^k(x) - \hat p(x'))^-$ and $\beta^k(x,x') = \hat p^k(x) \wedge \hat p(x')$ that
	\begin{align}\label{eq:lemsum 2}
	\begin{split}
	&\mathcal W_r^r((\hat\pi^k+\tilde\pi^k)_x,(\hat\pi+\tilde\pi)_{x'})\\
	&\phantom{\mathcal W_r^r}\le\mathcal W_r^r(\beta^k(x,x')\,\hat\pi^k_x,\beta^k(x,x')\,\hat\pi_{x'}) \\
  &\phantom{\mathcal W_r^r\le} + \mathcal W_r^r\Big( \alpha^k_+(x,x')\,\hat\pi^k_x+(1-\hat p^k(x))\,\tilde\pi^k_x,\alpha^k_-(x,x')\,\hat\pi_{x'}+(1-\hat p(x'))\,\tilde\pi_{x'}\Big) \\
  &\phantom{\mathcal W_r^r} \le \beta^k(x,x')\mathcal W_r^r(\hat \pi^k_x,\hat \pi_{x'}) + 2^{r-1} \left( \alpha_+^k(x,x')\mathcal W_r^r(\hat \pi^k_x, \delta_{y_0}) + (1 - \hat p^k(x)) \mathcal W_r^r(\tilde \pi^k_x,\delta_{y_0}) \right. \\
  &\phantom{\mathcal W_r^r\le} \left. + \alpha_-^k(x,x') \mathcal W_r^r(\hat \pi_{x'},\delta_{y_0}) + (1 - \hat p(x')) \mathcal W_r^r(\tilde \pi_{x'},\delta_{y_0})\right).
	\end{split}
	\end{align}
  Since $\beta^k(x,x') = \hat p^k(x) \wedge \hat p(x') \leq 1$, we deduce from \eqref{eq:lemsum 1}, \eqref{eq:lemsum 2} and $\mathcal{AW}_r$-optimality of $\hat \rho^k$
  \begin{align} \label{eq:lemsum 3}
  	\begin{split}
    & \mathcal{AW}_r^r(\hat \pi^k + \tilde\pi^k, \hat \pi + \tilde \pi) \le \mathcal{AW}_r^r(\hat \pi^k,\hat \pi) + \int_{X \times X} d_X(x,x')^r \, \tilde \rho^k(dx,dx') \\
    & \phantom{\mathcal{AW}_r^r(\hat \pi^k + \tilde\pi^k, \hat \pi + \tilde \pi) \le \mathcal{A}}
    + 2^{r-1}\int_{X \times X} \hat p^k(x) \mathcal W_r^r(\hat \pi^k_x,\delta_{y_0}) \, \tilde \rho^k(dx,dx') \\
    & \phantom{\mathcal{AW}_r^r(\hat \pi^k + \tilde\pi^k, \hat \pi + \tilde \pi) \le \mathcal{A}}
    + 2^{r-1}\int_{X \times X} \hat p(x') \mathcal W_r^r(\hat \pi_{x'},\delta_{y_0}) \, \tilde \rho^k(dx,dx') \\
    & \phantom{\mathcal{AW}_r^r(\hat \pi^k + \tilde\pi^k, \hat \pi + \tilde \pi) \le \mathcal{A}}
    + 2^{r-1} \int_{X \times X} \alpha_+^k(x,x') \mathcal W_r^r(\hat \pi^k_x, \delta_{y_0}) \, \rho^k(dx,dx') \\
    & \phantom{\mathcal{AW}_r^r(\hat \pi^k + \tilde\pi^k, \hat \pi + \tilde \pi) \le \mathcal{A}}
    + 2^{r-1} \int_{X \times X} (1 - \hat p^k(x)) \mathcal W_r^r(\tilde \pi^k_{x},\delta_{y_0}) \, \rho^k(dx,dx') \\
    & \phantom{\mathcal{AW}_r^r(\hat \pi^k + \tilde\pi^k, \hat \pi + \tilde \pi) \le \mathcal{A}}
    + 2^{r-1} \int_{X \times X} \alpha_-^k(x,x')\mathcal W_r^r(\hat \pi_{x'}, \delta_{y_0}) \, \rho^k(dx,dx') \\
    & \phantom{\mathcal{AW}_r^r(\hat \pi^k + \tilde\pi^k, \hat \pi + \tilde \pi) \le \mathcal{A}}
    + 2^{r-1} \int_{X \times X}  (1 - \hat p(x')) \mathcal W_r^r(\tilde \pi_{x'},\delta_{y_0}) \, \rho^k(dx,dx').
  \end{split}
  \end{align}
  Recall that $\tilde \rho^k$ has marginals $\tilde \mu^k$ and $\tilde \mu$ with total mass smaller than $\epsilon$.
  By \eqref{eq:def Iepsilonr} we find
  \begin{equation}
    \label{ineq:lemsum 1}
       \int_{X\times X} d_X(x,x')^r \, \tilde\rho^k(dx,dx') \le 2^{r-1} \left(I_\varepsilon^r(\tilde\mu^k)+I_\varepsilon^r(\tilde\mu)\right).
  \end{equation}
  Concerning the marginals of $\hat p^k(x) \, \rho(dx,dx')$ and $\hat p(x') \, \rho(dx,dx')$, we find the relations
  \[ \hat p^k(x) \, \tilde \mu^k(dx) = (1 - \hat p^k(x)) \, \hat \mu^k(dx), \quad \hat p(x') \, \tilde \mu(dx') = (1 - \hat p(x')) \, \hat \mu(dx').\]
  Again by \eqref{eq:def Iepsilonr}, we find since $\tilde \rho^k \in \Pi(\tilde \mu^k,\tilde \mu)$, $\hat \pi^k \in \Pi(\hat \mu^k,\hat \nu^k)$ and $\hat \pi \in \Pi(\hat \mu,\hat \nu)$ that
  \begin{gather} \label{ineq:lemsum 2}
    \int_{X \times X} \hat p^k(x) \mathcal W_r^r(\hat \pi^k_x,\delta_{y_0}) \, \tilde \rho^k(dx,dx') = \int_{X \times X} (1 - \hat p^k(x)) \mathcal W_r^r(\hat \pi^k_x,\delta_{y_0}) \, \hat \mu^k(dx) \leq I_\epsilon^r(\hat \nu^k), \\ \label{ineq:lemsum 3}
    \int_{X \times X} \hat p(x') \mathcal W_r^r(\hat \pi_{x'},\delta_{y_0}) \, \tilde \rho^k(dx,dx') = \int_{X \times X} (1 - \hat p(x')) \mathcal W_r^r(\hat \pi_{x'},\delta_{y_0}) \, \hat \mu(dx') \leq I_\epsilon^r(\hat \nu).
  \end{gather}
  We deduce from \eqref{eq:lemsum 3} and \eqref{ineq:lemsum 1}-\eqref{ineq:lemsum 3} that it is sufficient to show
  \begin{align} \label{ineq:lemsum 4}
	   \int_{X\times X} \alpha_+^k(x,x') \mathcal W_r^r(\hat \pi^k_x,\delta_{y_0}) \, \rho^k(dx,dx') &\le I_\epsilon^r(\hat\nu^k),
     \\ \label{ineq:lemsum 5}
	   \int_{X\times X} (1 - \hat p^k(x)) \mathcal W_r^r(\tilde \pi^k_x,\delta_{y_0}) \, \rho^k(dx,dx') & \le  I_\varepsilon^r(\tilde\nu^k),
     \\ \label{ineq:lemsum 6}
	   \int_{X\times X} \alpha_-^k(x,x') \mathcal W_r^r(\hat \pi_{x'},\delta_{y_0}) \, \rho^k(dx,dx') &\le I_\varepsilon^r(\hat\nu),
     \\ \label{ineq:lemsum 7}
	   \int_{X\times X} (1 - \hat p(x')) \mathcal W_r^r(\tilde\pi_{x'}, \delta_{y_0}) \, \rho^k(dx,dx') &\le I_\varepsilon^r(\tilde\nu).
  \end{align}
  To see \eqref{ineq:lemsum 5} and \eqref{ineq:lemsum 7}, note that
  \begin{equation} \label{eq:lemsum identity}
    (1 - \hat p^k (x)) \, \mu^k(dx) = \tilde \mu^k(dx)\quad\text{and}\quad(1 - \hat p(x')) \, \mu(dx') = \tilde \mu(dx').
  \end{equation}

  As a consequence, the first marginal of $(1 - \hat p^k(x))\, \rho^k(dx,dx')$ is $\tilde \mu^k$, whereas the second marginal of $(1 - \hat p(x')) \, \rho^k(dx,dx')$ coincides with $\tilde \mu$.
  Hence, as the mass of $\tilde \mu^k$ and $\tilde \mu$ does not exceed $\epsilon$, we have
  \begin{align*}
    \int_{X \times X} (1 - \hat p^k(x)) \mathcal W_r^r(\tilde \pi^k_x,\delta_{y_0}) \, \rho^k(dx,dx') &= \int_X \mathcal W_r^r(\tilde \pi^k_x,\delta_{y_0}) \, \tilde \mu^k(dx) = \mathcal W_r^r(\tilde \nu^k, \delta_{y_0}) = I_\epsilon^r(\tilde \nu^k), \\
    \int_{X \times X} (1 - \hat p(x')) \mathcal W_r^r(\tilde \pi_{x'},\delta_{y_0}) \, \rho^k(dx,dx') &= \int_X \mathcal W_r^r(\tilde \pi_{x'},\delta_{y_0}) \, \tilde \mu(dx') = \mathcal W_r^r(\tilde \nu, \delta_{y_0}) = I_{\epsilon}^r(\tilde \nu).
  \end{align*}
  Next, we show \eqref{ineq:lemsum 4} and \eqref{ineq:lemsum 6}.
  To this end, denoting $\rho^k(dx,dx')=\mu^k(dx)\,\rho^k_x(dx')=\mu(dx')\,\overleftarrow{\rho}^k_{x'}(dx)$, we have
  \begin{gather*}
    \alpha_+^k(x,x') \, \rho^k(dx,dx') \leq \hat p^k(x) \, \rho^k(dx,dx') = \frac{d \hat \mu^k}{d \mu^k}(x)\, \mu^k(dx) \, \rho^k_x(dx') = \hat \mu^k(dx) \, \rho^k_x(dx'), \\
    \alpha_-^k(x,x') \, \rho^k(dx,dx') \leq \hat p(x') \, \rho^k(dx,dx') = \frac{d \hat \mu}{d \mu}(x') \, \mu(dx') \, \overleftarrow{\rho}^k_{x'}(dx) = \hat \mu(dx') \, \overleftarrow{\rho}^k_{x'}(dx).
  \end{gather*}
  In particular, the first marginal of $\alpha_+^k(x,x') \, \rho^k(dx,dx')$, denoted here by $\tau^k$, is dominated by $\hat \mu^k$, whereas the second marginal of $\alpha_-^k(x,x') \, \rho^k(dx,dx')$, denoted here by ${\tau^k}'$, is dominated by $\hat \mu$.
  Concerning the masses of $\tau^k$ and ${\tau^k}'$, remember \eqref{eq:lemsum identity}, $\alpha_+^k(x,x') \leq 1 - \hat p(x')$ and $\alpha_-^k(x,x') \leq 1 - \hat p^k(x)$, thus,
  \begin{gather*}
    \tau^k(X) = \int_{X \times X} \alpha_+^k(x,x') \, \rho^k(dx,dx') \leq \int_{X} (1 - \hat p(x')) \, \mu(dx') = \tilde \mu(X) \leq \epsilon, \\
    {\tau^k}'(X) = \int_{X \times X} \alpha_-^k(x,x') \, \rho^k(dx,dx') \leq \int_X (1 - \hat p^k(x)) \, \mu^k(dx) = \tilde \mu^k(X) \leq \epsilon.
  \end{gather*}
  Using \eqref{eq:def Iepsilonr}, we conclude with
  \begin{gather*}
    \int_{X \times X} \alpha_+^k(x,x') \mathcal W_r^r(\hat \pi^k_x, \delta_{y_0}) \, \rho^k(dx,dx') = \int_X \mathcal W_r^r(\hat \pi^k_x,\delta_{y_0}) \, \tau(dx) \leq I_\epsilon^r(\hat \nu^k), \\
    \int_{X \times X} \alpha_-^k(x,x') \mathcal W_r^r(\hat \pi_{x'}, \delta_{y_0}) \, \rho^k(dx,dx') = \int_X \mathcal W_r^r(\hat \pi_{x'},\delta_{y_0}) \, \tau'(dx') \leq I_\epsilon^r(\hat \nu).
  \end{gather*}
\end{proof}
The addition on $\mathcal M_r(X\times Y)$ is continuous with respect to the adapted weak topology provided the limits have singular first marginal distributions. We recall that two positive measures $\mu,\nu$ are called singular \textcolor{red}{if and only if}
 there exists a measurable set $A\subset X$ such that $\mu(A^\complement)=0=\nu(A)$.

 \begin{lemma}\label{lem:continuous addition} Let $\pi,\chi\in\mathcal M_r(X\times Y)$ be such that their respective first marginals are singular. Let $\pi^k,\chi^k\in\mathcal M_r(X\times Y)$, $k\in\N$ converge to $\pi$ and $\chi$ respectively in $\mathcal{AW}_r$. Then
	\[
	\pi^k+\chi^k\underset{k\to+\infty}{\longrightarrow}\pi+\chi\quad\text{in }\mathcal{AW}_r.
	\]
\end{lemma}

\begin{proof} Let $\mu_1$, $\mu_2$, $\mu_1^k$ and $\mu_2^k$ denote the respective first marginals of $\pi$, $\chi$, $\pi^k$ and $\chi^k$. Due to singularity, there is a measurable set $A\subset X$ such that $\mu_1(A^\complement)=0=\mu_2(A)$.

	Suppose first that for all $k\in\N$, $\mu^k_1(A^\complement)=0=\mu^k_2(A)$. Let $\rho^k_1\in\Pi(\mu^k_1,\mu_1)$, resp.\ $\rho^k_2\in\Pi(\mu^k_2,\mu_2)$, be an optimal coupling for $\mathcal{AW}_r(\pi^k,\pi)$, resp.\ $\mathcal{AW}_r(\chi^k,\chi)$. Since almost surely
	\[
	(\pi^k+\chi^k)_x=\mathbbm1_A(x)\,\pi^k_x+\mathbbm1_{A^\complement}(x)\chi^k_x\quad\text{and}\quad(\pi+\chi)_x=\mathbbm1_A(x)\,\pi_x+\mathbbm1_{A^\complement}(x)\chi_x,
	\]
	we have
	\begin{align*}
	\mathcal{AW}_r^r(\pi^k+\chi^k,\pi+\chi)&\le\int_{X\times X}\left(d^r_X(x,x')+\mathcal W_r^r((\pi^k+\chi^k)_x,(\pi+\chi)_{x'})\right)\,(\rho^k_1+\rho^k_2)(dx,dx')\\
	&=\int_{X\times X}\left(d_X^r(x,x')+\mathcal W_r^r(\pi^k_x,\pi_{x'})\right)\,\rho^k_1(dx,dx')\\
	&\phantom{\le\ }+\int_{X\times X}\left(d_X^r(x,x')+\mathcal W_r^r(\chi^k_x,\chi_{x'})\right)\,\rho^k_2(dx,dx')\\
	&=\mathcal{AW}_r^r(\pi^k,\pi)+\mathcal{AW}_r^r(\chi^k,\chi)\to0,\quad k\to+\infty.
	\end{align*}
	Let us now go back to the general case. Let $\varepsilon>0$. Since $X$ is a Polish space, $\mu_1$ and $\mu_2$ are inner regular, so there exist two compact sets $K_1\subset A$ and $K_2\subset A^\complement$ such that
	\[
\mu_1(K_1^\complement)<\varepsilon\quad\text{and}\quad\mu_2(K_2^\complement)<\varepsilon.
	\]
  Since $X$ is \textcolor{blue}{metrizable}, it is normal, hence we can separate the closed, disjoint sets $K_1$ and $K_2$ by open, disjoint sets $\tilde K_1$ and $\tilde K_2$ where $K_1 \subset \tilde K_1$ and $K_2 \subset \tilde K_2$.
  Then Lemma \ref{lem:subprob convergence} \ref{it:subprob convergence1} provides sequences $(\tilde \mu^k_1 \times \pi^k_x)_{k\in\N}$ and $(\tilde \mu^k_2 \times \chi^k_x)_{k\in\N}$ with values in $\mathcal M(X\times Y)$ and null sequences $(\varepsilon_k)_{k\in\N}$ and $(\eta_k)_{k\in\N}$ with values in $[0,1]$, such that $\tilde\mu^k_1\le\mu^k_1\vert_{\tilde K_1}$, $\tilde\mu^k_2\le\mu^k_2\vert_{\tilde K_2}$ and
	\[
	\mathcal{AW}_r^r(\tilde\mu^k_1\times\pi^k_x,(1-\varepsilon_k)\pi\vert_{K_1\times Y})+\mathcal{AW}_r^r(\tilde\mu^k_2\times\chi^k_x,(1-\eta_k)\chi\vert_{K_2\times Y})\to0,\quad k\to+\infty.
	\]
  To apply Lemma \ref{lem:inequality addition AWr} \ref{it:addition AWr limsup}, let $0 < \epsilon' \leq \epsilon$ be such that $\epsilon' (\mu_1(K_1) + \mu_2(K_2)) < \epsilon$.
  Let $k$ be sufficiently large such that $\epsilon^k \wedge \eta^k < \epsilon'$.
  We consider the sequences
  \begin{gather*}
    \hat \pi^k = \frac{1 - \epsilon'}{1 - \epsilon^k} \tilde \mu^k_1 \times \pi^k_x + \frac{1 - \epsilon'}{1 - \eta^k} \tilde \mu^k_2 \times \chi^k_x, \quad \hat \pi = (1 - \epsilon')\left( \pi\vert_{K_1 \times Y} + \chi\vert_{K_2 \times Y}\right), \\
    \tilde \pi^k = \pi^k + \chi^k - \hat \pi^k, \quad \tilde \pi = \pi + \chi - \hat \pi,
  \end{gather*}
  where $\tilde \pi^k$ is well-defined in $\mathcal M_r(X \times Y)$ since $\epsilon^k < \epsilon'$ and $\eta^k < \epsilon'$.
  Note that as $k\to+\infty$,
  \begin{gather*}
    \mathcal{AW}_r^r\left(\frac{1 - \epsilon'}{1 - \epsilon^k} \tilde \mu^k_1 \times \pi^k_x, (1 - \epsilon') \pi\vert_{K_1 \times Y}\right) = \frac{1 - \epsilon'}{1 - \epsilon^k} \mathcal{AW}_r^r\left(\tilde \mu^k_1 \times \pi^k_x, (1 - \epsilon^k) \pi \vert_{K_1 \times Y}\right) \to 0, \\
    \mathcal{AW}_r^r\left(\frac{1 - \epsilon'}{1 - \eta^k} \tilde \mu^k_2 \times \chi^k_x, (1 - \epsilon') \chi\vert_{K_2 \times Y}\right) = \frac{1 - \epsilon'}{1 - \eta^k} \mathcal{AW}_r^r\left(\tilde \mu^k_2 \times \chi^k_x, (1 - \eta^k) \chi \vert_{K_2 \times Y}\right) \to 0.
  \end{gather*}
  Since the first marginal distributions of $\tilde\mu^k_1\times\pi^k_x$ and $(1-\varepsilon_k)\pi\vert_{K_1\times Y}$, resp.\ $\tilde\mu^k_2\times\chi^k_x$ and $(1-\eta_k)\chi\vert_{K_2\times Y}$, are concentrated on $\tilde K_1$, resp.\ $\tilde K_2$, and since $\tilde K_1$ and $\tilde K_2$ are disjoint, we have according to the preceding part that
	\[
	\mathcal{AW}_r^r(\hat \pi^k,\hat \pi)\to0,\quad k\to+\infty.
	\]
  Due to $\mathcal{AW}_r$-convergence of $(\pi^k)_{k\in\N}$ and $(\chi^k)_{k\in\N}$, we obtain $\mathcal W_r$-convergence of the marginals of $\pi^k + \chi^k$ to the marginals of $\pi + \chi$.
  Furthermore, we have
  \[\tilde \pi^k(X \times Y) = \tilde \pi(X \times Y) \leq \mu_1(K_1^c) + \mu_2(K_2^c) + \epsilon' (\mu_1(K_1) + \mu_2(K_2)) < 3 \epsilon.  \]
  Then \eqref{eq:addition AWr limsup} yields
	\begin{align*}
    \limsup_{k\to+\infty}\mathcal{AW}_r^r(\pi^k+\chi^k,\pi+\chi) &= \limsup_{k\to+\infty}\mathcal{AW}_r^r(\hat \pi^k+\tilde \pi^k,\hat \pi+\tilde \pi) \\ &\le C\left(I^r_{3\varepsilon}(\mu_1+\mu_2)+I^r_{3\varepsilon}(\nu_1+\nu_2)\right),
  \end{align*}
  where $\nu_1$ and $\nu_2$ denote the respective second marginals of $\pi$ and $\chi$, and the constant $C$ only depends on $r$.
  Therefore, the right-hand side vanishes as $\epsilon\to 0$ according to Lemma \ref{lem:uniform integrability} \ref{it:IepsilonrVanishes}, which concludes the proof.
\end{proof}

		\section{Auxiliary results on the convex order in dimension one}
		\label{sec: convex order}
		We recall that the convex order on $\mathcal M_1(\R)$ is defined by
		\[
		\mu \leq_c \nu \iff \quad \forall f\colon\R\to\R\text{ convex},\quad \mu(f) \leq \nu(f).
		 \]
		The following assertions can be found  \textcolor{blue}{for instance be found in \cite[Section 2]{HiRo12}}: for all $(m_0,m_1)\in\R_+^*\times\R$, there is a one-to-one correspondence between finite positive measures $\mu\in\mathcal M_1(\R)$ with mass $m_0$ such that $\int_\R y\,\mu(dy)=m_1$ and the set of functions $u \colon \R\to\R^+$ which satisfy
	\begin{enumerate}[(i)]
		\item \label{it:u1cx} $u$ is convex;
		\item \label{it:u2asy} $u(y) - m_0|y-m_1|$ goes to $0$ as $|y|$ tends to $+\infty$.
	\end{enumerate}
	Any function which satisfies \ref{it:u1cx} and \ref{it:u2asy} is then called a potential function. As noted above, the potential function of $\mu$ is denoted by
	\[
	u_\mu(y) = \int_\R |y-x|\,\mu(dx).
	\]
	\textcolor{red}{Potential functions can of course also be considered in greater generality than on the real line, but this is not relevant for our purposes.} 
	
	A sequence $(\mu^k)_{k\in\N}$ of finite positive measures with equal masses on the line  converges in $\mathcal{W}_1$ to $\mu$ \textcolor{red}{if and only if}
 the sequence of potential functions $(u_{\mu^k})_{k\in\N}$ converges pointwise to $u_\mu$. In that case, since for all $y\in\R$ the map $x\mapsto\vert y-x\vert$ is Lipschitz continuous with constant $1$, we have by Kantorovich and Rubinstein's duality theorem  that
	\[
	\sup_{y\in\R}\vert u_{\mu^k}(y)-u_\mu(y)\vert\le\mathcal W_1(\mu^k,\mu)\to0,\quad k\to+\infty,
	\]
	hence we even have uniform convergence on $\R$ of potential functions.

\textcolor{red}{For all $m_1\in\R$, the set of all finite positive measures on the real line} with mean $m_1$ is a lattice \cite[Proposition 1.6]{KeRo}, and even a complete lattice \cite{KeRob} \textcolor{red}{for the convex order}. Then all $\mu,\nu\in\mathcal M_1(\R)$ with mean $m_1$ have a supremum, denoted $\mu\vee_c\nu$, and an infimum, denoted $\mu\wedge_c\nu$, with respect to the convex order. In that context it is convenient to work with potential functions since they provide simple characterisations of those bounds:
 	\begin{gather*}
 		\mu\vee_{c}\nu \text{ is defined as the measure with potential function } u_\mu \vee u_\nu,\\
 		\mu \wedge_{c} \nu \text{ is defined as the measure with potential function } \operatorname{co}(u_\mu \wedge u_\nu),
 	\end{gather*}
 	where $\operatorname{co}$ is the convex hull.
  \begin{lemma} \label{lemcontinfsup}
    Let $(\mu^k)_{k\in\N}$, $(\nu^k)_{k\in\N}$ be two sequences of $\mathcal M_1(\R)$ converging respectively to $\mu$ and $\nu$ in $\mathcal W_1$. Suppose that there exists $(m_0,m_1)\in\R_+^*\times\R$ such that $\mu^k(\R)=\nu^k(\R)=m_0$ and $\int_\R x\, \mu^k(dx)=\int_\R y\, \nu^k(dy)=m_1$ for all $k \in \N$.
    Then
    \[\lim_{k\to+\infty} \mathcal W_1(\mu^k \vee_c \nu^k, \mu \vee_c \nu) = 0\quad\text{and}\quad \lim_{k\to+\infty} \mathcal W_1(\mu^k\wedge_c \nu^k, \mu \wedge_c \nu) = 0. \]
  \end{lemma}

  \begin{proof}
    Convergence in $\mathcal W_1$ is equivalent to pointwise convergence of the potential functions.
  	Thus, the convergence of $\mu^k\vee_c \nu^k$ to $\mu\vee_c\nu$ in $\mathcal W_1$ is a consequence of the pointwise convergence of $u_{\mu^k\vee_c \nu^k}=u_{\mu^k}\vee u_{\nu^k}$ to $u_\mu\vee u_\nu=u_{\mu\vee_c\nu}$.

  	To show convergence of $\mu^k\wedge_c \nu^k$ to $\mu\wedge_c\nu$ in $\mathcal W_1$, it is sufficient to show for all $x \in \R$
  	\begin{equation}\label{eq:convex hull convergence}
  	u_{\mu^k\wedge_c \nu^k}(x)=\operatorname{co}(u_{\mu^k}\wedge u_{\nu^k})(x)\to\operatorname{co}(u_\mu\wedge u_\nu)(x)=u_{\mu\wedge_c \nu}(x),\quad k\to+\infty.
  	\end{equation}

  	Since $u_{\mu^k}$ and $u_{\nu^k}$ converge uniformly on $\R$ to $u_\mu$ and $u_\nu$ respectively, we have uniform convergence of $u_{\mu^k}\wedge u_{\nu^k}$ to $u_\mu\wedge u_\nu$.
    Let $\varepsilon>0$ and $k_0\in\N$ be such that for all $k\ge k_0$,
    \[
    \sup_{x\in\R}\vert (u_{\mu^k}\wedge u_{\nu^k})(x)-(u_\mu\wedge u_\nu)(x)\vert\le\varepsilon.
    \]

    For all $k\ge k_0$, we find
    \begin{gather*}
       \operatorname{co}(u_{\mu}\wedge u_{\nu})-\varepsilon \leq (u_{\mu}\wedge u_{\nu}) - \epsilon \leq u_{\mu^k}\wedge u_{\nu^k},\\
      \operatorname{co}(u_{\mu^k}\wedge u_{\nu^k})-\varepsilon \leq (u_{\mu^k}\wedge u_{\nu^k}) - \epsilon \leq u_\mu\wedge u_\nu.
    \end{gather*}

    Thus, as the convex hull is the supremum over all dominated, convex functions, this yields
    \[
    \operatorname{co}(u_\mu\wedge u_\nu)-\varepsilon\le\operatorname{co}(u_{\mu^k}\wedge u_{\nu^k})\le\operatorname{co}(u_\mu\wedge u_\nu)+\varepsilon,
    \]
    which establishes \eqref{eq:convex hull convergence} and completes the proof.
 \end{proof}

\textcolor{red}{We now provide the proof of Proposition \ref{prop:approximationofdecomp} which is the key argument to see that it is enough to prove our main result, namely Theorem \ref{thm:1}, for irreducible pairs of marginals.}

\begin{proof}[Proof of Proposition \ref{prop:approximationofdecomp}]
  To construct the desired decomposition, pick for all $k\in\N$ a coupling $\pi^k \in \Pi_M(\mu^k,\nu^k)$.
  Let $l_n$ and $r_n$ denote the left and right boundary of the open interval $\{u_{\mu_n}<u_{\nu_n}\}$ on which $\mu_n$ is concentrated, and set
  \[ \mu^k_n(dx) = \int_{u = F_\mu(l_n)}^{F_\mu(r_n-)} \delta_{F_{\mu^k}^{-1}(u)}(dx) \, du, \quad \nu_n^k(dy)=\int_{u=F_\mu(l_n)}^{F_\mu(r_n-)}\pi^k_{F_{\mu^k}^{-1}(u)}(dy) \, du.\]
 These are the respective marginals of $\tilde\pi^{k,n}$ on $\R^2$ given by
 \begin{equation}\label{def:pikn} \tilde\pi^{k,n}(dx,dy)=\int_{u=F_\mu(l_n)}^{F_\mu(r_n-)}\delta_{F_{\mu^k}^{-1}(u)}(dx)\,\pi^k_{F_{\mu^k}^{-1}(u)}(dy) \, du.
 \end{equation}
 Since $\pi^k$ is a martingale coupling, we have $\mu^k_n\le_c\nu^k_n$.
 Finally define 
 \[
 J = [0,1] \backslash \bigcup_{n\in N}(F_\mu(l_n),F_\mu(r_n-)),
 \]
 and set
 \[ \eta^k(dx) = \int_{u \in J} \delta_{F_{\mu^k}^{-1}(u)}(dx) \, du,\quad \upsilon^k(dy) = \int_{u \in J} \pi^k_{F^{-1}_{\mu^k}(u)}(dy) \, du. \]
 These are the respective marginals of $\tilde\pi^k$ defined by
 \[ \tilde\pi^k(dx,dy) = \int_{u \in J} \delta_{F_{\mu^k}^{-1}(u)}(dx)\, \pi^k_{F_{\mu^k}^{-1}(u)}(dy) \, du, \]
 which is again a martingale coupling with marginals $(\eta^k,\upsilon^k)$, thus, $\eta^k\le_c\upsilon^k$.

 Using inverse transform sampling for the second equality, we find
 \begin{align*}
   \left(\tilde\pi^k + \sum_{n\in N} \tilde\pi^{k,n}\right)(dx,dy)
   &= \int_{u=0}^{1} \delta_{F_{\mu^k}^{-1}(u)}(dx)\,\pi^k_{F_{\mu^k}^{-1}(u)}(dy) \, du= \int_{x^k\in\R}\delta_{x^k}(dx)\, \pi^k_{x^k}(dy) \, \mu^k (dx^k)\\
   &=\mu^k(dx)\,\pi^k_x(dy)=\pi^k(dx,dy).
 \end{align*}
 Concerning the marginals, we deduce
 \[ \eta^k + \sum_{n\in N} \mu^k_n = \mu^k\quad\text{and}\quad\upsilon^k + \sum_{n\in N} \nu^k_n = \nu^k. \]
 For all $(\tau,u,l,r)\in\mathcal P_1(\R)\times(0,1)\times\R\times\R$, we have by \eqref{eq:equivalence quantile cdf3}:
 \begin{equation}\label{eq:jumps Fmu component}
 F_\tau(l)<u<F_\tau(r-)\implies l<F_\tau^{-1}(u)<r\implies F_\tau(l)<u\le F_\tau(r-).
 \end{equation}
  Since $\mu_n(dx)=\mathbbm1_{(l_n,r_n)}(x)\,\mu(dx)$, using \eqref{eq:jumps Fmu component} for the second equality we find
 \[
 \mu_n(dx)=\int_{x'\in(l_n,r_n)}\delta_{x'}(dx)\,\mu(dx)=\int_{u = F_\mu(l_n)}^{F_\mu(r_n-)} \delta_{F_{\mu}^{-1}(u)}(dx) \, du.
 \]
 We deduce that
 \begin{align*}
   \eta(dx)=\left(\mu-\sum_{n\in N}\mu_n\right)(dx)&=\int_{u=0}^1\delta_{F_\mu^{-1}(u)}(dx)\,du-\sum_{n\in N}\int_{u = F_\mu(l_n)}^{F_\mu(r_n-)} \delta_{F_{\mu}^{-1}(u)}(dx) \, du \\ &= \int_{u \in J} \delta_{F_{\mu}^{-1}(u)}(dx) \, du.
 \end{align*}
 Since the monotone rearrangement yields an optimal coupling, we have
 \[\mathcal W_1(\eta^k,\eta) + \sum_{n\in N} \mathcal W_1(\mu^k_n,\mu_n)= \int_0^1 | F_{\mu^k}^{-1}(u) - F_\mu^{-1}(u) | \, du = \mathcal W_1(\mu^k,\mu),\]
 hence
 \[ \lim_{k\to+\infty} \mathcal W_1(\eta^k,\eta) = 0 = \lim_{k\to+\infty} \mathcal W_1 (\mu^k_n,\mu_n), \quad \forall n \in N. \]
Since the marginals of $\pi^k$ converge weakly, the sequences $(\mu^k)_{k\in\N}$ and $(\nu^k)_{k\in\N}$ are tight, and so is $(\pi^k)_{k\in\N}$.
For $n \in N$, $\tilde\pi^{k,n}$ is dominated by $\pi^k$, hence $(\tilde\pi^{k,n})_{k\in \N}$ is tight and therefore relatively compact. 
Moreover, by $\mathcal W_1$-convergence of $(\mu^k)_{k\in\N}$ and $(\nu^k)_{k\in\N}$, the sequences $\left(\int_\R\vert x\vert\,\mu^k(dx)\right)_{k\in\N}$ and $\left(\int_\R\vert y\vert\,\nu^k(dy)\right)_{k\in\N}$ converge and are in particular bounded.
Hence the sequences $\left(\int_\R\vert x\vert\,\mu^k_n(dx)\right)_{k\in\N}$ and $\left(\int_\R\vert y\vert\,\nu^k_n(dy)\right)_{k\in\N}$ are bounded as well and admit convergent subsequences.
Since the $\mathcal W_1$-convergence is equivalent to the weak convergence plus convergence of the first moments, we deduce that the sequence $(\tilde\pi^{k,n})_{k \in \N}$ is relatively compact in $\mathcal W_1$.
Since $(\pi^k)_{k\in\N}$ is tight, from any subsequence we can extract a further subsequence denoted by $(\pi^{k_j})_{j\in\N}$ which converges weakly to some $\pi\in\Pi_M(\mu,\nu)$.
 There are subsequences $(\tilde\pi^{k_j,n})_{j \in \N}$ converging in $\mathcal W_1$ to a measure $\tilde \pi_n$. \textcolor{blue}{Moreover $\tilde \pi_n \leq \pi$ with $\pi\in\Pi_M(\mu,\nu)$ denoting the weak limit of a subsequence of the tight sequence $(\pi^{k_j})_{j \in \N}$.
 The first marginal of $\tilde \pi_n$ coincides with $\mu_n$ due to the continuity of the projection, thus,\[ \tilde \pi_n \leq \pi\vert_{\textcolor{red}{(l_n,r_n)} \times \R} =\textcolor{purple}{:} \pi_n. \]
 As $\tilde \pi_n(\R \times \R) = \mu_n(\textcolor{red}{(l_n,r_n)}) = \pi_n(\R\times \R)$, there must hold equality, i.e., $\tilde \pi_n = \pi_n$ and $\int_{x\in\R}\tilde\pi_n(dx,dy)=\nu_n(dy)$. By continuity of the projection, we deduce that $\lim_{j\to\infty}\mathcal W_1(\nu^{k_j}_n,\nu_n)=0$ and, since the limit does not depend on the subsequence, $(\nu^k_n)_{k\in\N}$ converges in $\mathcal W_1$ to $\nu_n$.} Analogously, we find that $(\upsilon^k)_{k \in \N}$ converges to $\eta$.
\end{proof}

	The next two lemmas explore the influence of certain scaling and restrictions of measure on condition that the transformed measures are in convex order.

\begin{lemma}\label{lem:scaling}
  Let $r\ge1$ and $\mu\in \mathcal M_r(\R^d)$ be a finite positive measure. Let $m_1=\int_\R x\,\mu(dx)$ and $\mu^\alpha$, $\alpha\in\R_+$ be the image of $\mu$ by $y \mapsto \alpha(y-m_1) + m_1$. Then for all $\alpha,\beta\in\R_+$,
  \begin{equation}\label{eq:Wasserstein scaling}
    \mathcal W_r(\mu^\alpha,\mu^\beta)=\vert\beta-\alpha\vert\left(\int_{\R^d}\vert x-m_1\vert^r\,\mu(dx)\right)^{\frac{1}{r}}\textcolor{red}{=\vert\beta-\alpha\vert\mathcal W_r(\mu^0,\mu^1)}.
  \end{equation}
  Moreover, $(\mu^\alpha)_{\alpha \in \R^+}$ constitutes a peacock, i.e., $\alpha \leq \beta \in \R_+$ implies  $\mu^\alpha \leq_c \mu^\beta$.
\end{lemma}




\begin{proof}
  Let $\alpha\le\beta\in\R_+$.
  By the triangle inequality we obtain
  \begin{align*}
    \left( \int_{\R^d} \vert x-m_1 \vert^r \, \mu^\beta(dx) \right)^{\frac{1}{r}} &=
     \mathcal W_r(\delta_{m_1},\mu^\beta) \le \mathcal W_r(\delta_{m_1},\mu^\alpha) + \mathcal W_r(\mu^\alpha,\mu^\beta) \\
     & = \left( \int_{\R^d} \vert x - m_1 \vert^r \, \mu^\alpha(dx) \right)^{\frac{1}{r}} + \mathcal W_r(\mu^\alpha,\mu^\beta).
  \end{align*}
  Thus,
  \begin{align*}
      \mathcal W_r(\mu^\alpha,\mu^\beta) &\ge \left( \int_{\R^d} \vert x - m_1 \vert^r \, \mu^\beta(dx) \right)^{\frac{1}{r}} - \left( \int_{\R^d} \vert x - m_1 \vert^r \, \mu^\alpha(dx) \right)^{\frac{1}{r}} \\
       & = ( \beta - \alpha) \left( \int_{\R^d} \vert x - m_1 \vert^r \, \mu(dx) \right)^{\frac{1}{r}}.
  \end{align*}
  Since the image of $\mu$ under $x\mapsto(\alpha(x-m_1)+m_1,\beta(y-m_1)+m_1)$ is a coupling between $\mu^\alpha$ and $\mu^\beta$, we also have the reverse inequality
  \[
    \mathcal W_r (\mu^\alpha,\mu^\beta) \le (\beta-\alpha) \left( \int_{\R^d} \vert x - m_1 \vert^r \, \mu(dx) \right)^{\frac{1}{r}},
  \]
  which proves \eqref{eq:Wasserstein scaling}.

	To see that $(\mu^\alpha)_{\alpha \in \R^+}$ is a peacock, we fix again $\alpha \leq \beta\in\R_+$ and a convex function $f$ on $\R^d$. By convexity, we have
	\begin{align*}
    \mu^\alpha(f) &= \int_{\R^d} f(\alpha(x-m_1)+m_1) \, \mu(dx) \\
    &\leq \int_{\R^d}\left( \frac{\alpha}{\beta} f(\beta(x-m_1)+m_1) +
    \left(1 - \frac{\alpha}{\beta}\right) f(m_1)\right) \, \mu(dx) \leq \mu^\beta(f).
	\end{align*}
\end{proof}

\begin{lemma}\label{lem:compact approx} For all $p\in\mathcal P_1(\R)$ with barycentre $m_1\in\R$ and $R\ge0$, let $p^R$ be defined by
	\[
	p^R=p\wedge_c\left(\frac{R-m_1}{2R}\,\delta_{-R}+\frac{R+m_1}{2R}\,\delta_R\right)\quad\text{if}\quad R\ge\vert m_1\vert,
	\]
	and $p^R=\delta_{m_1}$ otherwise. Then
	\begin{enumerate}[label=(\alph*)]
		\item\label{it:compact approx} For all $R>0$, $p^R\le_c p$, and if $R\ge\vert m_1\vert$, then $p^R$ is concentrated on $[-R,R]$.
		\item\label{it:compact approx2} We have
		\[
		\mathcal W_1(p^R,p)\underset{R\to+\infty}{\longrightarrow}0.
		\]
	\end{enumerate}
\end{lemma}
\begin{proof} Let $p\in\mathcal P_1(\R)$ be with barycentre $m_1\in\R$. For all $R\ge\vert m_1\vert$, let $\eta^R=\frac{R-m_1}{2R}\,\delta_{-R}+\frac{R+m_1}{2R}\,\delta_R$, so that $p^R=p\wedge_c\eta^R$. If $R<\vert m_1\vert$ then $p^R=\delta_{m_1}$ so we clearly have $p^R\le_c p$. Else, $p^R\le_c p$ still holds by definition of the convex infimum. Moreover, since $\eta^R$ is concentrated on $[-R,R]$, so is $p^R$ by domination in the convex order, hence \ref{it:compact approx} is proved.
	
	To show \ref{it:compact approx2}, \textcolor{red}{it suffices to verify pointwise convergence of the corresponding potential functions, i.e., for all} $y\in\R$,
	\begin{equation}\label{cvgencePotentialFunctions}
	u_{p\wedge\eta^R}(y)=\operatorname{co}(u_p\wedge u_{\eta^R})(y)\to u_p(y),\quad R\to+\infty.
	\end{equation}
	
	Let $\varepsilon>0$. Since $u_p(y)-\vert y-m_1\vert$ vanishes as $\vert y\vert\to+\infty$, there exists $M>0$ such that
	\[
	\forall y\in\R,\quad\vert y\vert> M\implies u_p(y)\le \vert y-m_1\vert+\varepsilon.
	\]
	Let $R_0=\vert m_1\vert+\sup_{\textcolor{red}{x}\in[-M,M]}u_p(\textcolor{red}{x})$ and $R\ge R_0$. The map $u_{\eta^R}$ is a piecewise affine function which changes slope at $-R$ and $R$ and such that $u_{\eta^R}(y)\to+\infty$ as $\vert y\vert\to+\infty$. It therefore attains its minimum either at $-R$ where it is equal to $R+m_1$ or at $R$ where it is equal to $R-m_1$, and this minimum is equal to $R-\vert m_1\vert$. We deduce that for all $y\in\R$, $u_{\eta^R}(y)\ge R-\vert m_1\vert$. Moreover, \textcolor{red}{$\delta_{m_1}\le_c {\eta^R}$}, hence we also have $u_{\eta^R}(y)\ge\vert y-m_1\vert$ for all $y\in\R$. Let $y\in\R$. If $\vert y\vert\le M$, then
	\[
	u_p(y)\le\sup_{\textcolor{red}{x \in }[-M,M]}u_p\textcolor{red}{(x)}=R_0-\vert m_1\vert\le R-|m_1|\le u_{\eta^R}(y).
	\]
 If, on the other hand,  $\vert y\vert> M$
 , then
	\[
	u_p(y)\le\vert y-m_1\vert+\varepsilon\le u_{\eta^R}(y)+\varepsilon.
	\]
	We deduce that for all $y\in\R$ and $R\ge R_0$, $u_p(y)-\varepsilon\le(u_p\wedge u_{\eta^R})(y)$. Thus, as the convex hull is the supremum over all dominated, convex functions, this yields
	\[
	u_p-\varepsilon\le\operatorname{co}(u_p\wedge u_{\eta^R})\le u_p,
	\]
	which proves \eqref{cvgencePotentialFunctions} and completes the proof.
\end{proof}

	 \section{Proof of the main theorem}
	 \label{sec:proof main thm}

	 We consider the setting of Theorem \ref{thm:1}. Before entering its technical proof, we argue that it is sufficient to consider the case $r=1$ and that we can assume w.l.o.g.\ that $(\mu,\nu)$ is irreducible.

	 When considering a sequence of couplings $(\pi^k)_{k \in \N}$ which converges in $\mathcal{AW}_1$ to $\pi\in\Pi(\mu,\nu)$, whose sequence of marginal distributions $(\mu^k,\nu^k)_{k\in\N}$ is converging in $\mathcal W_r$, one can deduce $\mathcal{AW}_r$-convergence for the sequence of couplings. This is due to \eqref{eq:AW=W}, and $\mathcal W_r$-convergence being equivalent to weak convergence plus convergence of the $r$-moments. To see the latter, we find, when equipping $X \times \mathcal P_r(Y)$ with the product metric $((x,p),(x',p'))\mapsto(d_X^r(x,x') + \mathcal W_r^r(p,p'))^{1/r}$,
\begin{align}\label{eq:Wr equiv weak moments}
	\begin{split}
		\int_{X \times\mathcal P_r(Y)}\left(d_X^r(x,x_0) +  \mathcal W_r^r(p,\delta_{y_0})\right)\, J(\pi^k)(dx,dp) = \mathcal W_r^r(\mu^k,\delta_{x_0}) + \mathcal W_r^r(\nu^k,\delta_{y_0}) \\ \underset{k\to+\infty}{\longrightarrow} \mathcal W_r^r(\mu,\delta_{x_0}) + \mathcal W_r^r(\nu,\delta_{y_0}) = \int_{X \times\mathcal P_r(Y)}\left( d_X^r(x,x_0) + \mathcal W_r^r(p,\delta_{y_0}) \right)\, J(\pi)(dx,dp).
	\end{split}
\end{align}

A direct consequence is the following lemma, according to which proving Theorem \ref{thm:1} for $r=1$ is sufficient.

%

	 \begin{lemma} \label{lem:r=1 is enough}
	 	In the setting of Theorem \ref{thm:1}, assume that there exists a sequence of martingale couplings $\pi^k\in\Pi_M(\mu^k,\nu^k)$, $k\in\N$ converging to $\pi$ in $\mathcal{AW}_1$. Then this sequence also converges to $\pi$ in $\mathcal{AW}_r$.
	 \end{lemma}

Next, Proposition \ref{prop:approximationofdecomp} is the key ingredient to show that it is enough to prove Theorem \ref{thm:1} when $(\mu,\nu)$ is irreducible.

\begin{lemma}\label{lem:irreducible is enough} If the conclusion of Theorem \ref{thm:1} holds for $r=1$ and for any irreducible pair of marginals $(\mu,\nu)$, then it holds for $r=1$ and for any pair $(\mu,\nu)$ in the convex order.
\end{lemma}
\begin{proof} In the setting of Theorem \ref{thm:1}, fix $\pi \in \Pi_M(\mu,\nu)$. Denote by $(\mu_n,\nu_n)_{n\in N}$ the decomposition of $(\mu,\nu)$ into irreducible components with
	\[
	\mu =\eta+ \sum_{n\in N} \mu_n,\quad \nu=\eta+\sum_{n\in N} \nu_n.
	\]

	By Proposition \ref{prop:approximationofdecomp}, we can find sub-probability measures $(\eta^k,\upsilon^k)_{k\in\N}$, $(\mu^k_n)_{(k,n)\in\N\times  N}$, $(\nu^k_n)_{(k,n)\in\N\times N}$ such that
	\begin{gather*}
	\eta^k \leq_c \upsilon^k,\quad \mu^k_n \leq_c \nu^k_n\quad \forall (k,n)\in\N\times N,\\
	\eta^k \to \eta,\quad \upsilon^k \to \eta,\quad\mu^k_n \to \mu_n,\quad \nu^k_n\to \nu_n\quad \text{in }\mathcal W_1,\quad k\to+\infty.
	\end{gather*}
	For $k\in\N$, let $\chi^k\in\Pi_M(\eta^k,\upsilon^k)$ be a  martingale coupling between $\eta^k$ and $\upsilon^k$. Since the marginals both converge to $\eta$ in $\mathcal W_1$, $(\chi^k)_{k\in\N}$ is tight and any accumulation point with respect to the weak topology belongs to $\Pi_M(\eta,\eta)$. Since $\chi:=(\operatorname{id},\operatorname{id})_\ast\eta$ is the only martingale coupling between $\eta$ and itself, $(\chi^k)_{k\in\N}$ converges weakly to $\chi$ as $k$ goes to $+\infty$ and even in $\mathcal W_1$ according to \eqref{eq:Wr equiv weak moments}. We can show that this convergence also holds in $\mathcal{AW}_1$.
	Indeed, according to Proposition \ref{prop:adapted approximation dim1}, there exists a sequence $\tilde\chi^k\in\Pi(\eta^k,\upsilon^k)$, $k\in\N$, converging to $\chi$ in $\mathcal{AW}_1$. Then
	\begin{align*}
	\mathcal{AW}_1 (\chi^k, \tilde\chi^k) & \le \int_\R \mathcal W_1 (\chi^k_x, \tilde \chi^k_x) \, \eta^k(dx) \le\int_\R\left(\mathcal W_1 (\chi^k_x, \delta_x) + \mathcal W_1 (\delta_x, \tilde \chi^k_x )\right) \, \eta^k(dx)\\
	&=\int_\R \int_\R \vert x' - x \vert \, ( \chi^k_x + \tilde \chi^k_x ) ( dx' ) \, \eta^k ( dx ) = \int_\R \vert x' - x \vert \, ( \chi^k + \tilde \chi^k ) ( dx, dx' ).
	\end{align*}
	Since $(x,x') \mapsto \vert x'-x \vert \in \Phi_1(\R^2)$ and $\chi^k$ and $\tilde\chi^k$ converge to $\chi$ in $\mathcal W_1$, we deduce, using \eqref{eq:convergencePr}, that
	\[ \int_\R\vert x'-x\vert\,(\chi^k+\tilde\chi^k)(dx,dx')\to 2\int_\R\vert x'-x\vert\,\chi(dx,dx')=0,\quad k\to+\infty, \]
	hence,
	\[ \mathcal{AW}_1 (\chi^k, \chi) \le \mathcal{AW}_1 (\chi^k, \tilde \chi^k) + \mathcal{AW}_1 ( \tilde\chi^k, \chi) \to 0,\quad k\to+\infty. \]
	By assumption, we can find for any $n\in N$ a sequence $(\pi^{k,n})_{k\in\N}$ of martingale couplings between $\mu^k_n$ and $\nu^k_n$, $k\in\N$, which converges in $\mathcal{AW}_1$ to $\pi_n$ as $k$ goes to $+\infty$, where $\pi_n$ denotes $\pi$ restricted to the $n$-th irreducible component given by \eqref{eq:decomposition martingale couplings}.
	By Lemma \ref{lem:continuous addition}, we have for all $p\in N$ that
	\[
	\chi^k + \sum_{n\in N,\ n\le p}\pi^{k,n} \to \chi + \sum_{n\in N,\ n\le p} \pi_n\quad \text{in }\mathcal{AW}_1,\quad k\to+\infty.
	\]
	Moreover, the respective marginals of $\chi^k+\sum_{n\in N}\pi^{k,n}$, namely $\mu^k$ and $\nu^k$, converge in $\mathcal W_1$ to the respective marginals of $\chi+\sum_{n\in N}\pi_n$, namely $\mu$ and $\nu$. Therefore, according to Lemma \ref{lem:inequality addition AWr} \ref{it:addition AWr limsup}, there exists a constant $C>0$ such that
	\[
	\limsup_k\mathcal{AW}_1\left(\chi^k + \sum_{n\in N} \pi^{k,n}, \chi + \sum_{n\in N} \pi_n \right)\le C\left(I_{\varepsilon_p}^1(\mu)+I_{\varepsilon_p}^1(\nu)\right),
	\]
	where $\varepsilon_p=\sum_{n\in N,n>p}\mu_n(\R)$ where by convention the sum over an empty set is $0$. Clearly, $(\epsilon_p)_{p\in N}$ tends to $0$, thus Lemma \ref{lem:uniform integrability} \ref{it:IepsilonrVanishes} reveals that the right-hand side vanishes as $p$ goes to $\sup N$. This proves that $\pi^k = \chi^k + \sum_{n \in  N} \pi^{k,n} \in \Pi_M(\mu^k,\nu^k)$ converges in $\mathcal{AW}_1$ to $\pi = \chi + \sum_{n \in  N} \pi^k \in \Pi_M(\mu,\nu)$.
\end{proof}

	 \begin{proof}[Proof of Theorem \ref{thm:1}] \emph{Step 1.} Due to Lemma \ref{lem:r=1 is enough} and Lemma \ref{lem:irreducible is enough}, we may suppose w.l.o.g.\ that $r=1$ and $(\mu,\nu)$ is irreducible with component $I = (\ell,\rho)$, $\ell \in \R \cup \{-\infty\}$, $\rho \in \R \cup \{ +\infty \}$. Next, we define auxiliary martingale couplings close to $\pi$ which will be easier to approximate in the limit.
	 	We define them with the same first marginal distribution whereby the second marginal distribution is smaller with respect to the convex order.
	 	These auxiliary couplings will satisfy two key properties:
	 	first, their second marginal distribution must be concentrated on a compact subset of $I$ when the first marginal distribution is itself concentrated on a certain compact subset $K$ of $I$.
	 	Second, it is essential that their second marginal distribution has positive mass on some two compact subsets of $I$ on both sides of $K$.

	 	Fix $\epsilon \in (0,\frac{1}{2})$.
	 	Choose a compact subset $K=[a,b]$ of $I$ with
	 	\begin{equation} \label{eq:estimate 1}
	 	\mu(K^\complement)<\varepsilon.
	 	\end{equation}
	 	Instead of directly approximating $\pi$, we initially consider the martingale coupling $\pi^{R,\alpha}$ whose definition is given below.
	 	For any $R>0$, let $(\pi^R_x)_{x\in \R}$ be the probability kernel obtained by virtue of Lemma \ref{lem:compact approx}. By Lemma \ref{lem:compact approx} \ref{it:compact approx} we have for all $ x \in \R$ that $\pi^R_x\le_c\pi_x$.
	 	Therefore,
	 	\[
	 	\mathcal W_1 ( \pi^R_x, \pi_x ) \le 2 \int_\R \vert  y \vert \, \pi_x (dy),
	 	\]
	 	where the right-hand side is a $\mu$-integrable function of $x$. By Lemma \ref{lem:compact approx} \ref{it:compact approx2} we find $\pi^R_x \to \pi_x$ in $\mathcal W_1$ as $R \to +\infty$.
	 	Let $\pi^R:=\mu\times\pi^R_x$, then dominated convergence yields
	 	\[
	 	\mathcal{AW}_1(\pi^R,\pi)\le\int_\R\mathcal{W}_1 (\pi_x^R,  \pi_x) \, \mu(dx) \to 0,\quad R\to +\infty.
	 	\]
	 	Denote by $\nu^R$ the second marginal of $\pi^R$.
	 	Consequently, $\nu^R$ converges to $\nu$ for the $\mathcal W_1$-distance and $\nu^R \leq_c \nu$ for all $R > 0$.
	 	Let $\tilde a$ and $\tilde b$ be real numbers such that $\tilde a\in(\ell,a)$ and $\tilde b\in(b,\rho)$, for instance
	 	\[
	 	\tilde a = \frac{\ell+a}{2} \vee (a-1) \text{ and } \tilde b = (b + 1) \wedge \frac{b+\rho}{2}.
	 	\]
Since $(\mu,\nu)$ is irreducible on $I$, \textcolor{blue}{according to Remark \ref{remassbord},} $\nu$ assigns positive mass to any neighbourhood \textcolor{red}{in $\overline  I$} of the endpoints $\ell$ and $\rho$ of $I$. From now on, we use the notational convention that for all $c\in\R\cup\{\pm\infty\}$,
	 	\[
	 	[-\infty,c)=\{x\in\R\mid x<c\},\quad(c,+\infty]=\{x\in\R\mid c<x\}\quad\text{and}\quad[-\infty,+\infty]=\R.
	 	\]
	 	In particular, $\overline I=[\ell,\rho]\subset\R$.
	 	Then $[\ell,\tilde a)$ and $(\tilde b, \rho]$ are relatively open on $\overline I$ with \textcolor{red}{$\nu^R(\overline I)=1=\nu(\overline I)$}, so Portmanteau's theorem yields
	 	\[
	 	\liminf_{R \to +\infty} \nu^R( [\ell,\tilde a) ) \ge \nu ( [\ell,\tilde a) ) > 0,\quad
	 	\liminf_{R \to + \infty} \nu^R( (\tilde b,\rho] ) \ge \nu( (\tilde b,\rho] ) > 0.
	 	\]
	 	Thus, we deduce that we can choose $R>0$ such that
	 	\begin{equation}\label{eq:choice of R}
	 	R\ge |a|\vee |b|,\ \quad\int_\R\mathcal{W}_1 (\pi_x^R,  \pi_x) \, \mu(dx)<\varepsilon,\quad \nu^R([\ell,\tilde a))>0,\quad\textrm{and}\quad\nu^R((\tilde b,\rho])>0.
	 	\end{equation}
	 	Let $\pi^{R,\alpha}_x$ be the image of $\pi^R_x$ by $y \mapsto \alpha (y - x) + x$ when $\alpha\in(0,1)$.
	 	Then $\pi^{R,\alpha}:=\mu\times\pi^{R,\alpha}_x$ satisfies
	 	by Lemma \ref{lem:scaling}
	 	\begin{align*}
	 	\mathcal{AW}_1(\varepsilon\pi+(1-\varepsilon)\pi^{R,\alpha},\pi)&\le\int_\R\mathcal W_1(\varepsilon\pi_x+(1-\varepsilon)\pi^{R,\alpha}_x,\pi_x)\,\mu(dx)\\
	 	&\le(1-\varepsilon)\int_\R\mathcal W_1(\pi^{R,\alpha}_x,\pi_x)\,\mu(dx)\\
	 	&\le\int_\R\mathcal W_1(\pi^{R,\alpha}_x,\pi^R_x)\,\mu(dx)+\int_\R\mathcal W_1(\pi^R_x,\pi_x)\,\mu(dx)\\
	 	&=(1-\alpha)\int_\R\int_\R\vert x-y\vert\,\pi^{R}_x(dy)\,\mu(dx)+\int_\R\mathcal W_1(\pi^R_x,\pi_x)\,\mu(dx)\\
	 	&\le(1-\alpha)\left(\int_\R\vert x\vert\,\mu(dx)+\int_\R\vert y\vert\,\nu^R(dy)\right)+\int_\R\mathcal W_1(\pi^R_x,\pi_x)\,\mu(dx),
	 	\end{align*}
	 	where the right-hand side converges to $\int_\R\mathcal W_1(\pi^R_x,\pi_x)\,\mu(dx)<\varepsilon$ for $\alpha \to 1$.
	 	Note that $\frac{2R-a-\tilde a}{2R-2\tilde a},\frac{b+\tilde b+2R}{2\tilde b+2R}\in(0,1)$, so we can choose $\alpha \in (0,1)$ such that
	 	\begin{equation}\label{eq:choice of alpha}
	 	\mathcal{AW}_1(\varepsilon\pi+(1-\varepsilon)\pi^{R,\alpha},\pi)< \epsilon\quad\text{and}\quad\alpha \ge \frac{2R-a-\tilde a}{2R-2\tilde a} \vee \frac{b+\tilde b+2R}{ 2\tilde b+2R} .
	 	\end{equation}
	 	 Let $L$ be a compact subset of $I$ such that the interior $\mathring L$ of $L$ satisfies
	 	 \[
	 	 [(-R) \vee (\alpha \ell + (1-\alpha) a), R \wedge (\alpha \rho + (1-\alpha) b )] \subset \mathring L.
	 	 \] 
	 	 Because $R \ge (-a)\vee b$ and thereby $[a,b] = K \subset [-R,R]$, we have that $\mu\vert_K\times\pi^R_x$ is concentrated on $K\times([-R,R] \cap \overline{I})$.
	 	 Furthermore, for any $(x,y)\in K \times ([-R,R] \cap \overline{I})$, we find $\alpha y + (1-\alpha) x \in \mathring L$.
	 	 Hence, $\mu \vert_K \times \pi^{R,\alpha}_x$ is concentrated on $K \times \mathring L$.
    
	 	 Denote the second marginal of $\pi^{R,\alpha}$ by $\nu^{R,\alpha}$. Since
	 	 \[
	 	 (x,y) \in (\ell, R) \times [\ell, \tilde a) \implies
	 	 \ell <  (1 - \alpha) x + \alpha y < R - \alpha(R - \tilde a) \leq \frac{a + \tilde a}{2},
	 	 \]
	 	 we have that
	 	 \begin{align*}
	 	 \nu^{R,\alpha}\left(\left(\ell,\frac{a+\tilde a}{2}\right)\right) &=\int_{\R^2}\mathbbm 1_{\left(\ell,\frac{a+\tilde a}{2}\right)}(y)\,\pi^{R,\alpha}(dx,dy) = \int_{\R^2}\mathbbm1_{\left(\ell,\frac{a+\tilde a}{2}\right)}(\alpha y+(1-\alpha)x)\,\pi^R(dx,dy)\\
	 	 &\ge\int_{\R^2}\mathbbm1_{(\ell,R)\times[\ell,\tilde a)}(x,y)\,\pi^R(dx,dy)=\int_{(\ell,R)}\pi^R_x((-\infty,\tilde a))\,\mu(dx).
	 	 \end{align*}   
	 	 If $x \in [R,+\infty)$, then $\pi^R_x = \delta_x$ and since $R\ge\tilde a$, $\pi^R_x((-\infty,\tilde a))=0$. Added to the fact that $\mu$ is concentrated on $I$, we obtain
	 	 \[
	 	 \int_{(\ell,R)}\pi^R_x((-\infty,\tilde a))\,\mu(dx)=\int_\R\pi^R_x((-\infty,\tilde a))\,\mu(dx)=\nu^R((-\infty,\tilde a))=\nu^R([\ell,\tilde a))>0.
	 	 \]
	 	 We deduce that
	 	 \begin{equation}\label{eq:positive mass nuRalpha}
	 	 \nu^{R,\alpha}\left(\left(\ell,\frac{a+\tilde a}{2}\right)\right)>0\text{, and similarly, }\nu^{R,\alpha}\left(\left(\frac{b+\tilde b}{2},\rho\right)\right)>0.
	 	 \end{equation}
    
	 	 To summarise, we have constructed a martingale coupling $\pi^{R,\alpha}\in\Pi_M(\mu,\nu^{R,\alpha})$ close to $\pi$ with respect to the $\mathcal{AW}_1$-distance \textcolor{red}{in view of \eqref{eq:choice of alpha}}, whose restriction $\pi^{R,\alpha}\vert_{K\times\R}$ is compactly supported on $K\times L$ and concentrated on $K\times\mathring L$. Moreover, the second marginal distribution $\nu^{R,\alpha}$ is dominated by $\nu$ in the convex order and assigns positive mass on both sides of $K$ \textcolor{red}{according to \eqref{eq:positive mass nuRalpha}}.

	 	\emph{Step 2.} In the next step we construct a sequence of sub-probability martingale couplings supported on a compact cube $J\times J$ ($K\subset J\subset I$) converging to $\pi^{R,\alpha}|_{K\times \R}$.	 	
	 	
	  Our first goal is to find a sequence $\nu^{R,\alpha,k}$, $k\in\N$, such that $\mu^k \leq_c \nu^{R,\alpha,k} \leq_c \nu^k$ and
	  	\begin{equation}\label{eq:convergence nuRalphak}
	  	\mathcal W_1(\nu^{R,\alpha,k},\nu^{R,\alpha}) \to 0,\quad k \to \infty.
	  	\end{equation}
	  	Defining $\nu^{R,\alpha,k}$ by
	  	\[
	  	\nu^{R,\alpha,k} = \nu^k \wedge_{c} (\mu^k \vee_{c} T_k(\nu^{R,\alpha})),
	  	\]
	  	where $T_k$ denotes the translation by the difference between the common barycentre of $\mu^k$ and $\nu^k$ and the common barycentre of $\nu$ and $\nu^{R,\alpha}$, i.e., $\int_\R y\, \nu^k(dy) - \int_\R y\, \nu^{R,\alpha}(dy)$, fulfils these requirements.
	  	Indeed
	  	\[
	  	\mathcal W_1 ( T_k ( \nu^{R,\alpha} ), \nu^{R,\alpha}) = \left|\int_\R y \,\nu^k(dy) - \int_\R y\, \nu(dy)\right| \le \mathcal W_1(\nu^k,\nu) \to 0,
	  	\]
	  	as $k$ goes to $+\infty$.
	  	Then Lemma \ref{lemcontinfsup} provides $\nu^{R,\alpha,k} \to \nu \wedge_c ( \mu \vee_{c} \nu^{R,\alpha}) = \nu^{R,\alpha}$ in $\mathcal W_1$ as $k$ goes to $+\infty$.
	 	By Proposition \ref{prop:adapted approximation dim1} we can approximate $\pi^{R,\alpha}$ with couplings $\pi^{R,\alpha,k} \in \Pi(\mu^k,\nu^{R,\alpha,k})$ in $\mathcal{AW}_1$.
	 	Unfortunately the sequence $\pi^{R,\alpha,k}$, $k\in\N$ does not have to consist of solely martingale couplings.
	 	Therefore, we have to adjust the barycentres of its disintegrations, $(\pi^{R,\alpha,k}_x)$ to obtain martingale kernels and thereby martingale couplings.
	 	Due to \eqref{eq:positive mass nuRalpha} and inner regularity of $\nu^{R,\alpha}$, we find compact sets
	 	\[
	 	L_- \subset \left( \ell , \frac{a+\tilde a}{2}\right),\quad L_+\subset \left(\frac{b+\tilde b}{2},\rho \right)
	 	\]
	 	with $\nu^{R,\alpha}$-positive measure. Let $\tilde \ell,\tilde \rho\in I$, be such that $\tilde \ell<\inf(L\cup L-)$ and $\sup(L\cup L_+)<\tilde \rho$. Then define
	 	\begin{equation}\label{eq:def Ltilde Ktilde}
	 	\tilde L_-=\left(\tilde \ell,\frac{a+\tilde a}{2}\right),\quad\tilde L_+=\left(\frac{b+\tilde b}{2},\tilde \rho\right)\quad\text{and}\quad\tilde K=\left(\frac{3a+\tilde a}{4},\frac{3b+\tilde b}{4}\right),
	 	\end{equation}
	 	so that $\tilde L_-$, $\tilde L_+$ and $\tilde K$ are bounded and open intervals covering respectively $L_-$, $L_+$ and $K$ and such that the distance $e$ between $\tilde L_-\cup\tilde L_+$ and $\tilde K$ is positive:
	 	\[
	 	e=\inf\left\{\vert x-y\vert\mid(x,y)\in(\tilde L_-\cup\tilde L_+)\times\tilde K\right\}\ge\frac{a-\tilde a}{4}\wedge\frac{\tilde b-b}{4}>0.
	 	\]
	 	Denoting $J=[\tilde \ell,\tilde \rho]$, Figure \ref{fig:bild2} summarises the construction.
	 	\begin{figure}[h]
	 		\centering
	 		\begin{tikzpicture}[decoration = {random steps,segment length=2pt,amplitude=0.1pt}]
	 		\draw[decorate, line width=\x mm, color = {rgb,255:red,92; green,74; blue,114}, (-)] (0,0) -- (10,0) ;

	 		\draw[decorate, line width=\x mm, color = {rgb,255:red,163; green,88; blue,109}, |-| ] (0.5,0.2) -- (9.5,0.2) ;
	 		
	 		\draw[decorate, line width=\x mm, color = {rgb,255:red,244; green,135; blue,75}, |-|] (1,-0.2) -- (2,-0.2) ;
	 		\draw[decorate, line width=\x mm, color = {rgb,255:red,244; green,135; blue,75}, |-| ] (8,-0.2) -- (9,-0.2) ;
	 		\draw[decorate, line width=\x mm, color = {rgb,255:red,194; green,85; blue,25}, |-| ] (4,-0.2) -- (7,-0.2) ;
	 		
	 		\draw (0,0.1) node[below left, scale = 1.2]{$\ell$} ;
	 		\draw (10,0) node[below right, scale = 1.2]{$\rho$} ;
	 		\draw (1.5,-0.2) node[below, scale = 1.2]{$L_-$} ;
	 		\draw (8.5,-0.2) node[below, scale = 1.2]{$L_+$} ;
	 		\draw (5.5,-0.2) node[below, scale = 1.2]{$K$} ;
	 		\draw (5,0.2) node[above, scale = 1.2]{$J$} ;
	 		\draw (0.625,0.2) node[above left, scale = 1.2]{$\tilde \ell$} ;
	 		\draw (9.375,0.2) node[above right, scale = 1.2]{$\tilde \rho$} ;
	 		\draw (4.125,-0.2) node[below left, scale = 1.2]{$a$} ;
	 		\draw (6.875,-0.1) node[below right, scale = 1.2]{$b$} ;
	 		\draw (1.5, 0.2) node[above, scale = 1.2]{$\tilde a$} ;
	 		\draw[decorate, line width=\x mm, |-| ] (1.5,0) -- (1.5001,0) ;
	 		\draw (8.5, 0.2) node[above, scale = 1.2]{$\tilde b$} ;
	 		\draw[decorate, line width=\x mm, |-| ] (8.5,0) -- (8.5001,0) ;
	 		
	 		\draw[decorate, line width=\x mm, color = {rgb,255:red,243; green,176; blue,90}, (-)] (0.5,-0.8) -- (2.75,-0.8) ;
	 		\draw (1.625,-0.8) node[below, scale = 1.2]{$\tilde L_-$} ;
	 		\draw[decorate, line width=\x mm, color = {rgb,255:red,243; green,176; blue,90}, (-)] (7.75,-0.8) -- (9.5,-0.8) ;
	 		\draw (8.625,-0.8) node[below, scale = 1.2]{$\tilde L_+$} ;
	 		\draw[decorate, line width=\x mm, color = {rgb,255:red,193; green,126; blue,40}, (-) ] (3.375,-0.8) -- (7.375,-0.8) ;
	 		\draw (5.375,-0.8) node[below, scale = 1.2]{$\tilde K$} ;
	 		
	 		\draw[decoration={brace,mirror,raise=5pt, amplitude=3pt}, decorate]
	 		(7.375,-0.8) -- node[below=6pt,scale=1.2] {$e$} (7.75,-0.8);
	 		\end{tikzpicture}
	 		\caption{Points and intervals involved in the proof. The boundaries of the closed intervals are vertical bars and those of the open intervals are parentheses.\label{fig:bild2}}
	 	\end{figure}
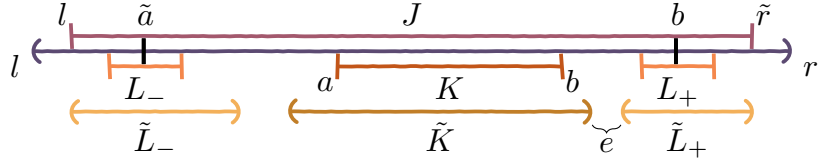

	 	The respective restrictions of $\nu^{R,\alpha,k}$ to $\tilde L_-$ and $\tilde L_+$ are denoted by $\nu_-^{k}$ and $\nu_+^{k}$, respectively.
	 	Since $\tilde L_-$ and $\tilde L_+$ are open, Portmanteau's theorem ensures that eventually (for $k$ sufficiently large) $\nu_-^{k}$ and $\nu_+^{k}$ each have more total mass than some constant $\delta > 0$.

	 	By Lemma \ref{lem:subprob convergence} \ref{it:subprog convergence2} there are
	 	$\hat \mu^k \leq \mu^k$, $\hat \nu^k \leq \nu^{R,\alpha,k}$, $\hat \pi^{k}=\hat\mu^k\times\hat\pi^k_x \in \Pi(\hat \mu^k, \hat \nu^k)$ concentrated on $\tilde K \times \mathring L$, and $\epsilon_k \ge 0$ such that
	 	\begin{equation}\label{eq:approx hatpi}
	 	\mathcal{AW}_1( \hat\pi^{k}, (1-\epsilon_k) \pi^{R,\alpha} \vert_{K\times\R} )+\epsilon_k \to 0,\quad k\to+\infty.
	 	\end{equation}
	 	The following procedure shows that there are for $\hat\mu^k(dx)$-almost every $x$ unique constants $c^k_-(x), c^k_+(x) \in [0,+\infty)$ and $d^k(x) \in [1,+\infty)$ such that
	 	\[
	 	\tilde \pi_x^{k} := \frac{\hat \pi^k_x + c_+^k(x) \nu_+^{k} + c_-^k(x) \nu_-^{k}}{d^k(x)} \in \mathcal P(\R),\quad \int_\R y\, \tilde \pi^k_x(dy) = x,\quad c_-^k(x) \wedge c_+^k(x) = 0.
	 	\]
	 	Note that the constraint $c_-^k(x) \wedge c_+^k(x) = 0$ provides
	 	\begin{align} \label{eq:step2.1}
	 	\int_\R y\, \hat \pi^k_x(dy)\leq x \implies c_-^k(x) = 0, \quad \int_\R y\,\hat \pi^k_x(dy) \geq x \implies c_+^k(x) = 0.
	 	\end{align}
	 	We require $\tilde \pi^k_x$ to be a probability measure with mean $x$, thus,
	 	\begin{gather} \label{eq:step2.2}
	 	1 + c_+^k(x) \nu^k_+(\R) + c_-^k(x) \nu^k_-(\R) = d^k(x), \\ \label{eq:step2.3}
	 	\int_\R y \, \hat \pi^k_x(dy) + c_+^k(x) \int_\R y \, \nu^k_+(dy) + c_-^k(x) \int_\R y \, \nu_-^k(dy) = x d^k(x).
	 	\end{gather}
	 	Combining \eqref{eq:step2.1} with \eqref{eq:step2.2} and \eqref{eq:step2.3} yields
	 	\begin{gather*}
	 	c_-^k(x) = \frac{\int_\R y \, \hat \pi^k_x(dy)-x}{ \int_\R (x - y) \, \nu^k_-(dy)} \vee 0 \in \left[ 0, \frac{\left\vert x-\int_\R y \, \hat \pi^k_x(dy)\right\vert}{ e \nu_-^k(\R)} \right], \\
	 	c_+^k(x) = \frac{x-\int_\R y \, \hat \pi^k_x(dy)}{\int_\R (y - x) \, \nu_+^k(dy)} \vee 0 \in \left[ 0,  \frac{\left\vert x-\int_\R y \, \hat \pi^k_x(dy)\right\vert}{e \nu_+^k(\R)}\right], \\
	 	d^k(x) = 1 + c_-^k(x) \nu_-^k(\R) + c_+^k(x) \nu_+^k(\R) \in \left[ 1, 1 + \frac{\left\vert x-\int_\R y \, \hat\pi^k_x(dy)\right\vert}{e}\right].
	 \end{gather*}
	 	Remember from \eqref{eq:def Ltilde Ktilde} that $L\cup\tilde L_-\cup\tilde L_+\subset[\tilde \ell,\tilde \rho]\subset I$. Then we obtain for $\hat \mu^k(dx)$-almost every $x$ the estimate
	 	\begin{align*}
	 	\mathcal W_1(\tilde \pi^k_x, \hat\pi^k_x) &\leq \mathcal W_1\left( \frac{c^k_+(x)\nu^k_++c_-^k(x) \nu^k_-}{d^k(x)}, \frac{d^k(x) - 1 }{ d^k(x) } \hat \pi^k_x \right)\leq \frac{d^k(x) - 1}{d^k(x)} |\tilde \rho - \tilde \ell|\\
	 	&\leq \frac{\left\vert x-\int_\R y \, \hat\pi^k_x(dy)\right\vert}{e}\vert\tilde \rho - \tilde \ell\vert.
	 	\end{align*}
	 	Hence, the adapted Wasserstein distance between $\hat \pi^k$ and $\tilde \pi^k = \hat \mu^k \times \tilde \pi^k_x$ satisfies
	 	\begin{align*}
	 	\mathcal{AW}_1(\tilde \pi^k,\hat \pi^k) &\leq \int_\R \mathcal W_1(\tilde \pi^k_x,\hat \pi^k_x) \, \hat \mu^k(dx)\leq \frac{|\tilde \rho - \tilde \ell|}{ e }\int_\R\left\vert x-\int_\R y \, \hat\pi^k_x(dy)\right\vert\,\hat\mu^k(dx)\\
	 	&\leq \frac{|\tilde \rho - \tilde \ell|}{ e } \mathcal{AW}_1( \hat \pi^k, (1 - \epsilon_k) \pi^{R,\alpha}\vert_{K \times \R}),
	 	\end{align*}
	 	where we used Remark \ref{rk: almost martingale} with exponent $r=1=2^{r-1}$ in the last inequality. The triangle inequality and \eqref{eq:approx hatpi} then yield
	 	\begin{equation}\label{eq:convergence piTilde}
	 	\lim_k \mathcal{AW}_1(\tilde \pi^k,(1 - \epsilon_k) \pi^{R,\alpha}\vert_{K \times \R}) \to 0,\quad k\to \infty.
	 	\end{equation}
	 	Next we bound the total mass which we require to fix the barycentres.
	 	We find that
	 	\begin{align*}
	 	\int_\R \frac{c_-^k(x) + c_+^k(x)}{d^k(x)} \, \hat \mu^k(dx) &\leq \frac{1}{e (\nu_-^k(\R)\wedge \nu_+^k(\R))}\int_\R\left\vert x-\int_\R y \, \hat\pi^k_x(dy)\right\vert\,\hat\mu^k(dx)\\
	 	&\leq \frac{\mathcal{AW}_1(\hat \pi^k, (1 - \epsilon_k) \pi^{R,\alpha}\vert_{K \times \R})}{e (\nu_-^k(\R)\wedge \nu_+^k(\R))} \to 0,\quad k\to+\infty,
	 	\end{align*}
	 	where we used Remark \ref{rk: almost martingale} again for the last inequality and the fact that $\nu_-^k(\R)\wedge \nu_+^k(\R)\ge\delta$ for $k$ large enough for the limit.
	 	Consequently, when $\tilde \nu^k$ denotes the second marginal of $\tilde \pi^k$, we have for $k$ sufficiently large that
	 	\[ (1 - 2 \epsilon) \tilde \nu^k \leq (1 - 2 \epsilon)\hat \nu^k + (1 - 2 \epsilon)(\nu^k_- + \nu^k_+) \int_\R \frac{c_-^k(x) + c_+^k(x)}{d^k(x)} \, \hat \mu^k(dx) \leq (1 - \epsilon) \nu^{R,\alpha,k}. \]

	 	\emph{Step 3.}
	 	In this step, we complement the martingale coupling $(1 - 2 \epsilon) \tilde \pi^k$ to a martingale coupling with marginals $\mu^k$ and $\epsilon \nu^k + (1-\epsilon)\nu^{R,\alpha,k}$ for $k$ sufficiently large.
	 	Recall that $\tilde \pi^k \in \Pi_M(\hat \mu^k, \tilde \nu^k)$ and $\pi^{R,\alpha}\vert_{K \times \R} \in \Pi_M(\mu\vert_K,\check \nu^{R,\alpha})$, where $\check\nu^{R,\alpha}$ is the second marginal distribution of $\pi^{R,\alpha}\vert_{K\times\R}$, are concentrated on the compact cube $J \times J$ and
	 	\[ \mathcal{AW}_1(\tilde \pi^k, (1 - \epsilon_k) \pi^{R,\alpha}\vert_{K \times \R}) \to 0,\quad k\to+\infty. \]
	 	Furthermore, since $(1-\epsilon) \pi^{R,\alpha} - ( 1 - 2\epsilon) \pi^{R,\alpha} \vert_{K\times\R}$ is a martingale coupling with marginals
	 	\[ (1-\epsilon)\mu - (1 - 2 \epsilon)\mu\vert_K \quad \text{and} \quad(1 - \epsilon) \nu^{R,\alpha} - (1 - 2 \epsilon)\check \nu^{R,\alpha},\]
	 	we deduce by irreducibility of the pair $(\mu,\nu)$ on $I$ irreducibility of the pair of sub-probability measures
	 	\[
	 	\varepsilon \mu + (1-\varepsilon)\mu - (1-2\varepsilon) \mu\vert_K \quad \text{and} \quad  \varepsilon \nu + (1-\varepsilon) \nu^{R,\alpha} - (1-2\varepsilon) \check \nu^{R,\alpha},
	 	\]
	 	whose potential functions satisfy
	 	\begin{align*}
	 	0 \leq u_\mu - u_{(1 - 2\varepsilon) \mu\vert_K} < u_{\varepsilon \nu + (1-\varepsilon) \nu^{R,\alpha}} - u_{(1-2\varepsilon) \check \nu^{R,\alpha}} \quad \text{on }I.
	 	\end{align*}
	 	Since those potential functions are continuous, there exists $\tau> 0$ such that they have distance greater $\tau$ on $J$. By uniform convergence of potential functions, for $k\in\N$ sufficiently large we have
	 	\begin{align*}
	 	0 \leq u_{\mu^k} - u_{(1-2\varepsilon) \hat \mu^k}  + \frac{\tau}{2} \leq u_{\varepsilon \nu^k + (1-\varepsilon) \nu^{R,\alpha,k}} - u_{(1-2\epsilon) \tilde \nu^{k}}\quad\text{on }J.
	 	\end{align*}
	 	On the complement of $J$ we have $u_{(1-2\varepsilon) \hat\mu^k} = u_{(1-2\varepsilon) \tilde \nu^{k}}$ since both measures are concentrated on $J$ and satisfy $(1-2\varepsilon)\int_\R x\, \hat\mu^k(dx)=(1-2\varepsilon)\int_\R y\,\tilde \nu^{k}(dy)$.
	 	Therefore,
	 	\begin{align*}
	 	0 \leq u_{\mu^k} - u_{(1-2\varepsilon)\hat \mu^k} \leq u_{\varepsilon \nu^k + (1-\varepsilon) \nu^{R,\alpha,k}} - u_{(1-2\varepsilon) \tilde \nu^{k}}\quad\text{on }J^c.
	 	\end{align*}
	 	By Strassen's theorem \cite{St65}, there exists $\eta^k\in \Pi_M(\mu^k - (1-2\varepsilon) \hat \mu^k,\varepsilon \nu^k + (1 - \varepsilon) \nu^{R,\alpha,k} - (1-2\varepsilon) \tilde \nu^{k})$.
	 	Finally, we write
	 	\[
	 	\overline\pi^k = \eta^k + (1-2\varepsilon) \tilde \pi^{k} \in \Pi_M(\mu^k,\varepsilon \nu^k + (1-\varepsilon) \nu^{R,\alpha,k}).
	 	\]

	 	\emph{Step 4.} In the last step, we show that the sequence constructed in this way is eventually close to the original martingale coupling $\pi$ in adapted Wasserstein distance.

	 	The marginals of $\overline\pi^k$ are converging in $\mathcal W_1$ to $(\mu,\epsilon \nu + (1 - \epsilon) \nu^{R,\alpha})$ as $k$ goes to $+\infty$.
	 	\textcolor{blue}{We have according to \eqref{eq:convergence piTilde} that
	 	\[ \mathcal{AW}_1\left((1-2\varepsilon)\frac{1-\varepsilon}{1-\varepsilon_k}\tilde\pi^k,(1-2\varepsilon)(1-\varepsilon)\pi^{R,\alpha}\vert_{K \times \R}\right) \to 0,\quad k \to \infty.\]
                For $k$ large enough so that $\varepsilon_k\le\varepsilon$,
 \begin{align*}
   \bar\pi^k(\R^2)&-(1-2\varepsilon)\frac{1-\varepsilon}{1-\varepsilon_k}\tilde\pi^k(\R^2)= \eta^k(\R^2)+(1-2\varepsilon)\frac{\varepsilon-\varepsilon_k}{1-\varepsilon_k}\tilde\pi^k(\R^2)\\&=\left(\varepsilon\pi+(1-\varepsilon)\pi^{R,\alpha}-(1-2\varepsilon)(1-\varepsilon_k)\pi^{R,\alpha}\vert_{K\times\R}\right)(\R^2)+(1-2\varepsilon)(\varepsilon-\varepsilon_k)\pi^{R,\alpha}\vert_{K\times\R}(\R^2)\\&=1-(1-2\varepsilon)(1-\varepsilon)\mu(K)\le 4\varepsilon,
 \end{align*}
 where we used $\mu(K) \ge 1-\varepsilon$ for the last inequality.
Hence applying Lemma \ref{lem:inequality addition AWr}
\ref{it:addition AWr limsup}, with $(\hat\pi^k,\hat\pi,\tilde\pi^k,\tilde\pi,\varepsilon)$ replaced by
\begin{multline*}
	\left((1-2\varepsilon)\frac{1-\varepsilon}{1-\varepsilon_k}\tilde\pi^k,(1-2\varepsilon)(1-\varepsilon)\pi^{R,\alpha}\vert_{K \times \R},\right. \\ 
	\left. \eta^k+(1-2\varepsilon)\frac{\varepsilon-\varepsilon_k}{1-\varepsilon_k}\tilde\pi^k,\varepsilon\pi+(1-\varepsilon)\left(\pi^{R,\alpha}-(1-2\varepsilon)\pi^{R,\alpha}\vert_{K \times \R}\right), 4\varepsilon\right),
\end{multline*}
we obtain\[
	 	\limsup_k\mathcal{AW}_1(\overline\pi^k,\varepsilon\pi+(1-\varepsilon)\pi^{R,\alpha})\le C(I_{4\varepsilon}(\mu)+I_{4\varepsilon}(\varepsilon\nu+(1-\varepsilon)\nu^{R,\alpha})),
	 	\]
	 	with $C$ given by Lemma \ref{lem:inequality addition AWr} \ref{it:addition AWr limsup} and depending only on the exponent $r = 1$.} Since $\nu^{R,\alpha}\le_c\nu$, then $\varepsilon\nu+(1-\varepsilon)\nu^{R,\alpha}\le_c\nu$, so using Lemma \ref{lem:uniform integrability} \ref{it:IepsilonrCvx}, the triangle inequality and \eqref{eq:choice of alpha}, we get
	 	\begin{align*}
	 	\limsup_k\mathcal{AW}_1(\overline\pi^k,\pi)&\le\limsup_k\left(\mathcal{AW}_1(\overline\pi^k,\varepsilon\pi+(1-\varepsilon)\pi^{R,\alpha})+\mathcal{AW}_1(\varepsilon\pi+(1-\varepsilon)\pi^{R,\alpha},\pi)\right)\\
	 	&\le C(I_{4\varepsilon}(\mu)+I_{4\varepsilon}(\nu))+\varepsilon.
	 	\end{align*}
	 	Since the right-hand side only depends on $\epsilon$ and vanishes as $\epsilon$ goes to $0$, we can reason like in the proof of Proposition \ref{prop:adapted approximation} (from \eqref{eq:limsup Awr pikepsilon}) to find a null sequence $(\tilde \epsilon_k)_{k \in \N}$, two sequences $(R_k)_{k\in\N}$, $(\alpha_k)_{k\in\N}$ with values respectively in $\R_+^*$ and $(0,1)$, and martingale couplings
	 	\[
	 	\mathring\pi^k\in \Pi_M(\mu^k, \tilde \epsilon_k \nu^k + (1-\tilde \epsilon_k) \nu^{R_k,\alpha_k,k}),\quad k\in\N
	 	\]
	 	such that
	 	\begin{equation}\label{eq:convergence pik}
	 	\mathcal{AW}_1(\mathring\pi^k,\pi)\to0,\quad k\to+\infty.
	 	\end{equation}
	 	In particular, the $\mathcal W_1$-distance of their second marginal distributions vanishes as $k$ goes to $+\infty$, hence the triangle inequality yields
	 	\[
	 	\mathcal W_1(\tilde \epsilon_k \nu^k + (1-\tilde \epsilon_k) \nu^{R_k,\alpha_k,k},\nu^k)\le\mathcal W_1(\tilde \epsilon_k \nu^k + (1-\tilde \epsilon_k) \nu^{R_k,\alpha_k,k},\nu)+\mathcal W_1(\nu,\nu^k)\to0,\quad k\to+\infty.
	 	\]
	 	Remember that $\nu^{R_k,\alpha_k,k}\le_c\nu^k$, hence $\tilde \epsilon_k \nu^k + (1-\tilde \epsilon_k) \nu^{R_k,\alpha_k,k}\le_c\nu^k$. Then by \cite[Theorem 2.12]{JoMa18}, there exist martingale couplings $M^k \in \Pi_M(\tilde \epsilon_k \nu^k + (1 - \tilde \epsilon_k)\nu^{R_k,\alpha_k,k},\nu^k)$, $k\in\N$ such that
	 	\begin{equation}\label{eq:convergence Mk}
	 	\int_{\R\times\R} |x-y|\,M^k(dx,dy) \leq 2\mathcal{W}_1(\tilde \epsilon_k \nu^k + (1 - \tilde \epsilon_k)\nu^{R_k,\alpha_k,k},\nu^k)\to0,\quad k\to+\infty.
	 	\end{equation}
	 	\textcolor{red}{Let then 
	 	\[
	 	\pi^k(dx,dy)=\mu^k(dx)\int_{z\in\R}M^k_z(dy)\,\mathring\pi^k_x(dz)\in\Pi_M(\mu^k,\nu^k).
	 	\]}
	 	Using the fact that for $\mu^k(dx)$-almost every $x$, $\mathring\pi^k_x(dz)\,M^k_z(dy)\in\Pi(\mathring\pi^k_x,\pi^k_x)$, we get
	 	\begin{align*}
	 	\mathcal{AW}_1(\pi^k,\mathring\pi^k)&\le\int_\R\mathcal W_1(\pi^k_x,\mathring\pi^k_x)\,\mu^k(dx)\le\int_{\R\times\R\times\R}\vert z-y\vert\,\mu^k(dx)\,M^k_z(dy)\,\,\mathring\pi^k_x(dz)\\
	 	&=\int_{\R\times\R}\vert z-y\vert\,M^k(dy,dz),
	 	\end{align*}
	 	where the right-hand side vanishes by \eqref{eq:convergence Mk} as $k$ goes to $+\infty$. Then \eqref{eq:convergence pik} and the triangle inequality yield
	 	\[
	 	\mathcal{AW}_1(\pi^k,\pi)\le\mathcal{AW}_1(\pi^k,\mathring\pi^k)+\mathcal{AW}_1(\mathring\pi^k,\pi)\to0,\quad k\to+\infty,
	 	\]
	 	which concludes the proof.
	 \end{proof}
 
 \bibliography{biblio1}
 \bibliographystyle{abbrv}
 \end{document}